\documentclass[12pt]{amsart} 
\usepackage{amsmath,amssymb,amsthm}
\usepackage{geometry}

\usepackage[utf8x]{inputenc}
\usepackage{enumitem}
\usepackage{tikz,tikz-cd}
\usepackage[colorlinks=true,citecolor=blue]{hyperref}
\usepackage{cleveref}
\usepackage{verbatim}
\usepackage{graphicx}
\usepackage{faktor}
\usepackage{xcolor}
\usepackage{xfrac}
\usepackage{todonotes}
\usepackage[normalem]{ulem}
\usepackage{setspace}
\usepackage{bbm}
\usepackage{calrsfs}
\usepackage{dutchcal}
\usepackage{seqsplit}
\usepackage{xstring}
\usepackage{mathtools}
\usepackage{thmtools}

\tikzset{
  arrow/.pic={\path[tips,every arrow/.try,->,>=#1] (0,0) -- +(0,4pt);},
  pics/arrow/.default={triangle 90}
}

\tikzset{->-/.style={decoration={
  markings,
  mark=at position .6 with {\arrow{latex}}},postaction={decorate}}
  }

\tikzset{
  c/.style={every coordinate/.try}
}

\DeclareFontFamily{U}{lasy}{}
\DeclareFontShape{U}{lasy}{m}{n}{
  <-5.5> lasy5
  <5.5-6.5> lasy6
  <6.5-7.5> lasy7
  <7.5-8.5> lasy8
  <8.5-9.5> lasy9
  <9.5-> lasy10
}{}

\newcommand{\nocontentsline}[3]{}
\let\origcontentsline\addcontentsline
\newcommand\stoptoc{\let\addcontentsline\nocontentsline}
\newcommand\resumetoc{\let\addcontentsline\origcontentsline}

\DeclareMathAlphabet{\pazocal}{OMS}{zplm}{m}{n}

\def\on#1{\operatorname{#1}}

\def\Spec{\on{Spec}}
\def\Hom{\on{Hom}}

\def\logGm{\mathbf G_{\log}}

\def\BGL{\mathrm{BGL}}

\def\tropGm{\mathbf G_{\rm trop}}

\def\Pic{\operatorname{Pic}}

\def\uHom{\underline{\Hom}}

\def\Sym{\operatorname{Sym}}

\def\Log{\mathbf{Log}}

\newcommand{\coker}{\operatorname{coker}}

\newcommand{\Gm}{\mathbf{G}_{\rm{m}}}

\newcommand{\Gtr}[1]{\mathbf{G}_{\rm{trop},#1}}

\newcommand{\OO}{\mathcal  O}
\newcommand{\Acal}{\mathcal A}

\newcommand{\Ecal}{\mathcal E}
\newcommand{\Ical}{\pazocal I}
\newcommand{\Fcal}{\mathcal F}

\newcommand{\Pcal}{\pazocal P}

\newcommand{\PP}{\mathbb P}

\newcommand{\Aaff}{\mathbf A}
\newcommand{\NN}{\mathbf N}
\newcommand{\ZZ}{\mathbf Z}

\newcommand{\mo}{P}
\newcommand{\mt}{Q}
\newcommand{\mm}{h}
\newcommand{\amb}[1]{{#1}^{\rm{a}}}

\newcommand{\bq}{\begin{equation}}
\newcommand{\eq}{\end{equation}}
\newcommand{\ba}{\begin{aligned}}
\newcommand{\ea}{\end{aligned}}
\newcommand{\be}{\begin{enumerate}}
\newcommand{\ee}{\end{enumerate}}
\newcommand{\bsm}{\left(\begin{smallmatrix}}
\newcommand{\esm}{\end{smallmatrix}\right)}                   
\newcommand{\bpm}{\begin{pmatrix}}
\newcommand{\epm}{\end{pmatrix}}
\newcommand{\barr}{\begin{displaymath}\begin{array}{cccc}}
\newcommand{\earr}{\end{array}\end{displaymath}}
\newcommand{\barrl}{\begin{displaymath}\begin{array}{lcl}}
\newcommand{\earrl}{\end{array}\end{displaymath}}
\newcommand{\barl}{\begin{displaymath}\begin{array}{l}}
\newcommand{\earl}{\end{array}\end{displaymath}}
\newcommand{\bxym}{ \begin{displaymath}\xymatrix }
\newcommand{\exym}{\end{displaymath}}
\newcommand{\bcd}{\begin{center}\begin{tikzcd}}
\newcommand{\ecd}{\end{tikzcd}\end{center}}

\theoremstyle{theorem}
\newtheorem{theorem}{Theorem}[section]
\newtheorem*{theorem*}{Theorem}

\newtheorem{corollary}[theorem]{Corollary}
\newtheorem{lemma}[theorem]{Lemma}
\newtheorem{proposition}[theorem]{Proposition}
\theoremstyle{definition}
\newtheorem{remark}[theorem]{Remark}
\newtheorem*{remark*}{Remark}

\newtheorem*{runningexample*}{Running example}

\newtheorem*{aside*}{Aside}

\newtheorem{definition}[theorem]{Definition}
\newtheorem{example}[theorem]{Example}

\newtheorem{proposition-definition}[theorem]{Proposition-Definition}

\Crefname{theorem}{Theorem}{Theorems}
\Crefname{claim}{Claim}{Claims}
\Crefname{conjecture}{Conjecture}{Conjectures}
\Crefname{corollary}{Corollary}{Corollaries}
\Crefname{lemma}{Lemma}{Lemmas}
\Crefname{proposition}{Proposition}{Propositions}
\Crefname{remark}{Remark}{Remarks}
\Crefname{assumption}{Assumption}{Assumptions}
\Crefname{aside}{Aside}{Asides}
\Crefname{condition}{Condition}{Conditions}
\Crefname{construction}{Construction}{Constructions}
\Crefname{convention}{Convention}{Conventions}
\Crefname{definition}{Definition}{Definitions}
\crefname{definition}{definition}{definitions}
\Crefname{example}{Example}{Examples}
\Crefname{exercise}{Exercise}{Exercises}
\Crefname{notation}{Notation}{Notations}
\Crefname{proposition-definition}{Proposition-Definition}{Proposition-Definitions}
\Crefname{question}{Question}{Questions}
\Crefname{setting}{Setting}{Settings}

\makeatletter
\newcommand\pcite[2][]{%
	\let\temp@cite\@cite%
	\def\@cite##1##2{\@citestyle \citeform {##1}\if@tempswa , ##2\fi}%
	\ifx#1\relax{\cite[#1]{#2}}\else\cite{#2}\fi%
	\let\@cite\temp@cite%
}
\def\thmhead#1#2#3{{#1}\thmnumber {\@ifnotempty {#1}{ }\@upn {#2}}\thmnote { {\the \thm@notefont #3}}}
\makeatother

\def\iff{\kern-.5ex\Leftrightarrow\kern-.5ex}

\def\Star{\operatorname{Star}}

\makeatletter
\def\tocdesc#1{\addtocontents{toc}{\parbox[t]{\dimexpr\linewidth-3em}{\leftskip2em\footnotesize #1} \par\medskip}}
\makeatother

\title{Vector bundles on Olsson fans}
\author{Luca Battistella}
\address{Università di Bologna, Piazza di Porta S. Donato 5, Bologna 40126, Italy}
\email{luca.battistella2@unibo.it}
\author{Francesca Carocci}
\address{Università di Roma Tor Vergata, Via della Ricerca Scientifica 1, Roma 00133, Italy}
\email{carocci@mat.uniroma2.it}
\author{Jonathan Wise}
\address{University of Colorado Boulder, Campus Box 395
Boulder, Colorado 80309-0395, US}
\email{jonathan.wise@colorado.edu}
\date{January 2026}

\begin{document}

\begin{abstract}
Artin fans are algebro-geometric incarnations of cone complexes. We study weakly convex Olsson fans, generalising Artin fans in two ways: first, they admit lineality spaces, thus including tropical tori as well; second, they are defined over a base logarithmic scheme, thus providing a relative version of equivariant toric geometries. We determine conditions under which Olsson fans are well-behaved, in the sense that their geometry is determined combinatorially. We undertake the study of quasicoherent sheaves, and in particular vector bundles, on Olsson fans: we describe their moduli stack in the case of a (subdivided) cone, and some failures of algebraicity in the general case; Weyl convexity shows up naturally in this context. 
\end{abstract}

\maketitle

\setcounter{tocdepth}{1}
\tableofcontents

\section{Introduction}
\subsection{Olsson fans and their étale sheaves}
 We undertake a preliminary study of a relative notion of Artin fans over a general fine and saturated logarithmic scheme.  We call these \emph{Olsson fans}, since they appeared first in Olsson's work on logarithmic structures and algebraic stacks~\cite{olssonlogarithmic}~: Olsson's stack of relative logarithmic structures is the prototypical example. Artin cones embody equivariant toric geometry~; Olsson fans allow gluing and variation in a family.

An Olsson fan over $S$ may be defined succinctly as an algebraic stack with logarithmic structure that is logarithmically \'etale over $S$.  Our first result asserts that Olsson fans always have (strict) \'etale charts by Olsson cones, namely those Olsson fans that are representable by a saturated extension of the characteristic monoid of $S$.

\begin{theorem} \label{thm:1}
	Let $S$ be a fine, saturated logarithmic scheme and let $\Sigma$ be an Olsson fan over $S$.  Then $\Sigma$ admits a strict \'etale cover by Olsson cones.
\end{theorem}

See Theorem~\ref{thm:localstructureOlssonfans} for the proof, which involves a study of \'etale sheaves on Olsson fans.  The next theorem, when combined with Theorem~\ref{thm:1}, shows that the cohomology of \'etale sheaves on Olsson fans is determined combinatorially~:

\begin{theorem} \label{thm:2}
	Let $S$ be a fine, saturated logarithmic scheme and let $\pi : \sigma \to S$ be an Olsson cone over $S$.  Then $\mathrm R^p \pi_\ast F = 0$ for all \'etale sheaves of abelian groups $F$ on $\sigma$ and all $p > 0$.
\end{theorem}
Indeed, working \'etale locally on $S$, these results imply that the cohomology of an \'etale sheaf $F$ on a Olsson fan $\Sigma$ covered by Olsson cones $\sigma_i\to S$ is the cohomology of the
\v{C}ech complex $C^{\bullet}(\mathcal U,F)=\prod_{(i_0\dots i_n)}\Gamma(\sigma_{i_0}\times_\Sigma\sigma_{i_1}\dots\times_\Sigma \sigma_{i_n},  F)$. 
See Corollary~\ref{cor:12}.

\subsection{Weakly convex cones and locally free sheaves}
Ultimately, our goal is to use Olsson fans in our work on the logarithmic Grassmannian and logarithmic linear series.  For these applications, we need to consider locally free sheaves on Olsson fans, as well as on tropical tori and non-strictly convex cones within them. Recall that the tropical multiplicative group is defined by
\begin{equation*}
	\tropGm(S) = \Gamma(S, \bar M_S^{\rm gp}) ;
\end{equation*}
a tropical torus over $S$ is a sheaf $\uHom(X, \tropGm)$ on fine, saturated logarithmic schemes over $S$, where $X$ is an \'etale sheaf of finitely generated, torsion-free abelian groups, called the character group of the torus.  In order to formulate our results in a uniform way that applies both to Olsson cones and to tropical tori, we introduce \emph{weakly convex Olsson cones}, which are convex --- but not necessarily strictly convex --- regions in tropical tori.

For the next few results, we recall an important construction.  Suppose that $S$ is a fine, saturated logarithmic scheme and $\alpha$ is a section of $\bar M_S^{\rm gp}$.  The preimage of $\alpha$ under the projection $M_S^{\rm gp} \to \bar M_S^{\rm gp}$ is a $\mathcal O_S^\ast$-torsor that we denote $\mathcal O_S^\ast(-\alpha)$.  We write $\mathcal O_S(\alpha)$ for the invertible sheaf of $\mathcal O_S^\ast$-equivariant morphisms $\mathcal O_S^\ast(-\alpha) \to \mathcal O_S$.

When $\sigma$ is an Olsson cone over $S$, we call the sections of $\bar M_\sigma^{\rm gp}$ the \emph{linear functions} on $\sigma$, and we call the sections of $\bar M_S^{\rm gp}$ the \emph{constants}.

\begin{theorem} \label{thm:4}
	Let $S$ be a fine, saturated logarithmic scheme.  Let $\pi : \sigma \to S$ be a universally surjective, weakly convex Olsson cone over $S$. 
    Then, locally in the strict \'etale topology of $S$, every invertible sheaf $\mathcal L$ on $\sigma$ is isomorphic to $\pi^\ast L(\lambda)$ for some invertible sheaf $L$ on $S$ and some linear function $\lambda$ on $\sigma$.  The linear function $\lambda$ is uniquely determined by $\mathcal L$ up to addition of constants.  Furthermore, $\mathrm R^p \pi_\ast \mathcal L = 0$ for all $p > 0$ and the natural homomorphism
	\begin{equation*}
		\underset{\mu \leq \lambda}{\varinjlim}\ L(\mu) \to \pi_\ast \mathcal L
	\end{equation*}
	is an isomorphism.  (The colimit is taken over all constants $\mu \leq \lambda$.)
\end{theorem}

\begin{theorem} \label{thm:3}
	Let $S$ be a fine, saturated logarithmic scheme.  Let $\sigma$ be a universally surjective, weakly convex Olsson cone over $S$.  Let $\mathcal E$ be a locally free sheaf on $\sigma$.  Then there is a strict \'etale cover $f : S' \to S$ such that $f^\ast \mathcal E$ splits as a direct sum of invertible sheaves.
\end{theorem}

Together, Theorems~\ref{thm:4} and~\ref{thm:3} characterize locally free sheaves on  weakly convex cones completely, generalizing Klyachko's description~\cite{Klyachko} of equivariant vector bundles on affine toric varieties.

\begin{corollary} \label{cor:2}
	Let $S$ be a fine, saturated logarithmic scheme.  Let $\sigma$ be a universally surjective, integral, weakly convex Olsson cone over $S$, and let $\Sigma$ be an integral subdivision of $\sigma$.  Then the category fibered in groupoids over $S$ parameterizing locally free sheaves on $\Sigma$ is an algebraic stack.
\end{corollary}

Theorems~\ref{thm:4} and~\ref{thm:3} can be generalized prosaically to Olsson fans.  We have not troubled to formulate these statements because Corollary~\ref{cor:2} does not generalize~: the stack of locally free sheaves on Olsson fans often fails to be algebraic (see Section~\ref{sec:moduli}).  However, we feel that the following corollary may give a clue about how to think about vector bundles on Olsson fans more generally~: 

\begin{corollary} \label{cor:7}
	Let $S$ be a fine, saturated logarithmic scheme.  Let $T$ be a tropical torus over $S$ spanned by a weakly convex Olsson cone $\sigma$.  Let $\tau$ be the Weyl convex hull of $\sigma$ in $T$.  Then any locally free sheaf on $\sigma$ extends uniquely (up to unique isomorphism) to $\tau$.
\end{corollary}
See Definition~\ref{def:3} and Corollary~\ref{cor:13} for the details.

\subsection{History}
A toroidal structure on a scheme $X$ is an open subset $U \subset X$ such that the pair $(X,U)$ is locally, in the \'etale topology, isomorphic to $(Y,V)$, where $Y$ is a toric variety and $V$ is its dense torus.  By way of the local isomorphism, the combinatorics of the fan of a local model $(Y,V)$ correspond to local geometric features of $(X,V)$~\cite{KKMSD}.

Kato observed~\cite{Kato-toric} that this relationship can be made more direct using logarithmic geometry.  A logarithmic scheme $X$ has a dual identity, both as a scheme via its structure sheaf $\mathcal O_X$, and as a monoidal space via its characteristic monoid $\bar M_X$.  Kato showed that the basic theory of schemes can be repeated with monoids --- gluing spectra of monoids along open subsets --- to yield a theory of what are now called \emph{Kato fans}~\cite{ACMUW} or \emph{monoschemes}~\cite{Ogus}.  Since a logarithmic scheme is also a monoidal space, one can describe morphisms from logarithmic schemes to Kato fans, and therefore transport combinatorial constructions from fans to logarithmic schemes by pullback.  For example, a toroidal blowup arises as the pullback of a subdivision of Kato fans.

The building blocks of Kato fans are spectra of monoids.  Each monoid spectum represents a functor on logarithmic schemes, and Olsson showed that the functor represented by a monoid $P$ is always representable by an algebraic stack $\mathcal A_P$ with a logarithmic structure.  Olsson used these local models to create a moduli stack of logarithmic structures~\cite{olssonlogarithmic}.

Olsson's stacks $\mathcal A_P$ can also be glued together to obtain algebraic stacks that represent Kato's fans~\cite{AWbirational,ACMUW,CCUW}.  In a moment of uninspired eponymization, these became known as \emph{Artin fans}.  

Initially, the utility of Artin fans was to guarantee algebraicity of certain constructions involving the interaction between algebraic geometry and combinatorics, such as stable maps into fans.  These applications only made use of maps \emph{from} logarithmic schemes \emph{into} fans.  It was therefore not been necessary to give any deep consideration to the base over which the Artin fans were defined.  However, it has recently become evident that maps \emph{from} fans into algebraic stacks also contain important information, relating to stability~\cite{Heinloth2017, HL2018}, the existence of good moduli spaces~\cite{AHLH2023}, and Donaldson--Thomas theory \cite{intrinsicDT}.

Finally, the expansions $X^\dagger$ of a normal crossing pair considered in logarithmic Gromov--Witten and Donaldson--Thomas theory \cite{DhruvExpansions,MaulikRanganathan1} are instances of Olsson fans (over $X\times\Spec(k,\NN)$). The study of families of coherent sheaves on them is the heart of the novel logarithmic sheaf 
theory \cite{MaulikRanganathan1,MaulikRanganathan2,Kennedy-Hunt,PKHDR}.  Relatedly, Olsson fans also appear to be a natural language in which to discuss Lafforgue's compactifications of thin Schubert cells~\cite{Lafforgue}.

Along with logarithmic linear series (to be described below), these applications motivate the present work.

\subsection{Logarithmic linear series}
Our initial motivation came from the following classical question~: what is the limit of a linear series when a smooth curve becomes nodal?  Recall Eisenbud and Harris' answer in the case of a curve $X_0$ with two components $D,E$ and a single node $q$ \cite{EHLLS}: let $L_D$ be a line bundle of multidegree $(d,0)$, and $L_E$ its twist of multidegree $(0,d)$; a refined \emph{limit linear series} consists of two vector spaces of sections $V_D\subseteq H^0(D,L_D)$ and $V_E\subseteq H^0(E,L_E)$ such that their \emph{vanishing sequences} at the node $q$ \emph{match}. The vanishing behaviour of sections at $q$ makes $V_D$ and $V_E$ into filtered vector spaces, and matching means precisely an identification of the associated graded vector spaces (up to a reversal and a shift) $\text{gr}(V_D^\bullet)\cong \text{gr}(V_E^{d-\bullet})$.
These data are precisely equivalent to the data of a vector bundle on $\Gamma=\mathcal A^1\bigsqcup_{B\Gm}\mathcal A_1$, and $\Gamma$ is the Olsson fan of the nodal curve $X_0$ when it is endowed with its natural logarithmic structure.

We were therefore led to consider the moduli problem parameterizing the following data, which we call a \emph{logarithmic linear series} \cite{lls}: a logarithmic curve $X$ with Olsson fan $\Gamma$, a logarithmic line bundle $L$ on $X$, and a vector subbundle $V$ of the pushforward of $L$ to $\Gamma$.  In investigating the algebraicity of this moduli problem we were led to study vector bundles on Olsson fans. 

We show below in \Cref{sec:moduli} that the stack parameterizing $V$ by itself is not algebraic in general.  The construction of an algebraic stack of logarithmic linear series will therefore require a restriction on $V$.  We defer further discussion of this additional restriction to future work, but we remark that a family of examples of suitably restricted vector bundles can be found in the tautological bundles on Lafforgue's compactifications of thin Schubert cells \cite{Lafforgue,lg}.

\subsection{Acknowledgments}
This paper is part of our ongoing effort to understand limit linear series from a logarithmic modular perspective. We have benefitted from conversations with Omid Amini, Andreas Gross, Patrick Kennedy-Hunt, Arne Kuhrs, Kevin K\"uhn, Diane Maclagan, Sam Molcho, Navid Nabijou, Sam Payne, Dhruv Ranganathan, Terry Song, Martin Ulirsch, Alejandro Vargas, Annette Werner, and from the hospitality of various institutions, including the Universities of Berlin (Humboldt), Bologna, Colorado (Boulder), Cambridge, Frankfurt, Geneva, Heidelberg, Lausanne (EPFL), and the Lorentz Center in Leiden.

\resumetoc
\tableofcontents

\section{Monoids and logarithmic structures}\label{sec:monoids}
\tocdesc{establishes logarithmic terminology, provides a dictionary between properties of monoid homomorphism and rational polyhedral cones}

\subsection{Properties of monoid homomorphisms}

Given a (commutative) monoid $P$, and a set with $P$-action $S$ (a $P$-set), we say $s_1\leq s_2$ in $S$ if there exists a $p\in P$ such that $p+s_1=s_2$ ; this applies in particular to $P$ itself, and to $P^{\rm gp}$. In this case, $\leq$ is a preorder relation, and an order relation if and only if $P$ is sharp. If $P$ is integral, the specification of $P$ is equivalent to the specification of the preordered abelian group $(P^{\rm gp},\leq)$. A monoid $P$ is called \emph{valuative} if the preordering of $P^{\rm gp}$ is total (in particular, $P$ is saturated).  A monoid \emph{valuation} is a monoid homomorphism to a valuative monoid. 
\begin{definition}
    A monoid homomorphism $\mm\colon\mo\to\mt$ is called \emph{local}%
		\footnote{This condition is also sometimes (particularly in works by the third author) called \emph{sharpness}.  Here we follow the terminology of \cite{Ogus}.}
		if $\mm^{-1}(\mt^*)=\mo^*$, or equivalently $\mm^{-1}(\mt^+)=\mo^+$, where $P^*$ denotes the subgroup of units of $P$, and $P^+=P\setminus P^*$ is the maximal ideal.
\end{definition}

\begin{definition}
    Let $\mo,\mt$ be integral monoids. A morphism $\mm\colon\mo\to\mt$ is called \emph{exact} if the following diagram is cartesian~:
    \bcd
    \mo\ar[r,"\mm"]\ar[d,hook] & \mt\ar[d,hook] \\
    \mo^{\rm gp}\ar[r,"\mm^{\rm gp}"] & \mt^{\rm gp}
    \ecd
		Equivalently~: if $x \in \mo^{\rm gp}$ and $\mm^{\rm gp}(x) \geq 0$ then $x \geq 0$.  Also equivalently~: $x,y \in \mo^{\rm gp}$ have comparable images in $\mt^{\rm gp}$, then $x$ and $y$ are already comparable in $\mo^{\rm gp}$.
\end{definition}

\begin{remark}\label{rmk:exact-local}
  An exact homomorphism $\mm$ is local, and it is injective if $\mo$ is sharp 
  (see \cite[Proposition I.2.1.15]{Ogus}~; in this case, $\mm^{\rm gp}$ is injective too).  
\end{remark}
 
\begin{example}
	 The morphism $s\colon\NN^2\to\NN,(a,b)\mapsto a+b$ is local but not injective. In particular it is not exact. 
\end{example}

 The next proposition gives a characterization of exactness in terms of the dual (strictly convex rational polyhedral) cones $\sigma_P=\on{Hom}(P,\mathbf R_{\geq 0})$ when both $\mo$ and $\mt$ are finitely generated and saturated.  In order to formulate the proposition for monoids that are not necessarily finitely generated, we need a suitable replacement for $\sigma_P$.  For this, we use the set of \emph{all} valuations, recognizing that $\sigma_P$ is essentially the set of rank~$1$ valuations (the only fine, sharp valuative monoids are $0$ and $\mathbf N$ \cite[Proposition I.2.4.2]{Ogus}).

\begin{remark}
    The spectrum of a valuative monoid is totally ordered with respect to the specialization relation~; hence, one can think of a higher-rank valuation of $P$ geometrically as the choice of a flag of cones in a subdivision of $\sigma_P$.
\end{remark}

We need one piece of terminology before stating the proposition.  Given a monoid homomorphism $\mm\colon\mo\to\mt$, we say that a valuation $g\colon\mt\to W$ \emph{extends} a valuation $f\colon\mo\to V$ if there exists a local morphism $V\to W$ making the following diagram commute~:
 \bcd
P\ar[r,"h"]\ar[d,"f"] & Q\ar[d,"g"]\\
V\ar[r] & W
 \ecd
\begin{proposition} \label{prop:6}
	Suppose that $\mm : \mo \to \mt$ is a homomorphism of saturated monoids.  Consider the following conditions~:
	\begin{enumerate}[label=(\roman*)]
		\item \label{it:5} The homomorphism $\mm$ is exact.
		\item \label{it:6} Every valuation of $\mo$ extends 
			to a valuation of $\mt$.
		\item \label{it:7} Every rank~$1$ valuation of $\mo$ extends to a valuation of $\mt$.
		\item \label{it:8} Every rank~$1$ valuation of $\mo$ extends to a rank~$1$ valuation of $\mt$.
	\end{enumerate}
	Conditions~\ref{it:5} and~\ref{it:6} are equivalent and imply~\ref{it:7}.  If $P$ is finitely generated then~\ref{it:7} is equivalent to~\ref{it:5} and~\ref{it:6}.  If $\mt$ is finitely generated then~\ref{it:5} implies~\ref{it:8}.  If both $P$ and $Q$ are finitely generated then all of the conditions are equivalent.
\end{proposition}

\begin{corollary} \label{cor:11}
	Let $\mm : \mo \to \mt$ be a homomorphism of integral, saturated, finitely generated monoids, let $\phi \colon \sigma_{\mt} \to \sigma_{\mo}$ be the homomorphism of strictly convex rational polyhedral cones dual to $\mm$.  Then $\mm$ is exact if and only if $\phi$ is surjective.
\end{corollary}

\begin{proof}[Proof of Proposition~\ref{prop:6}]
	Suppose first that every valuation of $\mo$ extends to a valuation of $\mt$ and that $x \in \mo^{\rm gp}$ is an element such that $\mm^{\rm gp}(x) \geq 0$.  Then $g(\mm^{\rm gp}(x)) \geq 0$ for all valuations $g$ of $\mt$, and therefore $f(x) \geq 0$ for all valuations $f$ of $\mo$.  But if $f(x) \geq 0$ for every valuation of $\mo$ then $x \geq 0$,\footnote{This follows from Hahn's embedding theorem \cite{Hahn}, but a direct proof is quicker~: Use Zorn's lemma to obtain a submonoid $V$ of $\mo^{\rm gp}$ that is maximal with respect to inclusion among the submonoids with $\mo[-x] \subset V$ and for which this inclusion is local.  Then $f : \mo \to V$ is a valuation such that $f(x) \leq 0$.  But $f(x) \geq 0$ as well,  by assumption.  Since $\mo[-x] \to V$ is local, 
    $x$ must be a unit of $\mo[-x]$, so $x \in \mo$. See also \cite[Proposition I.2.4.1]{Ogus}.} so we conclude that $\mm$ is exact.
	Thus~\ref{it:6} implies~\ref{it:5}.

	If $\mo$ is finitely generated (so $\mo/\mo^\ast$ is dual to a strictly convex rational polyhedral cone) then $x \geq 0$ if and only if $f(x) \geq 0$ for all rank~$1$ valuations of $\mo$ \cite[p.\ 9]{Fulton}. Thus the last step of the argument in the last paragraph only requires $f$ to have rank~$1$ when $\mo$ is finitely generated.  Thus~\ref{it:7} implies~\ref{it:5} when $\mo$ is finitely generated.

	Now assume that $\mm$ is exact and let $f$ be a valuation of $\mo$.  We wish to show $f$ extends to a valuation of $\mt$.  Since the pushout of an exact morphism of monoids is exact~\cite[Proposition~I.3.2~(2)]{Tsuji}, by pushing out along $f$ we reduce to the case where $\mo$ is valuative.  Since $\mm$ is local (\Cref{rmk:exact-local}), any local valuation of $\mt$ extends $f$. This shows that~\ref{it:5} implies~\ref{it:6}.

	If $\mt$ is finitely generated and $f$ is a rank~$1$ valuation then the argument of the last paragraph reduces to the case $\mo = \mathbf N$.  In this case, $\mt$ is finitely generated so it has a local rank~$1$ valuation, which automatically extends $f$.  This shows that~\ref{it:5} implies~\ref{it:8} when $\mt$ is finitely generated.

	Since \ref{it:8} clearly implies \ref{it:7}, we conclude that all of the conditions are equivalent when $P$ and $Q$ are both finitely generated.
\end{proof}

\begin{example}
	Let $\mt = \mathbf Nx + \mathbf Ny$ and let $\mo$ be the set of $ax + by$ in $\mt$ such that either $a = b = 0$ or $b > 0$.  Then $\mo \hookrightarrow \mt$ is not exact but every rank~$1$ valuation of $\mo$ extends to a rank~$1$ valuation of $\mt$.  There is a rank~$2$ valuation of $\mo$ that does not extend to a valuation of $\mt$, in which $y > 0$ and $x < 0$ and $nx \leq y$ for all integers $n$.
\end{example}

The next definition incorporates {\cite[Proposition 4.1]{KatoK}}.
\begin{definition}\label{def:integral}
Let $\mo$ and $\mt$ be integral monoids. A morphism $\mm\colon\mo\to\mt$ is called \emph{integral} if the following equivalent conditions are satisfied:
    \begin{enumerate}[label=(\arabic*)]
	\item \label{it:1} Given $x_1,x_2\in \mo$ and $y_1,y_2\in \mt$ such that $\mm(x_1)+y_1=\mm(x_2)+y_2$ then there exist $x_3,x_4\in \mo$ and $y\in \mt$ such that $y_1=\mm(x_3)+y,\;y_2=\mm(x_4)+y$,  and $x_1+x_3=x_2+x_4$.%
    \item For each morphism $\mo\to \mo'$ with $\mo'$ integral, the push-out $\mo'\oplus_{\mo}\mt$ is integral.
    \end{enumerate}
    Furthermore, if $\mm$ is injective, it is integral if and only if the induced ring homomorphism $\ZZ[\mo]\to\ZZ[\mt]$ is flat (equivalently, flat after base-change to any field).
   \end{definition}
   We shall denote by $X_P$ the toric variety $\Spec\ZZ[P]$.
\begin{remark}\label{rmk:integral-exact}
     A local, integral morphism is exact~\cite[Proposition~I.2.11]{Tsuji}. 
\end{remark}

 \begin{example}\label{exa:exact_not_integral}
	 The homomorphism
	 \begin{equation*}
		 \mm \: : \: \mo = \mathbf N x \!+\! \mathbf N y  \longrightarrow  \mt = \mathbf N x \!+\! \mathbf N y \!+\! \mathbf N (z\!-\!x) \!+\! \mathbf N (z\!-\!y) \quad \subset \quad \mathbf Z x \!+\! \mathbf Z y \!+\! \mathbf Z z
	 \end{equation*}
	 is exact but not integral.  
     Indeed, $\mm(x)+(z-x)=\mm(y)+(z-y)$ but there is no non-trivial decomposition of the form $\mm(p)+q$ for either $z-x$ or $z-y$.

 \end{example}
 
\begin{remark}{\label{rmk:integral-free}}
	Suppose that $\mm : \mo \to \mt$ is an integral morphism of fine, sharp monoids.  Then \Cref{def:integral}~\ref{it:1} shows that $\mt$ is filtered in the partial order determined by the action of $\mo$.  Since $\mt$ is also noetherian, this implies that $\mt$ is \emph{free} as a $\mo$-module \cite[\S 1]{KatoF}, in the sense that there exists a subset $S\subseteq \mt$ such that the function $\mo\times S\to \mt,\ (x,s)\mapsto \mm(x)+s$ is a bijection. Such a subset $S$ is referred to as a \emph{basis} for $\mt$ as a module over $\mo$~; clearly, $S$ is in bijection with $\mt/\mo$. A basis can be a submonoid of $\mt$ if and only if $\mm$ is a split injection.
\end{remark}

In the next proposition, we investigate some of the geometric consequences of integrality.  For the statement, recall that the localization of $P$ by a subset $S$ is the sharp quotient $P[-S]^\sharp = P[-S] / P[-S]^\ast$ of $P$.  The homomorphism $P \to P[-S]^\sharp$ is initial among homomorphisms from $P$ to sharp monoids $R$ that carry $S$ to $0$.  Similarly we can define the localization of a morphism $\mm\colon \mo\to\mt$ as the localization of $\mo$ at $\mm^{-1}(0)$ if $\mt$ is sharp.  If $\sigma$ is the rational polyhedral cone dual to a fine, saturated monoid $P$, and $\tau$ is a face of $\sigma$, then $\tau$ is dual to the localization of $P$ by the submonoid of linear functions vanishing on $\tau$.

\begin{proposition} \label{prop:1}
	 Let $\mm : \mo \to \mt$ be a homomorphism of fine, saturated, sharp monoids.  Consider the following conditions: 
	 \begin{enumerate}[label=(\roman*)]
		 \item \label{it:4} The homomorphism $\mm$ is integral.
		 \item \label{it:10} For every localization $\mt \to \mt'$ of $\mt$, the localization 
			 of the composition $\mo \to \mt \to \mt'$ is exact.
		 \item \label{it:11} Let $\mm' : \mo' \to \mt$ be the localization of $\mo \to \mt$.  For every $q \in \mt^{\rm gp}$, if there is some $p \in (\mo')^{\rm gp}$ such that $\mm'(p) \leq q$, then there is a maximal one.
	 \end{enumerate}
		 Then \ref{it:4} and \ref{it:11} are equivalent and imply \ref{it:10}.
 \end{proposition}

We will see in \Cref{prop:saturated}, below, that all three conditions are equivalent in the presence of a quasisaturation hypothesis.

 \begin{proof}
	 
     If $\mm$ is integral, and $\mt \to \mt'$ is a localization, then the localization $\mo' \to \mt'$ is integral by \cite[Propositon~I.2.3 and Proposition~I.2.7]{Tsuji}.  Since $\mo' \to \mt'$ is local by construction, 
     \Cref{rmk:integral-exact} implies it is exact.  Thus~\ref{it:4} implies~\ref{it:10}.  


	 To see that \ref{it:11} implies \ref{it:4}, we use Kato's equational criterion (Definition~\ref{def:integral}~\eqref{it:1}). 
     Suppose that $\mm(x_1) + y_1 = q = \mm(x_2) + y_2$.  Let $x$ be a maximal element of $\mo^{\rm gp}$ such that $\mm(x) \leq q$.  Set $y = q - \mm(x)$, which is in $\mt$ since $\mm(x) \leq q$.  Let $x_3 = x - x_1$ and $x_4 = x - x_2$.  These are in $P$ since $x \geq x_1, x_2$.  Then $x_1 + x_3 = x_2 + x_4$ and $\mm(x_3) + y = q - \mm(x_1) = y_1$ and $\mm(x_4) + y =q - \mm(x_2) = y_2$, as required.

     On the other hand, to see that \ref{it:4} implies \ref{it:11}, we assume without loss of generality (again by \cite[Propositon~I.2.3 and Proposition~I.2.7]{Tsuji}) that $\mm$ is local. 
     Kato's equational criterion implies that the set of lower bounds \[\mo_{\leq q}=\{x\in P^{\rm gp}\:|\:h(x)\leq q\}\] is either empty or filtered. Assume it is not empty. 
     By sending $x$ to $q - x$, an ascending sequence $x_1 \leq x_2 \leq \cdots$ in $\mo_{\leq q}$ induces a descending sequence $q - x_1 \geq q - x_2 \geq \cdots$ in $\mt$ (bounded below by $0$).  Any such sequence stabilizes because $\mt$ is finitely generated.  Since $\mm$ is injective, this implies the original sequence stabilizes.  Thus $\mo_{\leq q}$ is a filtered partially ordered set in which ascending chains stabilize, hence has a maximal element.
\end{proof}

We explore condition \ref{it:11}, the existence of lower bounds. We write $\inf(q)=\sup \mo_{\leq q}$ for the maximum value of $\mo_{\leq q}$, with the understanding that $\inf(q)= -\infty$ if $\mo_{\leq q}$ is empty.  If $\inf(q)$ exists and $\inf(q) \neq-\infty$, we also write $\inf(q) = \min(q)$.  We show that infima behave well under localization~:

		\begin{lemma} \label{prop:7}
			Suppose that
			\begin{equation*} \begin{tikzcd}
				\mo \ar[r] \ar[d] & \mo' \ar[d] \\
				\mt \ar[r] & \mt'
			\end{tikzcd}
			\end{equation*}
			is a commutative square of fine, saturated, sharp monoids in which the horizontal morphisms are localizations.
      Let $q$ be an element of $\mt$ and let $q'$ be its image in $\mt'$.  Then $\inf(q')$ is the image of $\inf(q)$ if the latter exists.
		\end{lemma}
		\begin{proof}
			Write $\mo' = \mo[-p]^\sharp$ 
						for some $p \in \mo$.  If $x \in \mo^{\rm gp}$ represents an element of $\mo'_{\leq q'}$ then $x \leq q + n p$ for some integer $n \geq 0$.  Therefore $x - n p \leq q$ so $x - n p$ lies in $\mo_{\leq q}$.  Thus $\mo_{\leq q}$ surjects onto $\mo'_{\leq q'}$.  In particular $\mo'_{\leq q'} = \varnothing$ if $\mo_{\leq q} = \varnothing$ so $\inf(q') = -\infty$ if $\inf(q) = -\infty$.  Moreover, the order on $\mo'_{\leq q'}$ is induced from the order of $\mo_{\leq q}$, so if $\mo_{\leq q}$ and $\mo'_{\leq q'}$ are nonempty then the maximal element of $\mo_{\leq q}$ must map to the maximal element of $\mo'_{\leq q'}$.  
		\end{proof}

The following condition ensures the finiteness of all infima.

 \begin{definition}[{\cite[Definition~I.4.3.1]{Ogus}}]
	 A morphism of integral monoids $\mm\colon\mo\to\mt$ is \emph{vertical} (or \emph{bounded}) if $\operatorname{coker}(\mm)$, or equivalently the image of ${\mt}$ in ${\mt}^{\rm gp} / {\mo}^{\rm gp}$, is a group.  Equivalently, for every $q\in {\mt}$, there exists $p\in {\mo}$ such that $q\leq \mm(p)$. 
 \end{definition}
 \begin{corollary} \label{cor:9}
			Let $\mm : \mo \to \mt$ be an integral and vertical homomorphism of fine, saturated, and sharp monoids.  Then every $q \in \mt^{\rm gp}$ has a minimum $\min(q)\in\mo^{\rm gp}$.
		\end{corollary}

We now return to condition \ref{it:10} of \Cref{prop:1}.  By \Cref{cor:11}, this says that every face $\tau$ of $\sigma_{\mt}$ surjects onto a face of $\sigma_{\mo}$.
 This condition is equivalent to the equidimensionality of the fibres of the corresponding morphism of toric varieties \cite[Lemma 4.1]{AbramovichKaru}. If, in addition, $X_\mo$ is a smooth toric variety, then, by miracle flatness, we have obtained a characterization of integrality in terms of cones. In general, though, surjectivity on cones is weaker than integrality, as the following example shows.
\begin{example}
    Let $\mo=\NN u\oplus\NN v\oplus \NN w/(u+w=2v)$ and $\mm\colon \mo\to\mt=\NN^2$ mapping $u\mapsto(2,0),\ v\mapsto(1,1),\ w\mapsto(0,2)$. Then $\mm$ is exact, and so are its localizations: inverting $u$ or $w$, we obtain the identity morphism $\NN\to\NN$; inverting more we obtain $0$. On the other hand, $\mm$ is not integral~: $\mt$ is clearly not free as a $\mo$-module~; alternatively,  $h(u)+(0,1)=h(v)+(1,0)$ holds, but there is no non-trivial decomposition of $(1,0)$ or $(0,1)$.  

		Geometrically, $\mt$ represents integer-valued functions on the interval with  integer slope and $\mo$ represents integer-valued functions on the interval with \emph{even} integer slope.  The element $(2,1)$ represents the function with values $2$ and $1$ at the endpoints of the interval and the functions $(2,0)$ and $(1,1)$ in $\mo_{\leq (2,1)}$ have no common upper bound.
\end{example}

Next we introduce saturation, which is a strengthening of integrality that admits a clearer geometric interpretation. We summarize \cite[Definition~I.3.5, Definition~I.3.7, Proposition~I.3.8]{Tsuji}~:
\begin{definition}\label{def:quasisaturated}
	Let $\mm : \mo \to \mt$ be a homomorphism of integral monoids.  Let $\mt'$ be the pushout of $\mm$ along the multiplication-by-$n$ homomorphism $[n] : \mo \to \mo$, in the category of integral monoids.  The multiplication-by-$n$ homomorphism $[n] : \mt \to \mt$ factors uniquely through $h' : \mt' \to \mt$.  We say that $\mm$ is \emph{quasisaturated} if the morphism $h'$ described above is exact for every positive integer~$n$.
\end{definition}

\begin{lemma}[{\cite[Theorem~2.1.4]{molcho2021universal}}] \label{lem:4}
	Suppose $\mo$ and $\mt$ are finitely generated, integral, sharp monoids and that $\mm : \mo \to \mt$ is exact and quasisaturated.  Then every homomorphism $\mo \to \NN$ factors through $\mt$.
\end{lemma}
\begin{proof}
	Suppose that $v : \mo \to \NN$ is a homomorphism that we would like to factor  through $\mt$.  We may replace $\mo$ with $\NN$, and $\mt$ with its pushout along $v$ (in the category of integral monoids).  We therefore assume that $\mo = \NN$.  In particular, $\mo$ is saturated, so $\mt$ is saturated by~\cite[Proposition~I.3.9]{Tsuji}.  We replace $\mt$ with a maximal localization such that $\mo \to \mt$ is still exact~; this is still quasisaturated over $\mo$ by \cite[Proposition~I.3.18]{Tsuji}.  The dual of $\mt$ is then a strictly convex rational polyhedral cone that surjects onto $\mathbf R_{\geq 0}$ and (since we have  chosen a \emph{maximal} localization preserving exactness) it has no proper faces that also surject onto $\mathbf R_{\geq 0}$.  Therefore we can identify $\mt$ with $\NN$ and $\mo \to \mt$ with the multiplication-by-$n$ map for some positive integer $n$.  But this map is saturated only for $n = 1$ \cite[Lemma~I.5.2]{Tsuji}.
\end{proof}

\begin{remark}
	Dually, Lemma~\ref{lem:4} says that an exact, quasisaturated morphism is surjective on integral lattices.
\end{remark}
\begin{example}
    Multiplication by $n$ from $\NN$ to itself is exact but not quasisaturated.
\end{example}
We now summarize \cite[Definition~I.3.12, Proposition I.3.14, Proposition~I.4.1, Theorem~I.6.3]{Tsuji}:
\begin{definition}
	A homomorphism $\mm\colon \mo\to \mt$ of saturated monoids is called \emph{saturated} if it is integral and any of the following equivalent conditions holds~:
	\begin{enumerate}
		\item $\mm$ is quasisaturated.
		\item For each saturated monoid $\mo'$ and for any morphism $\mo\to \mo'$ the pushout $\mo'\oplus_{\mo}\mt$ in the category of integral monoids is saturated.
		\item For any $p\in \mo$, $q\in \mt$, and $n\in\NN$ such that $\mm(p)\leq nq$, there exists a $p'\in \mo$ such that $p \leq n p'$ and $\mm(p') \leq q$.
		\item For any $p \in \mo^+$, $q \in \mt$, and $n \in \NN$ such that $\mm(p) \leq nq$, there exists a $p' \in \mo^+$ such that $\mm(p') \leq q$.
	\end{enumerate}
\end{definition}

The following proposition is essentially contained in \cite[Theorem~2.1.4]{molcho2021universal}, but we include a few additional details in the proof. 

 \begin{proposition} \label{prop:saturated}
	 Let $\mm : \mo \to \mt$ be a homomorphism of fine, saturated, sharp monoids.  If $\mm$ is quasisaturated, then the conditions of \Cref{prop:1} are equivalent, and they are all equivalent to saturation of $\mm$. 
 \end{proposition}
 \begin{proof}

	 We need to show that~\ref{it:10} implies~\ref{it:11} in \Cref{prop:1} under the additional assumption that $\mm$ is quasisaturated. We will assume without loss of generality that $\mm$ is local (by replacing $\mo$ with the localization $\mo'$ of $\mo \to \mt$ ; notice that $\mo\to\mo'$ is an integral morphism).
 
 Suppose that $q \in \mt$.  We abbreviate the set of lower bounds $P_{\leq q} \subset P$ to $S$. Let $p$ be \emph{any} element of $S$.  Let $\mt'$ be the minimal localization of $\mt$ where $q = \mm(p)$.  Let $\mo' = \mo'_p$ be the localization of $\mo \to \mt \to \mt'$.  Then $\mo' \to \mt'$ is exact by \ref{it:10}.  We argue that $\mo \to \mo'$ is universal among homomorphisms from $\mo$ to integral, saturated, sharp monoids $R$ such that $p \geq S$ in $R$.
     
     Indeed, let $\mo \to R$ be such a homomorphism.  Let $\mt_R$ be the pushout of $\mo \to \mt$ along $\mo \to R$, and let $\mt'_R$ be the pushout of $\mo \to \mt'$. Then $\mt'_R$ is the minimal localization of $\mt_R$ in which the images of $q$ and $h(p)$ agree.  
     To show that $\mo \to R$ factors through $\mo'$, it will be enough to show that $R \to \mt'_R$ is local~: then $\mo \to R$ factors through the localization of $\mo \to \mt'_R$, which is a localization of $\mo'$.

	 We check that the induced map $R \to \mt'_R$ is local. Suppose $x \in R$ has image $0$ in $\mt'_R$. Then there is a positive integer $n$ such that $x \leq n(q-p)$ in $Q_R$.  But $R \to \mt_R$ is the pushout of the quasisaturated morphism $\mo \to \mt$, hence is quasisaturated by \cite[Proposition~I.3.6]{Tsuji}, and $\mt_R \to \mt'_R$ is a localization, hence is quasisaturated by \cite[Lemma~I.3.17]{Tsuji}.  Therefore $R \to \mt'_R$ is the composition of quasisaturated morphisms, hence is quasisaturated \cite[Proposition~I.3.6]{Tsuji}.  Therefore~\cite[Proposition~I.4.1]{Tsuji} implies that there is a $y \in R$ such that $x \leq ny$ and $y \leq q - p$ in $\mt_R$. Finally, $p + y \leq q$ so the maximality of $p$ in $S$ on $R$ implies that $y = 0$.  So $x=0$ as well.

	 Now suppose that $v$ is a local valuation of $\mo$. Since by assumption $\mm$ is exact, $v$ extends to $Q$ by \Cref{prop:6}. Then $v(S)$ is bounded above by $v(q)$, so, by definition of a valuative monoid, there is some $p_0 \in S$ such that $v(p_0) \geq v(S)$.  By the universal property of $\mo'_{p_0}$, the valuation $v$ factors through $\mo'_{p_0}$.  But then $\mo \to \mo'_{p_0}$ is local.  Since it was defined to be a localization, this means it is an isomorphism.  Hence $p_0 \geq S$ in $\mo$.
 \end{proof}

\begin{theorem} \label{thm:6}
Let $\mm\colon\mo\to\mt$ be an injective homomorphism of fine, saturated, sharp monoids. 
The following conditions are equivalent:
\begin{enumerate}[label=(\roman*)]
	\item \label{it:12} The homomorphism $\mm$ is saturated;
	\item \label{it:13} The morphism of toric varieties $\mm^\vee\colon X_\mt\to X_\mo$ is flat with reduced fibers;
	\item \label{it:14} For every toric divisor $D_{\rho}\in X_{\mo}$, its preimage under $\mm^\vee$ is generically reduced; 
	\item \label{it:15} The morphism of cones $\phi\colon\sigma_{\mt}\to\sigma_\mo
    $ is \emph{weakly-semistable} in the sense of \cite[Definition~2.1.2]{molcho2021universal}, namely (with $N_P=\Hom(P^{\rm gp},\mathbf Z)$):
    \begin{itemize}
        \item every face $\tau$ of $\sigma_{\mt}$ surjects onto a face of $\sigma_{\mo}$, and
        \item there is an equality of monoids $\phi(\tau\cap N_{\mt})=\phi(\tau)\cap N_{\mo}$.
    \end{itemize}
    
\end{enumerate}

\end{theorem} 

\begin{proof}
	The equivalences $(1)\iff(2)\iff(3)$ are particular cases of the criteria proved in \cite[Theorems~II.4.2 and~II.4.11]{Tsuji}.  Proposition~\ref{prop:saturated}, combined with Proposition~\ref{prop:6} and Lemma~\ref{lem:4}, shows the equivalence of~\ref{it:12} and~\ref{it:15}.
\end{proof}

 \begin{example}
     An important class of morphisms useful in the study of degenerations is that of \emph{partition morphisms (with boundary)} : it is the smallest class of morphisms of fine monoids closed under composition, pushouts, and finite products, and containing the small diagonals $\Delta_k\colon\NN\to\NN^k$ (resp. and the zero sections $0_k\colon 0\to\NN^k$) \cite[Definition 5.8.1]{GillamLogFlat}. If $\mm\colon\mo\to\mt$ is a partition morphism, then it is saturated; 
     it is vertical if and only if it has no boundary \cite[Proposition 5.8.6]{GillamLogFlat}.
 \end{example}

\subsection{Logarithmic geometry and algebraic stacks} \label{sec:6}
 \begin{definition}
     Let $S$ be a scheme. A \emph{logarithmic structure} on $S$ is an \'etale (equivalently fppf \cite[Theorem A.1]{olssonlogarithmic}) sheaf of monoids $M_S$ on S together with a logarithmic%
		 \footnote{That is, every unit of $\mathcal O_S$ has a unique preimage in $M_S$.}
		 homomorphism $\varepsilon : M_S \to \mathcal O_S$, where $\mathcal O_S$ is given its multiplicative monoid structure.  The quotient $M_S/ \varepsilon^{-1} \mathcal O_S^\ast$ is called the \emph{characteristic} sheaf of monoids  and is denoted by $\bar{M}_S$. 

     A morphism of logarithmic structures is a morphism of sheaves of monoids $f\colon M_S\to M_S'$ such that $\varepsilon'\circ f=\varepsilon$.

     A logarithmic scheme $S$ is a pair $(\underline S,\epsilon \colon M_S\to\OO_S)$ where $\underline S$ is a scheme and $\varepsilon$ is a logarithmic structure on $S$.
 \end{definition}

 A morphism of logarithmic schemes $(\underline S, M_S) \to (\underline T, M_T)$ is a morphism of schemes $f : \underline S \to \underline T$ along with a homomorphism $f^{-1} M_T \to M_S$ such that the square
 \begin{equation*} \begin{tikzcd}
	 f^{-1} M_T \ar[r] \ar[d] & f^{-1} \mathcal O_T \ar[d] \\
	 M_S \ar[r] & \mathcal O_S
 \end{tikzcd}
 \end{equation*}
 commutes.  The category of logarithmic schemes is fibered over schemes by the projection that forgets the logarithmic structure.  Logarithmic structures on algebraic spaces and algebraic stacks are defined as cartesian sections of this projection.  See \cite[\S 5]{olssonlogarithmic} for logarithmic structures on algebraic stacks.

Suppose that $S$ is a logarithmic scheme.  There is an exact sequence of sheaves of abelian groups over $S$:
\[0\to\mathcal O_S^\ast\to M_S^{\rm{gp}}\to\bar{M}_S^{\rm{gp}}\to 0 \]

It follows that for each local section $\alpha\in \bar M_S^{\rm{gp}},$ the fiber of $ M_S^{\rm{gp}}\to\bar{M}_S^{\rm{gp}}$ is an $\mathcal O_S^\ast$-torsor.  We denote this torsor $\mathcal O_S^\ast(-\alpha)$ and we write $\mathcal O_S(\alpha)$ for the associated invertible sheaf of $\mathcal O_S^\ast$-equivariant morphisms $\mathcal O_S^\ast(-\alpha) \to \mathcal O_S$. When $\alpha\in \bar{M}_S$ the restriction of $\varepsilon : M_S \to \mathcal O_S$ to the fiber gives an $\mathcal O_S^*$-equivariant morphism  $\varepsilon_\alpha\colon \mathcal O_S^\ast(-\alpha)\to \mathcal O_S$ and therefore a section $\varepsilon_\alpha$ of $\mathcal O_S(\alpha)$, which we also regard as a morphism of invertible sheaves $\varepsilon_\alpha\colon \mathcal O_S(-\alpha)\to \mathcal O_S$.

\begin{definition}
    Let $S$ be a logarithmic scheme.
		A constant sheaf of monoids $P$ on $S$ equipped with a monoid homorphism $\varphi\colon P\to\mathcal O_S$ is said to be a \emph{chart} for $M_S$ if $M_S$ is the initial logarithmic structure $\varepsilon : P^a \to \mathcal O_S$ through which $\varphi$ factors.
    This initial logarithmic structure can be explicitly constructed as the push-out $P\oplus_{\varphi^{-1}\mathcal O_S^\ast}\mathcal O_S^\ast.$
\end{definition}
\begin{definition}
	A \emph{fine} (and \emph{saturated}) logarithmic scheme is a scheme with a logarithmic structure $\varepsilon : M_S \to \mathcal O_S$ that admits an \'etale cover by maps $U\to S$ such that $M_S\rvert_U$ has a chart by a finitely generated, integral (and saturated) monoid. 

    A morphism of logarithmic schemes $T \xrightarrow{f} S$
		is equivalently a morphism of schemes (denoted by the same symbol $f$) together with a morphism of logarithmic structures $f^\ast \colon {f}^*M_S\to M_T$, where 
		$f^*M_S=f^{-1}M_S\oplus_{f^{-1}(\mathcal O_S^\ast)}\mathcal O_T^\ast$.
   The morphism is said \emph{strict} if $f^\ast : f^\ast M_S \to M_T$ is an isomorphism.
\end{definition}


\begin{definition}\label{def:properties_of_hom_of_monoid_sheaves}
    Let $S$ be a fine and saturated logarithmic scheme.  Let $R$ be an \'etale sheaf of sharp, saturated monoids on $S$, equipped with a homomorphism $\bar M_S \to R$. 
    
		We call $R$ \emph{quasicoherent}, relative $\bar M_S$, if there is a cover of $S$ by \'etale morphisms $U \to S$ such that, as a sheaf of sharp, saturated monoids, $R \big|_U$ is defined relative to $\bar M_S \big|_U$ by global generators and relations.  In other words, 
        after passing to an \'etale cover of $S$, it is possible to find a homomorphism of fine, saturated, sharp monoids $P \to Q$, a morphism $P_S \to \bar M_S$, and a cocartesian square
		\begin{equation*} \begin{tikzcd}
			P_S \ar[r] \ar[d] & \bar M_S \ar[d] \\
			Q_S \ar[r] & R
		\end{tikzcd}
		\end{equation*}
		with $P_S$ and $Q_S$ denoting the constant sheaves on the small \'etale site of $S$ associated with $P$ and $Q$.
    
    We say $R$ is of \emph{finite type}, relative to $\bar M_S$, if it is quasicoherent relative to $\bar M_S$ and the set of global generators over the schemes $U$ in the \'etale cover above may be chosen finite.  We say $R$ is of \emph{finite presentation}, relative to $\bar M_S$, if it is of finite type relative to $\bar M_S$ and the set of global relations can also be chosen finite.

    More generally, let $\Pcal$ be a property of monoid homomorphisms which is stable under pushouts. We say that $R$ has property $\Pcal$ relative to $\bar M_S$ if we can find \'etale local charts as above such that $P\to Q$ has property $\Pcal$.
    \end{definition}

    \begin{remark} 
			If $S$ is a fine, saturated logarithmic scheme and $R$ is of finite type over $\bar M_S$ then R\'edei's theorem \cite[Corollary~1.3]{Grillet} implies that $R$ is also of finite presentation over $\bar M_S$.  Over more general bases, the concepts diverge.
    \end{remark}

The construction of charts of fine (and saturated) logarithmic structures and morphisms is discussed in \cite[\S 2]{KatoK} and \cite[\S 2]{olssonlogarithmic}. If $X_P$ denotes the toric variety (with its toric logarithmic structure) associated with the fine, saturated monoid $P$ then $X_P$ represents the functor on logarithmic schemes 
\begin{equation*}
	S \mapsto \on{Hom}\bigl(P,\Gamma(S,M_S)\bigr).
\end{equation*}
If $T$ is the dense torus of $X_P$ then $\Acal_P = [X_P/T]$ represents the functor
\begin{equation*}
	S  \mapsto \on{Hom}\bigl( P, \Gamma(S, \bar M_S) \bigr) .
\end{equation*}
Equivalently, a morphism of algebraic stacks $S \to \Acal_P$ (ignoring logarithmic structures) can be viewed as a system of invertible sheaves $\mathcal O_S(\alpha)$ for each $\alpha \in P^{\rm gp}$, along with morphisms $\mathcal O_S(\alpha) \to \mathcal O_S(\beta)$ whenever $\alpha \leq \beta$ and satisfying expected compatibilities with the group structure of $P^{\rm gp}$ and the partial order induced by $P$.  The stack $\Acal_P$ is known as the \emph{Artin cone} of $P$.

We summarize the main result of \cite[Theorem 1.1 and Corollary 5.25]{olssonlogarithmic}~:
\begin{theorem} \label{thm:7}
    Let $S$ be a fine logarithmic scheme. Consider the following category fibered over the category of $\underline{S}$-schemes:
    \begin{equation*}
        \Log_S(f\colon \underline T\to\underline S)=\{\text{fine logarithmic structures } M_T \text{ over } f^*M_S \text{ on } \underline{T}\}
    \end{equation*}
    with strict morphisms as arrows. Then $\Log_S$ is an algebraic stack of finite presentation over $\underline{S}$. Moreover, there is a representable, \'etale and surjective morphism:
    \begin{equation*}
        \bigsqcup_{\{U,\beta,\mm\}}\underline{U}\times_{\Acal_P}\Acal_Q\to\Log_S,
    \end{equation*}
    where $\underline{U}\to\underline{S}$ is \'etale, $\beta\colon P\to M_{S|U}$ is a chart, and $\mm\colon P\to Q$ is a morphism of finitely generated, integral monoids.
\end{theorem}
\begin{remark}\label{rmk:Log_S}
    An analogous statement can be made if one restricts to the category of fine and saturated logarithmic schemes. The absolute case corresponds to taking $S=\Spec\ZZ$ with the trivial logarithmic structure. Algebraic stacks of the form $\underline{U}\times_{\Acal_P}\Acal_Q$ are the centrepiece of our paper~; we call them \emph{Olsson precones}.
\end{remark}

Suppose $f : S \to T$ is a morphism of logarithmic schemes.  Then $f^\ast$ induces a morphism of characteristic sheaves $f^\ast \colon f^{-1}\bar{M}_T \to \bar{M}_S$.
In particular, local sections of characteristic monoids as well as of their groupifications can be pulled back along morphisms of logarithmic schemes. 
When the meaning is clear from the context, we will usually write $\alpha\in \bar{M}_S$ (instead of $f^\ast \alpha$) for this pullback.

\medskip



If $S$ is a fine, saturated logarithmic scheme, it has a minimal stratification on which $\bar M_S$ is locally constant.  We call this the \emph{logarithmic stratification} of $S$.

\begin{definition}[{\cite[Definition~2.2.2.2]{MolchoWise}}] \label{def:2}
	A fine, saturated logarithmic scheme $S$ is called \emph{atomic} if it is noetherian, has a unique, connected closed stratum, and $\Gamma(X, \bar M_X) \to \bar M_{X,x}$ is a bijection for all points $x$ of the closed stratum.
\end{definition}

\section{Olsson cones and Olsson fans}\label{sec:polyhedra}
\tocdesc{introduces Olsson cones and fans ; shows fans are glued from cones along \'etale relations}
Artin cones are equivariant toric geometries. Olsson (pre)cones are variations of equivariant toric geometries over a base scheme. We study their local structure with a view towards quasicoherent sheaves and vector bundles. We introduce a mild but useful generalization --- weakly convex (pre)cones --- by extending Olsson (pre)cones with a lineality space (tropical torus). Olsson fans are obtained by gluing Olsson cones in a way that can be described combinatorially.

	\subsection{Olsson cones}

\begin{definition}\label{def:Olssoncone}

    Let $S$ be a fine, saturated logarithmic scheme and $R$ a sheaf of fine, saturated, sharp monoids over $\bar M_S$ as in \Cref{def:properties_of_hom_of_monoid_sheaves}. 
    
    Define a sheaf $\Theta_R$ on fine and saturated logarithmic schemes $f : T \to S$ over $S$ to be the set of monoid homomorphisms $f^{-1} R \to \bar M_T$, commuting with the structural homomorphisms from $f^{-1} \bar M_S$~:
			\begin{equation*}
				\Theta_R(T) = \Hom_{f^{-1} \bar M_S}(f^{-1} R, \bar M_T)  .
			\end{equation*}

\end{definition}
\begin{remark}\label{rem:Olssonabsolute}
    If we apply the same construction to a fine, saturated, sharp monoid $R$, viewed as a sheaf on a point with trivial logarithmic structure, we recover the Artin cone:
			\begin{equation*}
				\mathcal A_R(T) = \Hom(R, \bar M_T) .
			\end{equation*}
\end{remark}
The representability of Artin cones extends easily to the relative setting.

		\begin{lemma}\label{lem:reprOlsson1}

			Let $S$ be a fine, saturated logarithmic scheme and let $R$ be an \'etale sheaf of integral, saturated, sharp monoids that is of finite type relative to $\bar M_S$.  Then $\Theta_R$ is representable by an algebraic stack with a fine, saturated logarithmic structure.  If $R$ is integral or saturated over $\bar M_S$ then $\Theta_R$ is integral or saturated over $S$, respectively.
\end{lemma}
			\begin{proof}
				The assertions are all local in the strict \'etale topology of $S$
and,  \'etale-locally in $S$, it is possible to find  fine, saturated, sharp monoids $P,Q$, and homomorphisms $P \to \Gamma(S, \bar M_S)$ and $P\to Q$ such that $R$ fits into a cocartesian diagram:
			\begin{equation*} \begin{tikzcd}
			P \ar[r] \ar[d] & \bar M_S \ar[d] \\
				Q \ar[r] & R
			\end{tikzcd}
			\end{equation*}
			see \cite[\S 2]{KatoK}.  
            Then locally  on $S$, we have $\Theta_R = S \mathop\times_{\Acal_P} \Acal_Q$ and the representability statement follows from the absolute case \cite[Proposition~5.17]{olssonlogarithmic}.
            Moreover, $Q$ can be taken integral or saturated over $P$ if $R$ was assumed to be integral or saturated over $\bar M_S.$
            We conclude by base change that $\Theta_R \to S$ is, respectively, integral or saturated. 
			\end{proof}

\begin{definition}\label{def:Olssoncone2}
   Let $S$ be a fine, saturated logarithmic scheme.  An \emph{Olsson precone} over $S$ is an algebraic stack with logarithmic structure over $S$ that is isomorphic to $\Theta_R$ for an \'etale sheaf of monoids $R$ of finite type over $\bar M_S$. 
\end{definition}

We have already seen that every Olsson precone admits a local presentation as $S \mathop\times_{\mathcal A_P} \mathcal A_Q$ in the strict \'etale topology of $S$.  We will see in Corollary~\ref{cor:1} that every algebraic stack with logarithmic structure over $S$ admitting such a local presentation is an Olsson precone.

We can give a very explicit description of the local structure of Olsson precones which will be useful in the study of quasicoherent sheaves and their sections.

\begin{proposition}\label{prop:local_description_of_relative_Artin_cones}

		Let $\mm\colon P \to Q$ be a homomorphism of fine, saturated, sharp monoids.  Let $S$ be an algebraic stack with a morphism $S \to \Acal_P$.  Let $X_P$ and $X_Q$ be the toric varieties associated with $P$ and $Q$, respectively.  Let $T_P$ and $T_Q$ be the algebraic tori with character groups $P^{\rm gp}$ and $Q^{\rm gp}$, respectively.
			\begin{enumerate}[label=(\roman{*})]
				\item \label{it:23} The product $S \times X_Q$ is represented by $\Spec_S A$ where \[\displaystyle A = \mathcal O_S[Q] = \bigoplus_{q \in Q} x^q \mathcal O_S,\] with $Q^{\rm gp}$-grading induced from $X_Q$. 
				
				\item \label{it:22} The fiber product $S \times_{\mathrm B T_P} X_Q $ is 
				represented by $\Spec_S B$ where 
					\begin{equation*}
						\displaystyle B = \bigoplus_{\alpha \in P^{\rm gp}} y^\alpha x^{\mm(\alpha)} A(\alpha) = 
						\bigoplus_{\substack{\alpha \in P^{\rm gp}, \beta \in Q^{\rm gp} \\ \mm(\alpha) \leq \beta}} y^\alpha x^\beta \mathcal O_S(\alpha).
					\end{equation*}
					The $Q^{\rm{gp}}$-grading extends so that $y^\alpha$ has degree~$0$.
				\item \label{it:19} The fiber product $S \times_{\Acal_P} X_Q $ is 
				represented by $\Spec_S C$ where
					\[C=B\Big/\bigl\langle (\varepsilon_p y^p - 1) x^{\mm(p)} B \bigr\rangle_{p \in P}= B \Big/ \bigl\langle ( \varepsilon_p - y^{-p} ) B(-p)  \bigr\rangle_{p \in P},\]
				and $\epsilon_p$ is the corresponding section of $\OO_S(p)$. This ideal is binomial and homogeneous with respect to the $Q^{\rm{gp}}$-grading.

				\item \label{it:24} $S \mathop\times_{\Acal_P} \Acal_Q = [ \Spec(C) / T_Q ]$ where the $T_Q$-action is induced by the $Q^{\rm{gp}}$-grading above. 

			\end{enumerate}
		\end{proposition}	

			\begin{proof}
				\begin{enumerate}[label=(\roman{*})]
				\item 	
					Omitted.

				\item 
					The fiber product $S \times_{\mathrm B T_P} X_Q$ is the universal scheme over $\Spec A$ equipped with isomorphisms between the invertible sheaves $x^{-\mm(\alpha)} A$ and $A(\alpha) = A \otimes_{\mathcal O_S} \mathcal O_S(\alpha)$ that depend homomorphically on $\alpha \in P^{\rm gp}$.  Equivalently, it is the universal scheme over $\Spec A$ equipped with trivializations of $x^{\mm(\alpha)} A(\alpha)$ depending homomorphically on $\alpha \in P^{\rm gp}$.  This is the spectrum of%
						\footnote{%
							Recall that if $Y$ is an algebraic stack, and $L,M$ are two invertible sheaves on $Y$, the sheaf of isomorphisms from $L$ to $M$ is representable by the following $\Gm$-torsor over $Y$: $$\Spec_Y\bigoplus_{n\in\ZZ}z^n(L \otimes M^\vee)^{\otimes n}.$$%
						}
				\begin{equation*}
					B = \bigoplus_{\alpha \in P^{\rm gp}} y^\alpha x^{\mm(\alpha)} A(\alpha) =
					\bigoplus_{\substack{\alpha \in P^{\rm gp} \\q \in Q }} y^\alpha x^{\mm(\alpha) + q} \mathcal O_S(\alpha)=
						\bigoplus_{\substack{\alpha \in P^{\rm gp}, \beta \in Q^{\rm gp} \\ \mm(\alpha) \leq \beta}} y^\alpha x^\beta \mathcal O_S(\alpha)
				\end{equation*}
						The elements $y^\alpha$ for $\alpha \in P^{\rm gp}$ are just placeholders.  Formally multiplying by $y^p$ gives an isomorphism from $x^{\mm(p)} B$ to $B(-p)$.  By definition, the isomorphisms $y^p$ descend to the quotient $[ \Spec(B) / T_Q ]$, so they must have degree~$0$ in the $Q^{\rm gp}$-grading.

				\item 
					The fiber product $S \mathop\times_{\Acal_P} X_Q$ is defined relative to $S \mathop\times_{\mathrm BT_P} X_Q$ by the condition that the maps of invertible sheaves $\mathcal O(-p) \to \mathcal O$, induced from the two interpretations of $\mathcal O(-p)$ via the two maps $S \times_{\mathrm BT_P} X_Q\to \Acal_P$, should agree.  In our coordinates, one of these maps is the inclusion $x^{\mm(p)} B \to B$ and the other is $\varepsilon_p : B(-p) \to B$.
					By construction, we have 
					\begin{gather*}
						x^{\mm(p)} B = \bigoplus_{\alpha \in P^{\rm gp}} y^\alpha x^{\mm(\alpha + p)} A(\alpha) = \bigoplus_{\alpha \in P^{\rm gp}} y^{-p+\alpha} x^{\mm(\alpha)} A(-p + \alpha) \\
						\intertext{and}
						B(-p) = \bigoplus_{\alpha \in P^{\rm gp}} y^\alpha x^{\mm(\alpha)} A(-p + \alpha)
					\end{gather*}
						These are isomorphic by multiplication by $y^p$.  The fiber product $S \mathop\times_{\Acal_P} X_P$ is defined by the identification of these two maps by way of this isomorphism.  Thus $S \mathop\times_{\Acal_P} X_P = \Spec C$ where $C$ is the quotient of $B$ by the ideal generated by $(\varepsilon_p y^p - 1) x^{\mm(p)} B$, as $p$ ranges through $P$.  Since $y^p : x^{\mm(p)} B \to B(-p)$ is an isomorphism, we can also write the generators for the ideal as $(\varepsilon_p - y^{-p}) B(-p)$.

					\item 
						Immediate.
				\end{enumerate}

			\end{proof}

	\begin{corollary} \label{cor:14}
		Let $\mm : P \to Q$ be a homomorphism of fine, saturated, sharp monoids and let $X_P$ be the toric variety associated with $P$.  Suppose the morphism $S\to\Acal_P$ admits a factorization through $S\to X_P$ corresponding to trivializations $v^p : \mathcal O_S \simeq \mathcal O_S(p)$ for all $p \in P^{\rm gp}$.
		\begin{enumerate}[label=(\roman*)$'$,start=2]
			\item \label{it:29} The fiber product $S \times_{\mathrm B T_P} X_Q $ is represented by $\Spec_S B'$ where 
					\begin{equation*}
						\displaystyle B' = \bigoplus_{\alpha \in P^{\rm gp}} u^\alpha A = \bigoplus_{\substack{\alpha \in P^{\rm gp} \\ q \in Q}} u^\alpha x^q \mathcal O_S.
					\end{equation*}
					The $Q^{\rm{gp}}$-grading extends so that $u^\alpha$ has degree~$\alpha$.
			\item  \label{it:30} The fiber product $S \times_{\Acal_P} X_Q $ is represented by $\Spec_S C'$ where
				\[C'= B' \Big/ \bigl\langle  \varepsilon_p(v^{-p}) - u^{-p} x^{h(p)}    \bigr\rangle_{p \in P} .\]
				The ideal is binomial and homogeneous with respect to the $Q^{\rm{gp}}$-grading.
		\end{enumerate}
	\end{corollary}

	\begin{proof}
		Using the trivializations $v^p$, the ring $B$ from \Cref{prop:local_description_of_relative_Artin_cones}~\ref{it:22} can also be presented as
		\begin{equation*}
			B = \kern-2ex \bigoplus_{\substack{\alpha \in P^{\rm gp}, \beta \in Q^{\rm gp} \\ h(\alpha) \leq \beta}} \kern-2ex v^\alpha y^\alpha x^\beta \mathcal O_S = \bigoplus_{\substack{\alpha \in P^{\rm gp} \\ q \in Q}} u^\alpha x^q \mathcal O_S
		\end{equation*}
		where $u^\alpha = v^\alpha y^\alpha x^{h(\alpha)}$ are units of $B$.  Since $v^{-p}$ generates $B(-p)$, the ideal defining $C$ is generated by $\varepsilon_p(v^{-p}) - y^{-p} v^{-p} = \varepsilon_p(v^{-p}) - u^{-p} x^{h(p)}$ for $p \in P$.
	\end{proof}

	In the setting of \Cref{cor:14}, let $M$ be the image of $P^{\rm gp} \to Q^{\rm gp}$.  Since $M$ is a subgroup of $Q^{\rm gp}$, it is finitely generated and torsion-free, hence free, so we may choose a splitting of the surjection $P^{\rm gp} \to M$.  This induces a splitting $g : P^{\rm gp} \to K$ of the inclusion $K \subset P^{\rm gp}$. We may use it to split off an algebraic torus factor from the spectra of the algebras $B'$ and $C'$ above, as follows.
    
    Writing $C'_q$ for the $q$-th graded piece of $C'$ (with $q \in Q^{\rm gp}$) we can identify $C'_q$ with $C'_{h(p) + q}$ by the unit $u^p$ whenever $p \in M$.  If $\alpha$ is a coset of $M$ in $Q^{\rm gp}$ then, using these identifications, we define $C''_\alpha$ to be the common value of $C'_q$ for all $q \in \alpha$.  This is well defined because $u^{p+p'} = u^p u^{p'}$ for $p \in P^{\rm gp}$.  Then $C'' = \bigoplus_{\alpha \in Q^{\rm gp} / P^{\rm gp}} C''_\alpha$ is a graded ring with generators
	\begin{enumerate}[label=(\roman*)]
		\item \label{it:33} $x^q$ for $q \in Q$, lying in graded degree $[q]$ (where $[q] = q + M$ is the $M$-coset of $Q^{\rm gp}$ containing $q$), and
		\item \label{it:34} $u^p$ for $p \in K$, lying in graded degree $0$,
	\end{enumerate}
	satisfying
	\begin{enumerate}[resume*]
		\item \label{it:35} $x^{q + q'} = x^q x^{q'}$ for all $q, q' \in Q$,
		\item \label{it:36} $u^{p + p'} = u^p u^{p'}$ for all $p, p' \in K$, 
		\item \label{it:37} $x^{h(p)} = u^p \varepsilon_p(v^{-p})$ when $p \in K$, and
		\item \label{it:38} $x^{h(p)} = \varepsilon_p(v^{-p})$ when $p \in M$.%
			\footnote{Note that here we are identifying $C'_q$ with $C'_{h(p)+q}$ using $u^p$ so we have suppressed the $u^p$ that appears in \ref{it:30} of \Cref{cor:14}.}
	\end{enumerate}
	It is possible to recover the graded ring $C'$ from $C''$ by setting $C'_q = C''_{[q]}$ and defining $u^p : C'_q \to C'_{h(p) + q}$ to be the identity map when $p \in M$.

	\begin{corollary} \label{cor:15}
		In the context of \Cref{cor:14}, let $T_{Q/P}$ be the diagonalizable group with character group $Q^{\rm gp} / P^{\rm gp}$.
		\begin{enumerate}[label=(\roman*)$''$,start=2]
			\item \label{it:31} $\bigl[ \Spec(B') / T_{Q} \bigr] \simeq \bigl[ \Spec(B'') / T_{Q/P} \bigr]$, where $B''$ is the graded ring with generators and relations as in~\ref{it:33} through~\ref{it:36}, above.
			\item \label{it:32} $\bigl[ \Spec(C') / T_Q \bigr] \simeq \bigl[ \Spec(C'') / T_{Q/P} \bigr]$, where $C''$ is defined relative to $B''$ by the relations~\ref{it:37} and~\ref{it:38}, above.
		\end{enumerate}
	\end{corollary}

	\begin{proof}
		We may identify $\Spec B' = T_M \times \Spec B''$ and $\Spec C' = T_M \times \Spec C''$ where $T_M$ is the algebraic torus with character group $M$.  Then $T_Q$ acts on the $T_M$ factor through the projection $T_Q \to T_M$.  The kernel of this projection is $T_{Q/P}$, so we may identify $[ \Spec(B') / T_Q ] = [ \Spec(B'') / T_{Q/P} ]$ and $[ \Spec(C') / T_Q ] = [ \Spec(C'') / T_{Q/P} ]$.
	\end{proof}

	We give a reformulation of \Cref{cor:15}, with a more geometric presentation of the proof, independent of (but essentially equivalent to) the argument in \Cref{prop:local_description_of_relative_Artin_cones}.
	
	\begin{proposition}\label{cor:1}
		Let $\mm : P \to Q$ be a homomorphism of fine, saturated, sharp monoids and let $X_P$ be the toric variety associated with $P$.  Let $K$ be the kernel of $\mm^{\rm gp}$ and let $L$ be the cokernel.  Suppose the morphism $S\to\Acal_P$ admits a factorization through $S\to X_P$.  Then
		\begin{equation*}
			S \mathop\times_{\Acal_P} \Acal_Q \simeq S \mathop\times_{X_P} \bigl( T_K \times [  X_Q / T_L ] \bigr)
		\end{equation*}
		where $T_K$ and $T_L$ are the diagonalizable groups with character groups $K$ and $L$, respectively.
	\end{proposition} 
	\begin{proof}
		Let $T_P$ and $T_Q$ denote the dense tori of the toric varieties $X_P$ and $X_Q$, respectively.  There is an exact sequence of diagonalizable groups~:
					\begin{equation} \label{eqn:2}
						0\to T_L \to T_Q\to T_P\to T_K \to 0
					\end{equation}
					Let $M$ be the image of $P^{\rm gp} \to Q^{\rm gp}$, and choose a splitting of the surjection $P^{\rm gp} \to M$.  This induces a splitting
					\begin{equation*}
						T_P \simeq T_K \times T_M
					\end{equation*}
					compatible with the exact sequence~\eqref{eqn:2}.

					Since $X_Q \to \Acal_P$ factors through $X_P$, we have an identification 
$X_Q \mathop\times_{\Acal_P} X_P \simeq X_Q \times T_P \simeq X_Q \times T_K \times T_M$.  The $T_Q$-action on $X_Q \times T_K \times T_M$ operates trivially on $T_K$, so $[ X_Q \times T_P / T_Q ] = T_K \times [ X_Q \times T_M / T_Q ] = T_K \times [ X_Q / T_L ]$. We therefore have 
		\begin{equation*}
			S \mathop\times_{\Acal_P} \Acal_Q = S \mathop\times_{\Acal_P} [ X_Q / T_Q ] = S \mathop\times_{X_P} [ X_Q \times T_P / T_Q ] = S \mathop\times_{X_P} \bigl( [ X_Q / T_L ] \times T_K \bigr) ,
		\end{equation*}
		as claimed.
	\end{proof}

	Note that the isomorphisms of \Cref{cor:15} and of \Cref{cor:1} depend on the choice of splitting of $K \to P^{\rm gp}$.

	The next two examples show some geometric constructions that can be effected using Olsson cones.

\begin{example}
    Let $f\colon S\to\Acal_P$ be a morphism of algebraic stacks, and let $\mm\colon P\to0$ be the zero morphism. Then $S\mathop\times_{\Acal_P}\Acal_0$ picks out the open substack where the logarithmic structure (pulled back along $f$) is trivial. So, for instance, the open complement of a simple normal crossing divisor is an Olsson cone.  The corresponding homomorphism of monoids $\bar M_S \to R$ is not exact.
\end{example}

\begin{example}
Charts of weighted blow-ups along regular embeddings are instances of Olsson cones. Weighted blow-ups in turn generalise both ordinary blow-ups and root stacks. Let $S$ be the spectrum of a regular, Noetherian, local ring, and let $f_0,\ldots, f_n$ be a regular sequence. The morphism $\NN^{n+1}\to\NN^{n+1}$ mapping $e_0\mapsto d_0e_0'$ and $e_i\mapsto e_i'+d_ie_0'$ for $i=1,\ldots,n$ makes $S\times_{\Acal^{n+1}}(\Acal^{n+1})'$ into a standard chart for the weighted blow-up of $S$ in the weighted regular centre $(f_0,d_0)+\ldots+(f_n,d_n)$ (see \cite[\S 5.4]{QuekRydh}). 
\end{example}

		\begin{example}\label{exa:relative_edge}
			The simplest instance, and main building block, of tropical curves is an edge $e$ of finite length $\delta$.
			Let $\Acal = \Acal_{\mathbf N}$ and let $\mu$ be the multiplication map $\Acal^2\to \Acal$. Let $S=\Spec(R)\to\Acal$ be the rank one logarithmic structure associated with an element $s$ of $R$.  According to \Cref{cor:14}, $S_e$ is the quotient stack $\tilde S / \Gm^2$, where $\tilde S$ is the spectrum of \[C= R[x_1,x_2,u^{\pm}]/(x_1x_2-u^{-1}s),\]and $\Gm^2$ acts on $x_1$, $x_2$, and $u$ with weights
			 \[\begin{pmatrix}
				1 & 0 & -1 \\
				0 & 1 & -1
			   \end{pmatrix}\]
			Corollary~\ref{cor:1} says that we can slice at $u=1$ to obtain an equivalent quotient.  We take $\tilde S' = \Spec(C')$ with 
			\begin{equation*}
				C' = R[x_1, x_2 ] / (x_1 x_2 - s)
			\end{equation*}
			and then $S_e = \tilde S' / \Gm$, with $\Gm$ acting via the antidiagonal of the previous action (that is, with weights $(1,-1)$ on $x_1$ and $x_2$). This is the stack $ST_R$ appearing in the definition of $S$-completeness \cite{AHLH2023}.

			The universal situation for the preceding discussion is $R = \mathbf Z[s]$.  In that case, we can eliminate the  variable $s$ as well.  The quotient of $\tilde S'$ by $\Gm$ may then be imagined as $\Aaff^1$ with the origin replaced by a stack with three points, corresponding to the three $\Gm$-orbits contained in the coordinate axes of $\Aaff^2$.  The general case is derived by base change from this one.

			In general, the $\Gm^2$-action on $\Spec C$ corresponds to the morphism
			\begin{align*}
			 C &\to C\otimes_R R[\lambda_1^{\pm},\lambda_2^{\pm}] \\
			 (x_1,x_2,u) &\mapsto (x_1\otimes\lambda_1, x_2\otimes\lambda_2, u\otimes\lambda_1^{-1}\lambda_2^{-1}) .
			\end{align*}
			At the level of underlying algebraic stacks (that is, not respecting logarithmic structures), a section of $S_e\to S$ corresponds to a $\Gm^2$-torsor $P\to S$ together with a $\Gm^2$-equivariant morphism $f\colon P\to\tilde{S}$. Dually, assuming the torsor to be trivial, this corresponds to a commutative diagram:
			\bcd \label{eqn:14}
			C\ar[r,"f^{\sharp}"]\ar[d] & R[\mu_1^{\pm},\mu_2^{\pm}]\ar[d]\\
			C\otimes_R  R[\lambda_1^{\pm},\lambda_2^{\pm}] \ar[r,"f^{\sharp}\otimes\on{id}"] & R[\mu_1^{\pm},\mu_2^{\pm}]\otimes_R  R[\lambda_1^{\pm},\lambda_2^{\pm}]
			\ecd
			Unsurprisingly, equivariance forces
			\[x_1\mapsto a_1\mu_1,\ x_2\mapsto a_2\mu_2,\ u\mapsto b\mu_1^{-1}\mu_2^{-1}\]
			with $a_1, a_2 \in R$ and $b\in R^\ast$ and $a_1a_2=b^{-1}s$.  Two diagrams~\eqref{eqn:14} are equivalent if they are related by the $\Gm^2$ action on $P$, which scales $a_1$ and $a_2$ by units.  In the universal situation, there are two sections up to equivalence, $(a_1=1,a_2=s,b=1)$ and $(a_1=s,a_2=1,b=1)$.  When $s=0$, the underlying stack of $S_e$ has one more section (up to nonunique isomorphism) given by $(a_1=a_2=0,b=1)$.

			Let $\delta$ denote the tropical parameter on $\Acal_{\mathbf N}$, and $\gamma$ the coordinate on $e$. Then $\mu^*(\delta)=(1,1)$ and $\gamma=(0,1)$ in $\NN^2=M_{\Acal_{\mathbf N}^2}$. The line bundle $\OO_{S_e}(\gamma)$ corresponds to the shift $x_2^{-1}C$.  The sections of $\OO_{S_e}(\gamma)$ correspond to the invariants of the $\Gm^2$-action, or equivalently, to the degree~$0$ part of $x_2^{-1} C$, and may be identified with $R = x_2^{-1} ( x_2 R )$.  The ray $\gamma=0$ corresponds to the section $x_1=u^{-1}s,x_2=1$ as above; the restriction of $\OO_{S_e}(\gamma)$ along this section is isomorphic to $\OO_S$ as expected. The ray $\gamma=\delta$ corresponds instead to the section $x_1=1,x_2=y^{-1}s$; here we find that the restriction of $\OO_{S_e}(\gamma)$ is isomorphic to $\OO_S(\delta)$.
            \end{example}

\begin{lemma}\label{lem:univSurj}
    An Olsson precone $\Theta_R\to S$ is universally surjective if and only if $\bar M_S \to R$ is exact.
\end{lemma}
\begin{proof}
    Locally on $S$, we may assume that $S$ is strictly henselian, and write $\Theta_R=S\times_{\Acal_P}\Acal_Q$ for $P=\overline{M}_{S,s}$. In particular, the map $S\to \Acal_P$ hits the closed point of the latter. Then the lemma follows from {\cite[A.1, Proposition]{Nakayama1}}, which we repeat below for the reader's convenience.
\end{proof}
\begin{proposition}[{\cite{Nakayama1}}]\label{Nakayama-exact}
	If $f : x \to y$ is a surjective morphism of fine, saturated logarithmic schemes whose underlying schemes are punctual then $f$ is universally surjective if and only if $f^\ast \bar M_y \to \bar M_x$ is exact.
\end{proposition}
From now on, we will use the two terms `exact' and `universally surjective' interchangeably for Olsson precones.

\begin{definition}
   We call $\Theta_R$ an \emph{Olsson cone} if $\bar{M}_S\to R$ is saturated.
\end{definition}

\begin{remark}			
The reason for calling  $\Theta_R$ an Olsson cone only when saturated is motivated by our structural result about Olsson fans, Theorem~\ref{thm:localstructureOlssonfans}.
\end{remark}

\subsection{Olsson fans}

		\begin{definition}\label{def:Olssonfan}
		    Let $S$ be a fine and saturated logarithmic scheme.  An \emph{Olsson fan}  over $S$ is an algebraic stack  $\Sigma\to S$ with a fine, saturated logarithmic structure that is saturated and logarithmically \'etale over $S$.
		\end{definition}

			The purpose of this section is to show that Olsson fans are glued from Olsson cones along strict \'etale relations.

		\begin{proposition} \label{prop:18}
			Let $S$ be the spectrum of a strictly henselian local ring.  Let $\mathcal A_Q \to \mathcal A_P$ be a morphism of Artin cones, dual to a local homomorphism of fine, saturated monoids $P \to Q$, and let $\sigma \to S$ be its base change to $S$ along a map $S \to \Acal_P$ that meets the closed stratum.  Then $\sigma$ has a unique closed stratum, $Z$. Assume that $E$ is an algebraic stack that is \'etale and quasiseparated over $\sigma$.  Then every section of $E$ over $Z$ extends uniquely to a section of $E$ over $\sigma$. 
		\end{proposition}

		\begin{proof}
			The unique closed stratum $Z$ is the preimage of the unique closed stratum of $\mathcal A_Q$.  Let $a$ be a section of $E$ over $Z$.  We will extend this to a section over $\sigma$.  We may reduce immediately to the case where $S$ is noetherian by factoring $S \to \mathcal A_P$ through the henselization of a scheme of finite type.

We shall use the notation of \Cref{cor:1}.
            Choose a factorization of $S \to \mathcal A_P$ through $X_P$. We may identify $\sigma$ with $T_K \times [U / T_L]$ where $U = \Spec ( \mathcal O_S \otimes_P Q )$.  Let $\sigma_0$ be the fiber of $\sigma$ over the closed point of $S$.  Then $\sigma_0$ is identified with $T_K \times [U_0/ T_L]$ where $U_0 = \Spec( k[Q] / P^+ k[Q])$, writing $k$ for the residue field of $S$ at its closed point and $P^+$ for the maximal ideal of $P$.  
            The vanishing locus of the maximal ideal $Q^+$ of $Q$ in $U_0$ is a point $z$ of $U$ with residue field $k$.  Since $z$ is fixed by $T_L$, we may identify $Z$ with $T_K \times [z/T_L] = T_K \times \mathrm BT_L$.  The composition 
			\begin{equation*}
				\mathcal A_Q \to \mathrm B T_Q \to \mathcal A_Q
			\end{equation*}
			gives a retraction of $\mathcal A_Q$ onto its closed stratum.  This restricts to a retraction of $\sigma_0$ onto $Z$.
			Composing with this retraction, we extend $a$ to a section $a_0$ of $E$ over $\sigma_0$.

			Let $V = T_K \times U$ so that $\sigma = [V/T_L]$.  Let $I$ be the maximal ideal of $\mathcal O_S$.  For each positive integer $n$, let $S_n$ be the vanishing locus of $I^{n+1}$ in $S$ and let $\sigma_n \subset \sigma$ and $V_n \subset V$ be its preimages.  Since $E$ is \'etale, $a_0$ extends uniquely to a section $a_n$ of $E$ over $\sigma_n$.  Let $\hat V$ be the completion of $V$ along $V_0$.  Then $E \mathop\times_{\sigma} \hat V$ is \'etale and of finite presentation over $\hat V$.  In particular, it is noetherian and has quasifinite diagonal (since we assumed $E$ \'etale and quasi-separated over $\sigma$). Therefore \cite[Corollary~1.5]{HallRydh2019} implies that the system morphisms $\left(a_n \big|_{V_n}\right)$ extends to a morphism $\hat b : \hat V \to E$.  Since $\hat V$ is affine, a second application of \cite[Corollary~1.5]{HallRydh2019} implies that this extension is a section of $E$ over $\hat V$.
			
			Now we apply Artin's approximation theorem.  Let $\pi : V \to S$ be the projection and let $F = \pi_\ast E$.  Since $V$ is quasicompact and quasiseparated, and since $E$ is locally of finite presentation over $\sigma$, it follows that $F$ is also locally of finite presentation~\cite[Proposition~(8.13.1)]{ega4-3} over $S$.  By Artin's approximation theorem~\cite[Theorem~(1.12)]{Artin1969}, there is a $\tilde b \in F(S) = E(V)$ whose restriction to $S_0$ coincides with the restriction of $\hat b$.  In particular, the restriction of $\tilde b$ to $Z \mathop\times_\sigma V$ coincides with the restriction of $a$.

			Since $V$ is affine over $\sigma$, so are $V \mathop\times_\sigma V$ and $V \mathop\times_\sigma V \mathop\times_\sigma V$.  We may therefore repeat the same reasoning and obtain unique maps $V \mathop\times_\sigma V \to E$ and $V \mathop\times_\sigma V \mathop\times_\sigma V \to E$ extending the restriction of $a$.  Together, these constitute a descent datum for sections of $E$ over $\sigma$.  We conclude by faithfully flat descent that $b$ descends uniquely to a section $\tilde a$ of $E$ over $\sigma$ that restricts to $a$ on $Z$.

			To conclude the proof, we show that $\tilde a$ is unique up to unique isomorphism.  To this end, let $E'$ be the sheaf of isomorphisms between two extensions, $\tilde a$ and $\tilde a'$ of $a$ to $E$.  Then the agreement of $\tilde a \big|_Z$ and $\tilde a' \big|_Z$ gives a section of $E'$ over $Z$.  Applying the above reasoning to $E'$ in place of $E$, we obtain an isomorphism between $\tilde a$ and $\tilde a'$ over $\sigma$.  A third application shows that the isomorphism between $\tilde a$ and $\tilde a'$ is unique.
		\end{proof}

		\begin{remark} \label{rem:2}
			We conjecture that \Cref{prop:18} is true without the quasisepartation hypothesis.  If this conjecture is false, quasiseparation or some other condition should probably be added to the definition of an Olsson fan.
		\end{remark}

		\begin{corollary} \label{cor:etale_sheaves}
			Let $\pi : \sigma \to S$ be an Olsson precone over a fine, saturated logarithmic scheme $S$.  Let $F$ be an \'etale sheaf of abelian groups on $\sigma$.  If $\pi$ is saturated or if $F$ is torsion-free then $\mathrm R^1 \pi_\ast F = 0$.
		\end{corollary}
		\begin{proof}
			We assume without loss of generality that $S$ is local and strictly henselian (which is justified because $\sigma$ is quasicompact and quasiseparated over $S$).  Let $Z$ be the closed stratum of the closed fiber of $\sigma$ over $S$.  
            Since $F$ is an \'etale sheaf on $\sigma$, to give a section, it is equivalent to give a section over $Z$ by Proposition~\ref{prop:18}.  But, in the notation of Proposition~\ref{prop:18}, $Z \simeq T_K\times \mathrm BT_L$ for a diagonalizable group $T_L$.  Let $U_\bullet$ be the groupoid associated with the cover of $Z$ by $U_0 = T_K\times \Spec k$, where $k$ is the residue field of $S$ at its closed point.  Then $U_n = T_K\times T_L^n$ for all $n$.  The \v Cech complex of $F$ over $U_\bullet$ is $C^n = \Gamma(T_K\times T_L^n, F_0)$ where $F_0 = F \big|_{U_0}$.  The coboundary condition identifies $H^1(U_\bullet, F)$ with  morphisms $\varphi\colon T_K\times T_L \to F_0$ such that $\varphi(z,xy)-\varphi(z,x)-\varphi(z,y)=0$ for each $z,x,y$.  That is, $\varphi(z,-)\colon T_L\to F_0\rvert_{z}$ is a homomorphism for each $z\in T_K$.  If $\pi$ is saturated then $T_L$ is an algebraic torus, hence connected, and $F_0\rvert_{z}$ is discrete, so every homomorphism $T_L \to F_0\rvert_{z}$ vanishes.  In general, $T_L$ is diagonalizable, so its component group is finite so every homomorphism $T_L \to F_0\rvert_{z}$ will vanish if $F_0\rvert_{z}$ is torsion-free.
		\end{proof}

		\begin{corollary} \label{cor:12}
			Let $\pi : \sigma \to S$ be an Olsson cone over a fine, saturated logarithmic scheme $S$.  Let $F$ be an \'etale sheaf of abelian groups \'etale on $\sigma$.  Then $\mathrm R^p \pi_\ast F = 0$ for all $p > 0$.
		\end{corollary}
		\begin{proof}
			We saw already in Corollary~\ref{cor:etale_sheaves} that $\mathrm R^1 \pi_\ast F = 0$ for all sheaves of abelian groups $F$ on $\sigma$.  Choose an injective homomorphism $F \to I$ where $I$ is injective.  Then we may identify $\mathrm R^p \pi_\ast F$ with $\mathrm R^{p-1} \pi_\ast (I/F)$ for all $p \geq 2$.  We conclude by induction.
		\end{proof}

			\begin{theorem}\label{thm:localstructureOlssonfans}
			   Let $S$ be a fine and saturated logarithmic scheme, and let  $\Sigma\to S$ be a quasiseparated Olsson fan. Then $\Sigma$ admits a strict \'etale cover by Olsson cones over $S$.
			\end{theorem}

		\begin{proof}
			Let $x$ be a geometric point of $\Sigma$ lying over a geometric point $s$ of $S$.  We argue that there is an Olsson cone $\sigma$ over $S$ and a strict \'etale $S$-morphism $\sigma \to \Sigma$ through which $x$ factors.  After replacing $S$ by a strict \'etale neighborhood of $s$, we can assume that the homomorphism $\Gamma(S, \bar M_S) \to \bar M_{S,s}$ has a section.  Writing $P = \bar M_{S,s}$, this gives us a map $S \to \Acal_P$.  Let $Q = \bar M_{\Sigma,x}$.  Then we set
			\begin{equation*}
				\sigma = \Acal_Q \mathop\times_{\Acal_P} S .
			\end{equation*}
            
			We need to lift the following diagram, with $\sigma \to \Sigma$ strict, at least after \'etale localization around $s$ in $S$~:
			\begin{equation} \label{eqn:4} \begin{tikzcd}
				x \ar[r] \ar[d] & \Sigma \ar[d] \\
				\sigma \ar[r] \ar[ur,dashed] & S
			\end{tikzcd}
			\end{equation}

			Let $\Log_S$ be Olsson's stack of fine, saturated logarithmic structures~\cite{olssonlogarithmic} (see \Cref{thm:7}).   To lift~\eqref{eqn:4}, it is equivalent to lift the following diagram of (non-logarithmic) schemes and algebraic stacks~:
			\begin{equation*} \begin{tikzcd}
				x \ar[r] \ar[d] & \Sigma \ar[d] \\
				\sigma \ar[r] \ar[ur,dashed] & \Log_S
			\end{tikzcd}
			\end{equation*}
			Since $\Sigma$ is \'etale over $\Log_S$ by assumption, this is equivalent to finding a section of the \'etale algebraic stack $\Sigma \mathop\times_{\Log_S} \sigma$ over $\sigma$.  Let $\tilde S$ be the henselization of $S$ at $s$ and let $\tilde \sigma$ be the base change of $\sigma$ to $\tilde S$. By hypothesis $P\to Q$ is integral (because saturated) and local (because it is induced by a morphism of logarithmic structures) so exact by \Cref{rmk:integral-exact}, which in turn implies injective since $P$ is sharp.  We may thus identify the closed stratum of $\sigma$ with $\mathrm BT$, where $T=T_L$ is the diagonalizable group dual to $L=Q^{\rm gp} / P^{\rm gp}$.  Since $\Sigma$ is saturated, $Q^{\rm gp} / P^{\rm gp}$ is torsion-free, so $T$ is in fact an algebraic torus.  Therefore $T$ is connected, so every section of an \'etale sheaf over $x$ is $T$-equivariant.  Therefore the section $x \to \Sigma$ descends to a section over the closed stratum of $\tilde\sigma$.  By Proposition~\ref{prop:18}, this section extends uniquely to a section over $\tilde\sigma$.  Finally, $\sigma$ is quasicompact and quasiseparated over $S$, so this section extends to an \'etale neighborhood of $s$ in $S$~\cite[Proposition~(18.13.1)]{ega4-3}.%
			\footnote{Technically, the cited proposition concerns only sections of \'etale sheaves over schemes, and $\sigma$ is an algebraic stack.  However, the proposition can be applied to a groupoid presentation of $\sigma$ to obtain the required conclusion.}
		\end{proof}

        \subsection{Weakly convex cones}
    \begin{definition}
     Denote by $\logGm$ and $\tropGm$ the group objects over fine saturated schemes whose values at  $S$ are
		\begin{equation*}
			\logGm(S)=\Gamma(S,  M_S^{\rm gp})	\qquad\text{and}\qquad \tropGm(S) = \Gamma(S, \bar M_S^{\rm gp}) .
		\end{equation*}
		Then a \emph{tropical torus}  $T$ over $S$ is a sheaf of groups in the strict \'etale topology of logarithmic schemes over $S$ that, locally in the strict \'etale topology of $S$, is isomorphic to $\tropGm^r \times S$, for some integer $r$.

				The sheaf $X=\uHom_S(T,\Gtr{S})$ of homomorphisms $T \to \tropGm$, on the strict-\'etale site of $S$, is called the \emph{character group} of $T$.  
    \end{definition}
	By \cite[Proposition~A.2]{wise2023monodromy}, the character group of $X$ is, locally in the strict \'etale topology of $S$, a constant sheaf of finitely generated free abelian groups.  The same result shows that the tropical torus $T$ is determined by its character group~: $T = \uHom_S(X,\Gtr{S})$.



\begin{definition}\label{notation:fine_po_groups}
    Let $\amb{R}$ denote a partially ordered abelian group whose underlying group is finitely generated.  Notice that $\amb{R}$ is automatically free.\footnote{This would not be the case if the group were only preordered.} Let $R=(\amb{R})_{\geq 0}$ denote the submonoid of non-negative elements. If $R$ is finitely generated as a monoid then we will call $\amb R$ a \emph{fine partially ordered group}. The category of fine, sharp monoids can be embedded in that of fine partially ordered groups~; notice, however, that $R^{\rm{gp}}$ may be strictly contained in $\amb{R}$ in general. We shall say that $\amb{R}$ is saturated if $R$ is. 
\end{definition}

		\begin{definition}\label{def:polyhedraldomain}
			Let $S$ be a fine and saturated logarithmic scheme and let $\amb{R}$ be a sheaf of fine  and saturated partially ordered groups on $S$, equipped with a homomorphism from $\bar M_S^{\rm gp}$. 

			We define $\Theta_{\amb{R}}$ to be the 
            sheaf on fine, saturated logarithmic schemes over $S$ whose value on $f : T \to S$ is the set of homomorphisms of partially ordered groups $f^{-1} \amb{R} \to \bar M_T^{\rm gp}$ that restrict to the structural morphism $f^{-1} \bar M_S \to \bar M_T$.  
    
				In other words, a $T$ point of $\Theta_{\amb{R}}$ is a commutative square~:
        \begin{equation*} \begin{tikzcd}
					f^{-1}\bar{M}_S\ar[r] & f^{-1}R\ar[hookrightarrow]{d} \ar[r] & \bar{M}_T \ar[hookrightarrow]{d}  \\
				& f^{-1}\amb{R} \ar[r]  & \bar{M}^{\rm gp}_T 
			\end{tikzcd}
			\end{equation*}
				The sheaf $\Theta_{\amb{R}}$ is called  a \emph{weakly convex Olsson precone}, and a weakly convex Olsson cone if $\bar M_S^{\rm gp}\to \amb{R}$ is saturated.
    \end{definition}
    \begin{proposition}\label{prop:polyhedral domain}
    Let $\Theta_{\amb{R}}$ be a weakly convex Olsson precone. 
    The natural  morphism \[\Theta_{\amb{R}}\to\Theta_{R}\]
    exhibits   $\Theta_{\amb{R}}$ as a torsor over the Olsson cone $\Theta_{R}$ under the tropical torus with character group $\amb{R} / R^{\rm gp}$.  This torsor is locally trivial in the strict \'etale topology of~$S$.
    \end{proposition}
    \begin{proof}
    Given a fine and saturated logarithmic scheme $f\colon S'\to S$ and a point 
			$x\in \Theta_{R}(S'),$  namely a monoid homorphism $x\colon f^{-1}R\to\bar{M}_{S'}$ compatible with the map from $f^{-1}\bar{M}_S$, we  obtain an associated  homomorphism  of groups $x^{\rm gp}\colon f^{-1}R^{\rm gp}\to\bar{M}^{\rm gp}_{S'}.$
   To get $\tilde x\in \Theta_{\amb{R}}(S')$ amounts to the choice of an extension  $\tilde{x}\in\Hom(f^{-1} \amb{R},\bar{M}_{S'}^{\rm gp})$ that restricts to $x^{\rm gp}$ along the inclusion $f^{-1}R^{\rm gp}\to f^{-1} \amb{R}$.
			The set of such lifts is a torsor under $\Hom(\amb{R} \slash R^{\rm gp},\bar{M}^{\rm gp}_{S'})$.  This is the group of $S'$-points of the tropical torus with character group $\amb{R} / R^{\rm gp}$, we may conclude that $\Theta_{\amb R}$ is a torsor over $\Theta_{R^{\rm gp}}$ under the tropical torus with character group $\amb R / R^{\rm gp}$.

			Furthermore, the lifts of $S' \to \Theta_R$ to $S' \to \Theta_{\amb{R}}$ form a torsor, on the strict \'etale site of $S'$, under the group $\Hom(\amb{R} / R^{\rm gp}, \bar M_{S'}^{\rm gp})$.  The espace \'etal\'e of this torsor is an algebraic stack with an \'etale morphism to $\Theta_R$ that is representable by algebraic spaces.  By Corollary~\ref{cor:etale_sheaves}, this torsor can be trivialized locally in the strict \'etale topology of $S$.
   \end{proof}
   \begin{definition}
        The tropical torus  $\Hom_S(\amb{R} \slash R^{\rm gp},\Gtr{S})$ is called the  \emph{lineality space}, and the tropical torus dual to $\amb{R}$ 
        is called the \emph{span} of the weakly convex cone $\Theta_{\amb{R}}$.
    \end{definition}
    \begin{remark}
    Notice that both Olsson cones  and tropical tori over $S$  are particular cases of weakly convex cones over  $S$.
			Indeed, if $R^{\rm gp}=\amb{R}$, it follows immediately from the definitions that $\Theta_R \simeq \Theta_{\amb{R}}$ as presheaves over fine and saturated logarithmic schemes over $S$.
			In the other extreme case, if $R=\bar{M}_S$ and $\amb{R} = R^{\rm gp} \oplus X$ for a sheaf of finitely generated, free abelian groups over $\bar{M}_S$, then it follows from \cite[Proposition~A.2]{wise2023monodromy} that $\Theta_{\amb{R}}$ is the tropical torus $T$ over $S$ with character lattice $X$.
    \end{remark}

\begin{definition}
	Let $\varphi : \amb P \to \amb Q$ be a homomorphism of sheaves of fine partially ordered abelian groups and let $P \to Q$ be the associated homomophism of submonoids of elements $\geq 0$.  We say that $\varphi$ is a \emph{localization} if its kernel is generated by elements in $P$.  In this case, we say that the dual morphism of weakly convex Olsson cones is a \emph{face morphism}.
\end{definition}

\begin{definition}
  A  \emph{weakly convex Olsson fan} or \emph{weak cone complex} $\Sigma$ over $S$ is a stack in the strict \'etale topology on logarithmic schemes over $S$ obtained as the colimit of finitely many weakly convex Olsson cones along face  morphisms. 
\end{definition}

\subsection{The star of a subcone}
We turn to a local characterization of the star construction from toric geometry, based on a monoidal analogue of the Zariski tangent space from algebraic geometry. We start with an analogue of the ring of dual numbers. Recall the notation of \Cref{notation:fine_po_groups}.

\begin{definition}
    Let $\amb{M}$ and $\amb{N}$ be fine, saturated, partially ordered groups.  The \emph{infinitesimal extension} of $\amb{M}$ by $\amb{N}$ is the fine, saturated, partially ordered group $\amb{M[\varepsilon N]}$
with $m + \varepsilon n \geq 0$ if $m > 0$ or $m = 0$ and $n \geq 0$.
\end{definition} For instance, $\NN[\varepsilon\NN]$ is the standard valuative monoid of rank two. Notice that $\amb{M[\varepsilon N]}$ comes with both an inclusion of $\amb{M}$ (making it a partially ordered group over a base, if $\amb{M}$ is) and a projection to $\amb{M}$. We think of the dual cone of $\amb{M[\varepsilon N]}$ as an infinitesimal thickening of that of $\amb{M}$.

\begin{definition}
	Let $\sigma$ be a covariant functor on fine, saturated monoids (or on partially ordered abelian groups).  Suppose that $x \in \sigma(M)$ for some fine, saturated monoid $M$.  We write $\Star_x(\sigma)$ for the functor on fine, saturated monoids (or on partially ordered abelian groups) that sends $N$ to the set of lifts of $x$ to $\sigma(M[\varepsilon N])$ and call this the \emph{star} of $x$.  That is~:
	\begin{equation}\label{eq:starsigma}
		\Star_x(\sigma)(N) = \sigma(M[\varepsilon N]) \mathop\times_{\sigma(M)} \{ x \}
	\end{equation}
\end{definition}

\begin{remark}
    Given a morphism $f\colon\sigma\to\sigma'$ and a point $x$ of $\sigma$, we obtain a morphism \[\Star_x(f)\colon \Star_x(\sigma) \to \Star_{f(x)}({\sigma'}).\]
\end{remark}

\begin{lemma}
    The star construction commutes with arbitrary limits and with filtered colimits in $\sigma$.
\end{lemma}
\begin{proof}
    This follows from a general statement about the commutativity of limits and colimits \cite[\S IX.2]{MacLane}.
\end{proof}

\begin{lemma}\label{lem:star_cone}
    The star of a (weakly convex) cone is again a (weakly convex) cone.
\end{lemma}
\begin{proof}
    If $\sigma$ is representable by a fine partially ordered abelian group $\amb{R}$ by the formula
\begin{equation*}
	\sigma(M) = \Hom(\amb{R}, M^{\rm gp})
\end{equation*}
then, unraveling \eqref{eq:starsigma}, we get that $\Star_x({\sigma})$  is also representable by a partially ordered abelian group $\amb{Q}$ with the same underlying group as $\amb{R}$, but with $a \leq b$ in $\amb{Q}$ if $a \leq b$ in $R$ and $x(a) = x(b)$.
\end{proof}
\begin{remark}
    Given $(S,M_S)$ a fine and logarithmic scheme with a chart $\varphi\colon P\to \mathcal O_S$ we will denote by $\Star_x(S)$ the log scheme $(S,M'_S)$ where $M'_S$ is the log structure associated to $Q\hookrightarrow P\xrightarrow{\varphi}\mathcal O_S$ and $Q$ is the monoid described in the proof of \Cref{lem:star_cone}.
\end{remark}

\begin{definition}
    Let $\sigma$ be the covariant functor on fine, saturated monoids represented by a saturated partially ordered abelian group $\amb{R}$.  A \emph{subdivision} of $\sigma$ is a subfunctor $\Sigma$ of $\sigma$ such that there is a finite family of subfunctors $\rho_i \subset \Sigma$ such that
\begin{enumerate}
	\item \label{it:2} each $\rho_i$ is representable by a finitely generated refinement of the partial order on $\amb{R}$, 
	\item \label{it:16} for every $i$ and $j$, the intersection $\rho_i \cap \rho_j$ is a face of each, and
	\item \label{it:3} if $M$ is a valuative monoid then $\Sigma(M) \to \sigma(M)$ is a bijection.
\end{enumerate}
\end{definition}

\begin{lemma} \label{lem:1}
	Let $\sigma$ be a fine weakly convex Olsson precone. 
    Let $\Sigma$ be a subdivision of $\sigma$.  Let $x$ be an $M$-point of $\Sigma$ ; we shall denote the induced point of $\sigma$ by the same letter. 
    Then $\Star_x(\Sigma)$ is a subdivision of $\Star_x(\sigma)$.
\end{lemma}
\begin{proof}
	Let $\amb{R}$ be the finitely generated partially ordered abelian group corresponding to $\sigma$, and let $R$ the submonoid of elements $\geq 0$.  Suppose that $\Sigma = \bigcup \rho_i$ where each $\rho_i$ is representable by a finitely generated refinement $\amb R_i$ of the partial order of $\amb{R}$.  Let $x : \amb{R} \to M^{\rm gp}$ be the given point of $\sigma$.  Let $\amb Q$ be the partially ordered group representing $\Star_x(\sigma)$ given in Lemma~\ref{lem:star_cone}. In particular we saw that the underlying group of $\amb Q$ is the same as that of $\amb{R}$, and the submonoid $Q$ of elements $\geq 0$ is $R \cap \ker(x)$.  Since $R$ is the intersection of a rational polyhedral cone in $\amb{R} \otimes \mathbf R$ with $\amb{R}$, and since $\ker(x)$ is a rational subspace of $\amb R$, it follows that $Q = R \cap \ker(x)$ is also the set of integral points of a rational polyhedral cone, hence is finitely generated as a monoid.  The same reasoning applies to each of the $R_i$, and $R_i \cap \ker(x)$ is then a finitely generated extension of $R \cap \ker(x)$.  This shows that
	\begin{equation*}
		\Star_x(\Sigma) = \bigcup_{x \in \rho_i} \Star_x({\rho_i})
	\end{equation*}
	is the union of finitely many functors that are representable by finitely generated extensions of the partially ordered abelian group representing $\Star_x(\sigma)$.

	Next, we show that $\Star_x({\rho_i}) \cap \Star_x({\rho_j})$ is a face of each of $\Star_x({\rho_i})$ and $\Star_x({\rho_j})$.  To have $\Star_x({\rho_i})$ and $\Star_x({\rho_j})$ both be nonempty, it is necessary that $x$ factor through $\rho_i \cap \rho_j$, which is a face of both $\rho_i$ and $\rho_j$, by assumption.  That is, there is some $f \in \amb R$ that is contained in both $R_i$ and $-R_j$ such that $\rho_i \cap \rho_j = \rho_i \cap \{ f = 0 \} = \rho_j \cap \{ f = 0 \}$.  In other words, $R_i+R_j=R_i[-f]=R_j[f]$, and $x^\ast f = 0$.  Thus,
	\begin{equation*} \begin{aligned}
		( R_i \cap \ker(x) ) + ( R_j \cap \ker(x) )
		& \supset ( R_i \cap \ker(x) )[-f] \\
		& = R_i[-f] \cap \ker(x) \\
		& = ( R_i + R_j ) \cap \ker(x) \\
		& \supset ( R_i \cap \ker(x) ) + ( R_j \cap \ker(x) )
	\end{aligned}
	\end{equation*}
	so
	\begin{equation*}
		( R_i \cap \ker(x) ) + ( R_j  \cap \ker(x) )
		= ( R_i \cap \ker(x) )[-f] .
	\end{equation*}
	Since $( R_i \cap \ker(x) ) + ( R_j \cap \ker(x) )$ represents $\Star_x({\rho_i}) \cap \Star_x({\rho_j})$, and since $f \in R_i \cap \ker(x)$, this shows that $\Star_x({\rho_i}) \cap \Star_x({\rho_j})$ is a face of $\Star_x({\rho_i})$.  By symmetry, it is also a face of $\Star_x({\rho_j})$.

	To conclude the proof, we show that every valuative point of $\Star_x(\sigma)$ lifts to $\Star_x(\Sigma)$.  Suppose that $\amb{R}\to M[\varepsilon N]$ is a valuative point of $\Star_x\sigma)$ (so $N$ is a valuative monoid).  Choose a local valuation $M \to V$.  Then the diagram
	\begin{equation} \label{eqn:3} \begin{tikzcd}
	M[\varepsilon N] \ar[r] \ar[d] & V[\varepsilon N] \ar[d] \\
		M \ar[r] & V
	\end{tikzcd}
	\end{equation}
	is cartesian and $V[\varepsilon N]$ is valuative.  Let $y$ be the $V$-point of $\sigma$ induced by $M \to V$.  Then any $N$-valued point of $\Star_x(\sigma)$ induces a point of $\Star_y(\sigma)$ by composition with the upper horizontal arrow in the diagram~; this is equivalently a $V[\varepsilon N]$-valued point of $\sigma$, and this lifts uniquely to a $V[\varepsilon N]$-valued point of $\Sigma$ by hypothesis.  But then since~\eqref{eqn:3} is cartesian, this is induced from a uniquely determined point of $\Star_x(\Sigma)$ that projects as required to $\Star_x(\sigma)$.
\end{proof}

 \begin{lemma}\label{lem:exact_stars}
     Let $f\colon\sigma\to\sigma'$ be an integral morphism of weakly convex Olsson precones, and let $x\colon\tau\to\sigma$ a morphism from an Olsson precone. Then $f$ is universally surjective if and only if $\on{Star}_x(f)$ is.
 \end{lemma}
 \begin{proof}
     Dually, let us start from morphisms of partially ordered abelian groups:
     \[\amb{R}\xrightarrow{F}\amb{(R')}\xrightarrow{X}M^{\rm gp}.\]
     Recall from \Cref{lem:star_cone} that $\on{Star}_x({\sigma})$ is represented by the monoid $Q=R\cap (X\circ F)^{-1}(0)$ (within the same partially ordered group $\amb{R}$), and similarly for $Q'$.

     First, assume that $R\to R'$ is exact. Then
     \[ F^{-1}(Q') = F^{-1}(X^{-1}(0) \cap R') = (X \circ F)^{-1}(0) \cap F^{-1}(R') = (X \circ F)^{-1}(0) \cap R = Q, \]
     so $Q\to Q'$ is exact as well.

     Conversely, assume that $Q\to Q'$ is exact. Since $R\to R'$ is integral, it is enough to prove that it is local in order for it to be exact (\Cref{rmk:integral-exact}). So let $\alpha\in R$ be such that $F(\alpha)=0$. But then $\alpha$ is in $Q$, so since $Q\to Q'$ is local and $Q$ is sharp we conclude that $\alpha=0$.
 \end{proof}
 
\section{On toric vector bundles}\label{sec:overview}
\tocdesc{reviews results of Klyachko, Perling, and Payne on toric vector bundles and their moduli}

We review some aspects of the theory of torus-equivariant sheaves on toric varieties, which are equivalently quasicoherent sheaves on Olsson cones (or, more generally, subdivisions of weakly convex cones) over a trivial logarithmic point.
This topic has drawn attention in the past fifty years \cite{Kaneyama,Klyachko,Perling,Payne,PayneBranched,DRJS,KavehManon,kaveh2025toricvectorbundlesdiscrete}, including connections with physics \cite{FLTZ}.  We summarize some of these results in this section before generalizing them to Olsson cones in the next.


\subsection{Equivariant sheaves as graded modules} \label{sec:2}

We start from the affine case. Fix a base field $k$. Let $P$ be a fine, saturated, sharp monoid, let $\sigma=\on{Hom}_{\on{Mon}}(P,\mathbf R_{\geq0})$ be the corresponding strictly convex rational polyhedral cone, and let $X_P$ (or $X_\sigma$) denote the associated affine toric variety $\on{Spec}(k[P])$, with dense torus $T=\Hom(P^{\mathrm gp},\Gm)$. Let $\mathcal E$ be a $T$-equivariant  quasicoherent sheaf on $X_P$, or equivalently a quasicoherent sheaf on the Artin cone $\Acal_P=[X_P/T]$ (or $\Acal_\sigma$).




A $T$-equivariant  quasicoherent sheaf on $X_P$ is the same as a $P^{\rm gp}$-graded $k[P]$-module.
Let $p\colon \Acal_P\to \mathrm BT$ be the natural structure morphism, and $\pi$ the composition of the latter with $\mathrm BT\to\Spec k.$ Then pushing forward along $p$ forgets the $k[P]$-module structure (Weil restriction), so $p_*\mathcal E$
 is a $P^{\rm gp}$-graded $k$-vector space. Pushing forward further to the point corresponds to taking the degree $0$ part of this graded vector space. We may therefore write the pieces of the isotypical decomposition as follows: for all $\delta\in P^{\rm gp}$ we have
\[E_{\delta}^\sigma \coloneq (p_*\mathcal E)_{\delta}=\pi_*(\mathcal E(\delta))=H^0(X_P,\mathcal E(\delta))^T.\] 
\begin{remark}
    If $\Ecal$ is coherent, then $E_{\delta}^\sigma$ is a vector space of finite dimension for every $\delta\in P^{\rm gp}$. So, despite being far from proper, $\Acal_P$ satisfies a form of the coherent pushforward property.
\end{remark}

For every $\delta_1\leq\delta_2$ in $P^{\rm gp}$ (meaning $\delta_2-\delta_1$ lies in the monoid $P$) the logarithmic structure induces a map of sheaves $\OO_{X_P}(\delta_1)\to \OO_{X_P}(\delta_2)$ between the corresponding line bundles, in turn inducing a linear map of $k$-vector spaces:
\[\chi^{\delta_2}_{\delta_1}\colon E^{\sigma}_{\delta_1}\to E^{\sigma}_{\delta_2}. \]
This encodes the structure of $k[P]$-module on $E = \bigoplus_{\delta} E^\sigma_\delta$. 
We now summarise the previous discussion in the notation of \cite{Perling}~:

\begin{definition}[{\cite[Definitions~5.1, 5.2, 5.4]{Perling}}] \label{def:4}
A $\sigma$-\emph{family} $E^{\sigma}$ or $P$-\emph{filtered vector space} is a collection $\left\{E^{\sigma}_{\delta}\right\}_{\delta\in P^{\rm gp}}$ of $k$-vector spaces, together with a collection of linear maps $\chi^{\delta_2}_{\delta_1}\colon E^{\sigma}_{\delta_1}\to E^{\sigma}_{\delta_2}$ whenever $\delta_1\leq\delta_2$ in $P^{\rm gp}$ in the partial ordered induced by $P$, subject to the compatibility condition $\chi^{\delta_3}_{\delta_2}\circ\chi^{\delta_2}_{\delta_1}=\chi^{\delta_3}_{\delta_1}$ for $\delta_1\leq\delta_2\leq\delta_3$ in $P^{\rm gp}$. We call the $\chi^{\delta}$ the \emph{multiplication maps}.

\end{definition}

\begin{proposition}[{\cite[Proposition~5.5]{Perling}}]\label{prop:equivariantsheaves}
The following categories are equivalent~:
\begin{itemize}
\item $T$-equivariant quasicoherent sheaves over $X_P$~;
\item $P^{\rm gp}$-graded $k[P]$ modules (morphisms are graded homomorphisms of degree $0$)~;
\item $\sigma$-families.
\end{itemize}
\end{proposition}

In order to treat general toric varieties, we need to discuss how quasicoherent sheaves pull back along torus-equivariant maps, keeping in mind the important special case of face morphisms (see \cite[Definition 5.6, Proposition 5.7]{Perling}).

\begin{proposition}\label{prop:pullback}
	Let $\mm\colon P\to Q$ be a morphism of fine, saturated, sharp monoids, and let $f\colon X_Q\to X_P$ denote the corresponding equivariant morphism of affine toric varieties. Let $\Ecal$ be a $T_P$-equivariant quasicoherent sheaf on $X_P$ and let $E$ be the corresponding graded $k[P]$-module. Up to isomorphism, the pullback equivariant quasicoherent sheaf $f^*\Ecal$ on $X_Q$ can be described as follows:
    \begin{itemize}
        \item As a $Q^{\rm gp}$-graded $k[Q]$-module~: it is the graded tensor product $E\otimes_{k[P]}k[Q]$.
        \item As a $\sigma_Q$-family~: the grading is
					\[(f^*E)_\delta=\underset{\mm(\gamma)\leq\delta}{\varinjlim} E_{\gamma}\]
        with the naturally induced multiplication maps.
    \end{itemize}
\end{proposition}
This applies in particular to localisations $P_\sigma=P\to P[-\delta]=P_\tau$ for a face $\tau\prec\sigma$ and the corresponding open embeddings of affine toric varieties $X_\tau\subset X_\sigma$, which Perling uses to characterize equivariant quasicoherence on general toric varieties \cite[Definition~5.8 and Theorem~5.9]{Perling}. Consider that in the setting of toric varieties it is natural to fix the character lattice of the torus $P^{\rm gp}$ (where $P$ is the monoid of linear functions on any maximal cone), but when passing to a face $\tau$ the multiplication maps corresponding to invertible elements become isomorphisms \cite[Lemma 5.3]{Perling}, hence we obtain an equivalent but combinatorially simpler $\tau$-family by passing to the sharpening $P[-\delta]^\sharp$.
\begin{definition}[{\cite[Definition~5.8]{Perling}}]
    Let $\Delta$ be the fan of a normal toric variety $X_{\Delta}$. Let $\iota_{\tau,\sigma}$ denote the open embedding $U_\tau\hookrightarrow U_\sigma$ for a face $\tau\prec\sigma$ in $\Delta$. A \emph{$\Delta$-family} if a collection of $\sigma$-families $\{ E^\sigma\}_{\sigma\in\Delta}$, together with isomorphisms $\eta_{\tau,\sigma}\colon \iota_{\tau,\sigma}^* E^\sigma\simeq  E^\tau$ for $\tau\prec\sigma$, satisfying the cocycle condition $\eta_{\tau,\rho}=\eta_{\tau,\sigma}\circ\iota_{\tau,\sigma}^*(\eta_{\sigma,\rho})$ for $\tau\prec\sigma\prec\rho$.
\end{definition}

Therefore $\Delta$-families  represent equivariant quasicoherent sheaves on $X_\Delta$.
Some properties of quasicoherent sheaves can be characterised conveniently in terms of the corresponding $\Delta$-families. In fact, the following properties are local and therefore only depend on the $\sigma$-families making up $\Ecal$.

\begin{proposition}[{\cite[Definition~5.10 and Proposition~5.11]{Perling}}]
An equivariant quasicoherent sheaf $\Ecal$ on a toric variety $X_{\sigma}$ is coherent if and only if the corresponding $\sigma$-family $E$ is \emph{finite}, in the sense that it has the following properties~:
	\begin{enumerate}[label=(\roman*)]
\item every vector space $E^\sigma_\delta$ is of finite dimension over $k$~;
\item for every strictly descending chain $\delta_1>\delta_2>\cdots$ in $P^{\rm gp}$ there exists an $n_0$ such that $E_{\delta_n}$ is $0$ for $n\geq n_0$;
\item \label{it:17} there are only finitely many linear functions $\delta_1,\ldots,\delta_N$ such that the sum of multiplication maps
\begin{equation}\label{sum_mult}
    \bigoplus_{\gamma<\delta}E^\sigma_\gamma\to E^\sigma_\delta
\end{equation}
is not surjective.
\end{enumerate}
\end{proposition}
The $k[P]$ module $E$ is finitely generated by generators of $E_{\delta_1},\ldots,E_{\delta_N}$, where $\delta_1, \ldots, \delta_N$ are as in~\ref{it:17}. More precisely, one can find a basis by lifting bases for the finitely many, finite dimensional cokernels of the maps in~\eqref{sum_mult}. It follows that, for every infinite ascending chain in $P^{\rm gp}$, the multiplication maps eventually stabilise to isomorphisms.

\begin{proposition}[{\cite[Proposition~5.13]{Perling}}]
    An equivariant coherent sheaf $\Ecal$ on a toric variety $X_{\sigma}$ is torsion-free if and only if the corresponding multiplication maps $\chi_{\delta_1}^{\delta_2}$ of the associated $\sigma$-family  $E^\sigma$ are all injective.
\end{proposition}
In this case, the vector spaces $E^\sigma_{\delta}$ can all be viewed as subspaces of their colimit $\mathbf{E}^\sigma$. Moreover, localisation maps induce injective maps (which turn out to be isomorphisms, \cite[Corollary 5.16]{Perling}) $\mathbf{E}^\sigma\hookrightarrow\mathbf{E}^\tau$ for $\tau\prec\sigma$, so every vector space in a $\Delta$-family can be viewed as a subspace of $\mathbf{E}^0$ -where $0$ is the minimal cone of $\Delta$, which in turn is a vector space of the same dimension as the (generic) rank of $\Ecal$.
Reflexivity is reflected by the fact that all $\sigma$-families are determined only by the ones corresponding to the rays of $\Delta$ \cite[Theorem 5.19]{Perling}. 

This brings us to Klyachko's description of framed toric vector bundles on $X_\Delta$ in terms of compatible, bounded, increasing $\ZZ$-filtrations of a fixed vector space $\mathbf{E}^0$ (which, by \Cref{prop:pullback}, can be identified with the fiber of $\Ecal$ at $1\in T$). 
 
\subsection{Moduli of toric vector bundles}
Let $X_{\sigma}$ be an affine toric variety. The next proposition is \cite[Proposition 2.1.1]{Klyachko}~; see also \cite[Proposition~2.2]{Payne}.
\begin{proposition}\label{paynelocfree}
Every toric vector bundle on  $X_{\sigma}$ splits equivariantly into a direct sum
of toric line bundles. The underlying line bundles of these are trivial, while the torus action is encoded by a character $u\in P^{\rm gp}$.
In particular, we have a bijective correspondence between finite multisets $\mathbf{u}(\sigma)\subseteq P^{\rm gp}$ and isomorphism classes of toric vector bundles:
\[\mathbf{u}(\sigma)\;\;\;\mapsto\;\;\;\mathcal E=\bigoplus_{u\in\mathbf{u}(\sigma)} \OO_{\sigma}(u).\]

	If $\tau\preceq\sigma$ is a face, and $U_{\tau}\subseteq X_\sigma$ the corresponding torus-invariant open subset, the restriction of $\Ecal$ to $U_\tau$ corresponds to the image multiset $\mathbf u(\tau)=[\mathbf{u}(\sigma)]$ along the projection $P^{\rm gp} / (\rho^\perp \cap P^{\rm gp}) \to P^{\rm gp}/(\tau^\perp\cap P^{\rm gp})$.
\end{proposition}
\begin{remark}
    The multiset $\bf{u}(\sigma)$ can be interpreted as the equivariant Chern character of $\Ecal$, see \cite[Section~3]{Payne}. The rank of $\Ecal$ is the cardinality $r$ of $\mathbf u$.
\end{remark}
Rewritten in the language of moduli stacks, the previous statement reads as follows:
\begin{proposition}\label{lem:vbonaffine}
    The stack $\mathfrak V_{\bf{u}(\sigma)}$ of toric vector bundles with fixed Chern character $\bf{u}(\sigma)$ is isomorphic to the classifying stack $\mathrm{B}G(\bf{u}(\sigma))$, where $G(\bf{u}(\sigma))$ is the group of automorphisms of $\mathcal  E=\oplus_{u\in \bf{u}(\sigma)}\OO_{\sigma}(u).$
\end{proposition}

In order to understand the automorphism group $G(\mathbf u(\sigma))$ better, we make the following observations.

First, since $\Ecal$ splits, so does its endomorphism algebra:
\[\Hom_\sigma(\Ecal,\Ecal)=\bigoplus_{u,u'\in\mathbf u(\sigma)}\Hom(\OO_{\sigma}(u),\OO_{\sigma}(u'))=\bigoplus_{u,u'\in\mathbf u(\sigma)}H^0(\Acal_\sigma,\OO_{\sigma}(u'-u)),\]
where each factor is non-zero if and only if $u\leq u'$ in $P^{\rm gp}$, and in this case $H^0(\Acal_\sigma,\OO_{\sigma}(u'-u))\simeq k$ is $1$-dimensional.

Second, if $\tau\prec\sigma$ is a face, then  
 there is a natural inclusion 
\begin{equation}\label{eqn:9}
	G({\bf{u}}(\sigma))\hookrightarrow G({\bf{u}}(\tau)).
\end{equation}
Taking $\tau= 0$ we get an inclusion of $G(\mathbf u(\sigma))$ into $\operatorname{GL}_r(k)$.
We shall present $G(\bf{u}(\sigma))$ as the intersection of the parabolic subgroups 
preserving certain partial flags.




 \begin{lemma}\label{lem:parabolic}
     When $\rho\prec\sigma$ is a ray, the group $G(\bf{u}(\rho))$ is a parabolic subgroup of $\operatorname{GL}_r(k)$.
 \end{lemma}
\begin{proof}
    
    The multiset ${\bf{u}}(\rho)$ takes values in $P^{\rm gp}/(\rho^\perp\cap P^{\rm gp})\simeq\mathbf Z$.  Explicitly, we can write ${\bf{u}}(\rho)=(\delta_1^{r_1},\dots,\delta_k^{r_k})$, where the characters are in decreasing order $\delta_1>\delta_2>\dots>\delta_k$, and their multiplicities satisfy $\sum_{i=1}^k r_i=r$.

From the point of view of the $P^{\rm gp}$-graded $k[P]$-modules, the line bundle $\OO_\sigma(u)$ corresponds to $x^{-u}k[P]$ (having invariant sections if and only if $u\geq0$). In particular, for $\Ecal=\Ecal(\mathbf u(\rho))$ as above we obtain:
\[\dim_k(E^\rho_\delta)=\#\{u\in\mathbf u(\rho):\delta\geq-u\}.\]
The data of the weighted, increasing filtration:
\[0=W_0\subset W_1=E^\rho_{-\delta_1}\subset\ldots\subset W_k=E^\rho_{-\delta_k}=\underset{\delta}{\varinjlim} E^\rho_{\delta}=\Ecal\rvert_{1}\]
is then equivalent to the vector bundle
\[\Ecal=\bigoplus_{i=1}^k\OO_\rho(\delta_i)\otimes W_i/W_{i-1}.\]
The automorphism group $G(\mathbf u(\rho))\subseteq \on{GL}(\Ecal\rvert_{1})$ is the parabolic subgroup that fixes the partial flag $W_\bullet$.
\end{proof}

\begin{lemma}\label{lem:autocone}
    Given a cone $\sigma$ with rays $\rho\in\sigma(1)$, the automorphism group of $\Ecal(\mathbf u(\sigma))$ is the intersection of the parabolic subgroups corresponding to the rays:
	\begin{equation} \label{eqn:10}
	G(\mathbf{u}(\sigma))\cong \bigcap_{\rho\in\sigma(1)}G(\mathbf{u}(\rho))\subset \on{GL}_r(k).
	\end{equation}
\end{lemma}
\begin{proof}
Recall that $G(\mathbf u(\sigma))$ is the group of invertible elements in \[\on{End}_\sigma(\Ecal)=\bigoplus_{u,u'\in\mathbf u(\sigma)}H^0(X_\sigma,\OO_\sigma(u'-u))^T.\]
The toric variety $X_\sigma$ is normal, so by Hartogs' theorem:
\[H^0(X_\sigma,\OO_\sigma(u'-u))^T=\bigcap_{\rho\in\sigma(1)} H^0(U_\rho,\OO_\rho([u]'-[u]))^T\]
in the function field of $X_\sigma$. We conclude by summing over $\mathbf u(\sigma)^2$ and intersecting with $\on{GL}_r(k)$.
\end{proof}

Let now $\Delta$ be the fan of a normal toric variety $X_{\Delta}$. The Chern character of an equivariant bundle can be thought of as a collection $\Psi=\left\{\bf{u}(\sigma)\right\}_{\sigma\in\Delta}$ of multisets compatible with face restrictions, see \cite[Section~3]{Payne}. We fix such a $\Psi$. 

\begin{proposition}\label{thm:vbgeneral}
 The stack $\mathfrak V_{\Psi}$ of vector bundles on
 $\Acal_{\Delta}=[X_{\Delta}\slash T]$   with fixed Chern character $\Psi$ can be identified with the inverse limit 
	\[\varprojlim_{\sigma \in \Delta} \;\;\mathrm{B}G(\bf{u}(\sigma)).\]
\end{proposition}
\begin{proof}
    The Artin fan $\Acal_\Delta$ has a Zariski open cover given by the Artin cones $\Acal_\sigma,\ \sigma\in\Delta$, so the statement follows from \Cref{lem:vbonaffine}.
\end{proof}
Payne's formulation \cite[Theorem~3.9, Corollary~3.11]{Payne} uses the following definition~:
    \begin{definition}
        Let $\mathcal M_{\Psi}^{\text{fr}}$ denote the moduli stack of \emph{framed} equivariant vector bundles on $X_\Delta$ with fixed equivariant Chern character $\Psi$.  That is, an $S$-point of $\mathcal M_\Psi^{\text{fr}}$ is an equivariant vector bundle $\mathcal E$ on $X_\Delta \times S$ along with an isomorphism $\mathcal E \big\rvert_{1\times S}\simeq E\otimes\OO_S$, where $1\in X_\Delta$ is the unit of the dense torus, and $E$ is a fixed $k$-vector space of dimension~$r$.
    \end{definition}
Payne's description of the moduli space of framed vector bundles rests upon Klyachko's classification \cite[Theorem~2.2.1]{Klyachko}.  We write $\Delta(1)$ for the set of rays of a fan $\Delta$.
\begin{theorem}
  The category of $T$
  -equivariant vector bundles $\mathcal E$ on $X_{\Delta}$ with a fixed framing $\mathcal E\rvert_{1}\simeq E$ is equivalent to the category of weakly increasing filtrations $\left\{E^\rho_\delta\right\}^{\rho\in\Delta(1)}_{\delta\in\mathbf Z}$,  satisfying the following compatibility condition~:

	For every cone $\sigma\in\Delta$ denote by $P_\sigma=P^{\rm gp}\slash(\sigma^\perp\cap P^{\rm gp})$~; then, there exists a (not necessarily unique) decomposition $E=\oplus_{u\in P_\sigma} {E_{u}}$  such that%
	
\[E^\rho_{\delta}=\sum_{u(v_\rho)\geq -\delta} E_{u}\]
for every $\rho\prec\sigma$, where $v_\rho$ denotes the primitive generator of $\rho$.\footnote{Our conventions on filtrations are reversed with respect to those of \cite{Klyachko} and \cite{Payne}.}
\end{theorem}


\begin{corollary} \label{cor:8}
	The framed moduli space of vector bundles with Chern character $\bf{u}(\rho)$ on a  ray $\Acal_{\rho}\cong\mathbf[\mathbf A^1\slash \Gm]$ is the partial flag variety $\mathcal {Fl}(\bf{u}(\rho))$.
\end{corollary}
   
\begin{proof}
The compatibility condition is trivial in this case. See also  \Cref{lem:parabolic}.
\end{proof}

\begin{theorem}[{\cite[Theorem~3.9]{Payne}}]
	The moduli space $\mathcal M_{\Psi}^{\text{fr}}$ is the locally closed subscheme of the product of partial flag varieties $\prod_{\rho\in\Sigma(1)}\mathcal{Fl}(\bf{u}(\rho))$ consisting of those collection of flags $(W^\rho_\bullet)^{\rho\in\Delta(1)}$ satisfying the following condition for all cones $\sigma$ of $\Delta$~:
	\[\operatorname{dim}(\bigcap_{\rho \in \sigma(1)} W^\rho_{\delta} )=\#\left\{[u]\in{\bf{u}}(\sigma)\;\;\big|\;\; \forall \rho \in \sigma(1), \; [u](v_{\rho})\geq -\delta \right\}.\]
\end{theorem}

This numerical condition is a rephrasing of Klyachko's compatibility condition. We refer the reader to \cite[Proposition~3.13]{Payne} for the proof. 
The moduli stack $\mathfrak{V}_{\Psi}$ is nothing but the quotient of $\mathcal M_{\Psi}^{\text{fr}}$ by the action of the group $\on{GL}_r$ on the choice of a framing. We proceed to compare Payne's description with ours from \Cref{thm:vbgeneral}. 

\subsubsection{Rays}
Since $\operatorname{GL}_r(k)$ acts transitively on the variety $\mathcal{Fl}(\mathbf{u}(\rho))$  and $G(\mathbf{u}(\rho))$ is the stabilizer of a partial flag $W_{\bullet},$ we conclude that
\[[\mathcal{Fl}(\mathbf{u}(\rho))\slash \operatorname{GL}_r(k)]\simeq
\mathrm{B}G(\mathbf{u}(\rho))\]
proving that, for a single ray, our description of $\mathfrak V_{\mathbf u(\rho)}$ coincides with Payne's.
\subsubsection{Cones}
Let $\sigma$ be a cone. 
It follows from the general theory of actions of reductive groups on quasiprojective schemes and homogeneous spaces (see for example \cite{milne2012reductive}) that we have a locally closed immersion
\[\operatorname{GL}_r(k)\backslash  G(\mathbf{u}(\sigma))\hookrightarrow \prod_{\rho\in\sigma(1)}\mathcal{Fl}(\mathbf{u}(\rho))\]
since, by Lemma~\ref{lem:autocone}, $ G(\bf{u}(\sigma))$ is the stabilizer of a collection of flags $\left\{W^\rho_\bullet\right\}^{\rho\in\sigma(1).}$

On the other hand, $\operatorname{GL}_r(k)\backslash G(\bf{u}(\sigma))$ coincides with Payne's moduli space $\mathcal M_{\bf{u}(\sigma)}^{\text{fr}}$ since the action of $\operatorname{GL}_r(k)$ on $\mathcal M_{{\bf u}(\sigma)}^{\text{fr}}$ is transitive on the set of splittings of $E$ satisfying the Klyachko compatibility condition.

Forgetting the framing we then get the isomorphism of algebraic stacks
\[\mathrm{B}G(\mathbf{u}(\sigma))\cong [\mathcal M_{{\mathbf u}(\sigma)}^{\text{fr}}\slash \operatorname{GL}_r(k)]\]

\subsubsection{Fans} For $\Delta$ the fan of a normal toric variety and $\Psi=\left\{ {\mathbf u}(\sigma)\right\}$ a compatible collection of multisets, \cite[Theorem~3.9]{Payne} shows that 
\[\mathcal M_{\Psi}^{\text{fr}}=\bigcap_{\sigma\in\Delta} \left(\mathcal M_{{\mathbf u}(\sigma)}^{\text{fr}}\times\prod_{\rho\notin\sigma(1)} \mathcal{Fl}(\bf{u}(\rho))\right)\]
where the intersection is in the product of partial flag varieties indexed by $\rho\in\Delta(1).$

It follows from the explicit presentation as an intersection of locally closed subschemes given above that for each cone $\sigma$ we have a morphism
\[[\mathcal M_{\Psi}^{\text{fr}}\slash \operatorname{GL}_r(k)]\to \mathrm{B}G(\mathbf{u}(\sigma)) \]
compatible with face morphisms. 
The universal property of limits implies then that we have an isomorphism
\[[\mathcal M_{\Psi}^{\text{fr}}\slash \operatorname{GL}_r(k)]\to \mathfrak V_{\Psi}.\]

\begin{remark}
    The advantage of presenting $\mathfrak V_{\Psi}$ as an inverse limit is that this perspective generalizes to the moduli stack of vector bundles 
    on Olsson fans, whereas the meaning of fixing a framing would not be clear.
\end{remark}
\begin{remark}
   Payne's construction of the moduli stack of equivariant vector bundles can be generalized to moduli of equivariant coherent sheaves with fixed equivariant Chern character. This was first done in M. Kool's PhD thesis \cite{kool2010moduli}.
\end{remark}

\section{Quasicoherent sheaves on Olsson fans}\label{sec:sheaves}
\tocdesc{relates invertible, quasicoherent, and locally free sheaves on Olsson cones, as well as sections and cohomology, in terms of combinatorial data}

In this section we prove that locally free sheaves on Olsson cones behave in many ways similarly to equivariant locally free sheaves on affine toric varieties.

			\subsection{Sections} \label{sec:7}

\begin{proposition} \label{prop:exactness}
				Let $S$ be a fine, saturated logarithmic scheme and let $\pi : \sigma \to S$ be an Olsson precone over $S$.  Then $\pi_\ast$ is exact on quasicoherent sheaves, and the higher derived push-forwards vanish.
			\end{proposition}

			\begin{proof}
				The assertion is \'etale-local on $S$, so we can assume that there is a local homomorphism of fine, saturated monoids $P \to Q$ and a strict morphism $S \to \mathcal A_P$ such that $\sigma = S \mathop\times_{\mathcal A_P} \mathcal A_Q$.  By Proposition~\ref{prop:local_description_of_relative_Artin_cones}, there is an affine $S$-scheme $Y = S \mathop\times_{\Acal_P} X_Q $ and a diagonalizable group $T=T_Q$ over $S$ such that $\sigma = [ Y / T ]$.  The projection $\sigma \to S$ now factors:
				\begin{equation*}
					\sigma = [ Y / T ] \to \mathrm BT \to S
				\end{equation*}
				Since in each of the two steps, pushforward is exact and higher derived pushforwards of quasicoherent sheaves vanish, the same conclusions hold for the composition by the Leray spectral sequence.
			\end{proof}

			In the next proposition, we use the notation of Proposition~\ref{prop:local_description_of_relative_Artin_cones}.  We write $\sigma$ for $\Acal_Q \mathop\times_{\Acal_P} S$.  For each $\beta \in Q^{\rm gp}$, we write $\mathcal O_\sigma(\beta)$ for the quasicoherent sheaf on $\sigma$ whose $\gamma$-graded piece is $C_{\beta + \gamma}$.  This notation apparently conflicts with pullback along the projection $\pi : \sigma \to S$, because if $\beta \in P^{\rm gp}$, then $\pi^\ast \mathcal O_S(\beta)$ is the quasicoherent sheaf whose $\gamma$-graded piece is $C_\gamma(\beta)$.  However, Proposition~\ref{prop:local_description_of_relative_Artin_cones}~\ref{it:19} shows that these two graded modules are canonically isomorphic by multiplication by $y^\beta$.  When $\beta \in P^{\rm gp}$, we identify these two natural interpretations of $\mathcal O_\sigma(\beta)$ by way of this isomorphism.

		\begin{proposition} \label{prop:min}
			Let $\mm\colon P \to Q$ be a 
            homomorphism of fine, saturated, sharp monoids and let $\pi : \sigma = S\mathop\times_{\Acal_P} \Acal_Q \to S$ be the Olsson precone induced by base change along a morphism $S \to \mathcal A_P$. If $\gamma \in Q^{\rm gp}$ then 
			\begin{equation*}
				\pi_\ast \mathcal O_\sigma(\gamma) = \underset{\substack{\alpha \in P^{\rm gp} \\ \mm(\beta) \leq \gamma}}{\varinjlim} \: \OO_S(\alpha),
			\end{equation*}
			and $\mathrm R^i \pi_\ast\mathcal O_\sigma(\gamma) = 0$ for all $i > 0$.
		\end{proposition}

		\begin{proof}
			The vanishing of higher cohomology comes from Proposition~\ref{prop:exactness}.  Pushforward along $\pi$ corresponds to taking invariants under the action of the torus $T_Q=\Hom(Q^{\rm gp}, \Gm)$.  Since $C(\gamma)$ is $C$ with its grading shifted by $\gamma$, the pushforward $\pi_\ast \mathcal O_\sigma(\gamma)$ is the sum of the graded pieces of $C$ lying in graded degree~$\gamma$.  By \Cref{prop:local_description_of_relative_Artin_cones}, we can write
			\begin{equation*}
				C = \bigoplus_{\substack{\alpha \in P^{\rm gp}, \beta \in Q^{\rm gp} \\ \mm(\alpha)\leq\beta}} y^\alpha x^\beta \mathcal O_S(\alpha) \Big/ \left\langle (\varepsilon_p - y^{-p}) \bigoplus y^\alpha x^\beta \mathcal O_S(\alpha - p) \right\rangle_{p\in P} .
			\end{equation*}
			The piece in graded degree $\gamma$ is isomorphic to
			\begin{equation*}
				\bigoplus_{\substack{\alpha \in P^{\rm gp} \\ h(\alpha) \leq \gamma}} y^\alpha x^\gamma \mathcal O_S(\alpha) \Big/ \bigl\langle \varepsilon_p - y^{-p} \bigr\rangle_{p \in P} .
			\end{equation*}

			This is the usual presentation of $\varinjlim_{\mm(\alpha) \leq \gamma} \mathcal O_S(\alpha)$ as the quotient of the direct sum.
		\end{proof}

		\begin{corollary}\label{cor:structuresheaf}
			Let $P \to Q$ be an exact homomorphism of fine, saturated, sharp monoids.  Let $S$ be a fine, saturated logarithmic scheme and let $\pi : \sigma \to S$ be the base change of $\mathcal A_Q \to \mathcal A_P$ along $S \to \mathcal A_P$.  Then $\pi_\ast \mathcal O_\sigma = \mathcal O_S$.
		\end{corollary}

		\begin{proof}
                Let $\mm : P \to Q$ be an exact homomorphism and let $P_\gamma$ be the set of all $\beta \in P^{\rm gp}$ such that $\mm(\beta) \leq 0$.  Since $\mm$ is exact, this is the same as the set of all $\beta \in P^{\rm gp}$ such that $\beta \leq 0$, which clearly has a unique final object,~$0$.  Therefore $\pi_\ast \mathcal O_\sigma = \mathcal O_S(0) = \mathcal O_S$.
		\end{proof}

		\begin{corollary} \label{cor:4}
			Let $T$ be a tropical torus over a fine, saturated logarithmic scheme $S$ and let $\gamma$ be a character of $T$.  Then $\mathrm R^p \pi_\ast \mathcal O_T(\gamma) = 0$ for all $p > 0$ and $\pi_\ast \mathcal O_T(\gamma) = 0$ unless $\gamma = 0$,  in which case $\pi_\ast \mathcal O_T = \mathcal O_S$.
		\end{corollary}
		\begin{proof}
			This assertion is local in the strict-\'etale topology of $S$ so we assume that $S$ is strictly henselian.  Let $\mathcal C$ be the category of all universally surjective $\sigma \to T$ where $\sigma$ is an Olsson cone over $S$.  If $\sigma_1 \to T$ and $\sigma_2 \to T$ are both in $\mathcal C$ then so is $\sigma_1 \mathop\times_T \sigma_2$. 

			We argue that the collection of $\sigma\in\mathcal  C$ defines a cover of $T$ in the strict \'etale topology.  If $U$ is any fine, saturated logarithmic scheme then any morphism $U \to T$ factors, locally in the strict \'etale topolgy of $U$, through some $\sigma$ in $\mathcal C$.  Indeed, $U \to T$ certainly factors locally through some Olsson precone $\tau$ over $S$.  But then $\tau \to T$ factors through the universal surjection $\mu : \tau \times \sigma \to T$, for any $\sigma \in \mathcal C$, with $\mu$ denoting the sum of the maps $\tau \to T$ and $\sigma \to T$ in the group structure of $T$.  

			Since $\mathcal C$ covers $T$, there is a \v Cech spectral sequence:
			\begin{equation*}
				H^p(\mathcal C, H^q(\sigma, \mathcal O_\sigma(\gamma))) \Rightarrow H^{p+q}(T, \mathcal O_T(\gamma))
			\end{equation*}
			where $H^p(\mathcal C, -)$ is the Čech cohomology.

			But $\mathcal O(\gamma)$ is acyclic on the objects of $\mathcal C$ by Proposition~\ref{prop:min}, so
			\begin{equation*}
				H^p(T, \mathcal O_T(\gamma)) = H^p(\mathcal C, \mathcal O(\gamma))  .
			\end{equation*}
			The presheaf $\pi_*\mathcal O(\gamma)$ is constant on $\mathcal C$ with value $\mathcal O_S$ if $\gamma = 0$ and value $0$ otherwise. Indeed, by Proposition~\ref{prop:min}, if $\gamma$ is nonzero, then it is unbounded below on each object of $\mathcal  C$. On the other hand, for $\gamma=0$,  since $\pi$ is universally surjective  and thus exact on logarithmic structures, the equality $\pi_*\mathcal O(\gamma)=\OO_S$ is proved in Corollary~\ref{cor:structuresheaf}.

            Finally, if $\sigma_1$ and $\sigma_2$ are in $\mathcal C$ then so is $\sigma_1 \mathop\times_T \sigma_2$, so the complex of sections of $\mathcal O(\gamma)$ on $\mathcal C$ is the same as the cellular cohomology of an infinite simplex and therefore vanishes in positive degrees.
		\end{proof}

		\begin{corollary} \label{cor:5}
			Let $S$ be a fine, saturated logarithmic scheme.  Let $\pi : \sigma \to S$ be a weakly convex Olsson precone over $S$.  Let $\gamma$ be a linear function on $\sigma$.  Then $\mathrm R^i \pi_\ast \mathcal O_\sigma(\gamma) = 0$ for all $i > 0$ and the natural map
			\begin{equation} \label{eqn:5}
				\underset{\substack{\beta \in \bar M_S^{\rm gp} \\ \beta \leq \gamma}}{\varinjlim} \mathcal O_S(\beta) \to \pi_\ast \mathcal O_\sigma(\gamma)
			\end{equation}
			is an isomorphism.
		\end{corollary}
		\begin{proof}
			By Proposition~\ref{prop:polyhedral domain}, we may identify $\sigma$ with the product of a tropical torus $T$ and an Olsson precone $\sigma'$ by working locally in the strict \'etale topology of $S$.  For $i > 0$, the assertion follows immediately from Proposition~\ref{prop:min} and Corollary~\ref{cor:4} and the Leray spectral sequence.  By Corollary~\ref{cor:4}, both sides of~\eqref{eqn:5} are zero when $\gamma$ is not constant on the fibers of $\sigma$ over $\sigma'$.  The remaining case is where $\gamma$ is pulled back from $\sigma'$, in which case the conclusion is immediate from Proposition~\ref{prop:min} and Corollary~\ref{cor:4}.
		\end{proof}

		Corollary~\ref{cor:structuresheaf} also implies that the base of a universally surjective Olsson precone is a good moduli space:

   \begin{proposition}\label{gms_univclosed}
              Let $S$ be a fine, saturated logarithmic scheme and let $\pi : \sigma \to S$ be an Olsson precone such that the projection to $S$ is universally surjective.  Then $\pi : \sigma \to S$ is a good moduli space. In particular, it is surjective and universally closed. 
            \end{proposition}
            \begin{proof}
							By definition \cite[Section~1.2]{ALPgood}, we must show that $\pi$ is quasicompact (immediate), that $\pi_\ast$ is exact on quasicoherent sheaves (Proposition~\ref{prop:exactness}) and that $\mathcal O_S \to \pi_\ast \mathcal O_\sigma$ is an isomorphism (Corollary~\ref{cor:structuresheaf}).  The surjectivity and universal closedness follow from \cite[Theorem~4.16]{ALPgood}.
            \end{proof}

		\begin{corollary} \label{cor:10}
			Let $\mm : P \to Q$ be a local, integral homomorphism of fine, saturated, and sharp monoids and suppose $\gamma \in Q^{\rm gp}$.  Let $S$ be a fine, saturated logarithmic scheme and let $\pi : \sigma \to S$ be the Olsson precone obtained from $\mathcal A_Q \to \mathcal A_P$ by base change along a morphism $S \to \mathcal A_P$.  For each geometric point $s$ of $S$, let $\gamma_s$ be the image of $\gamma$ in $\sigma_s = \pi^{-1}(s)$.  Then there is a section of $\bar M_S^{\rm gp} \cup \{ -\infty \}$ whose stalk at each geometric point $s$ of $S$ is the greatest lower bound $\gamma_s$.
		\end{corollary}
		\begin{proof}
			Since $\mm$ is local and integral, \Cref{prop:1} guarantees that $\gamma$ has a greatest lower bound $\mu \in P^{\rm gp} \cup \{ -\infty \}$.  It is compatible pullback to $s$ by Proposition~\ref{prop:7}.
		\end{proof}

		\begin{corollary}\label{cor:pushforward}
			Let $P \to Q$ be a local, integral (resp.\ local, integral, and vertical) homomorphism of fine, saturated, sharp monoids.  Let $\pi : \sigma \to S$ be the base change of $\mathcal A_Q \to \mathcal A_P$ along a morphism $S \to \mathcal A_P$.  Let $\delta$ be the section of $\bar M_S^{\rm gp} \cup \{ -\infty \}$ guaranteed by Corollary~\ref{cor:10}.  Then $\pi_\ast \mathcal O_\sigma(\gamma) = \mathcal O_S(\delta)$, with the understanding that $\mathcal O_S(-\infty) = 0$.  In particular, $\pi_\ast \mathcal O_\sigma(\gamma)$ is a locally free sheaf of rank $\leq 1$ (resp.\ an invertible sheaf).
		\end{corollary}
		\begin{proof}
			This is immediate from Corollary~\ref{cor:9} and Proposition~\ref{prop:min}.
		\end{proof}

		\begin{remark}
			While $\delta$ is uniquely determined by $\gamma$ in Corollary~\ref{cor:pushforward}, only $\mathcal O_S(\delta)$ is determined uniquely by $\mathcal O_\sigma(\gamma)$.  In other words, it is possible that the line bundle $\mathcal O_S(\delta)$ does not determine~$\delta$. This is the case, for instance, when all line bundles on $S$ are trivial.
		\end{remark}

\subsection{$Q/P$-systems} \label{sec:1}
It follows immediately from the local description of Olsson precones in \Cref{prop:local_description_of_relative_Artin_cones} that quasicoherent sheaves on $\sigma$ can be represented, locally on $S$, as $Q^{\rm gp}$-graded $C$-modules.
We make this description more explicit by working with the graded pieces, generalizing Perling's $\sigma$-families.

\begin{definition} \label{def:6}
	Let $P \to Q$ be a homomorphism of fine, saturated, sharp monoids. Let $S$ an algebraic stack with a 
	morphism to $\Acal_P$. A \emph{$Q/P$-filtered} quasicoherent sheaf on $S$ is a system of quasicoherent sheaves $\{F_{\alpha}\}_{\alpha\in Q^{\rm gp}}$ on $S$, together with maps $f_{\alpha,\beta}\colon F_\alpha \to F_\beta$ whenever $\alpha \leq \beta$ in the partial order induced by $Q\subseteq Q^{\rm{gp}}$ such that
	\begin{enumerate}[label=(\roman*)]
		\item $f_{\beta,\gamma} \circ f_{\alpha,\beta} = f_{\alpha,\gamma}$ whenever $\alpha \leq \beta \leq \gamma$~;
		\item \label{it:18} there are isomorphisms $F_{\alpha + \delta} \simeq F_\alpha\otimes\OO_S(\delta)$ when $\delta \in P^{\rm gp}$, depending homomorphically on $\delta$~; and
		\item under the isomorphism from \ref{it:18}, the map $F_\alpha \to F_{\alpha + \delta}$ coincides with $\varepsilon_\delta$ when $\delta \in P$.
	\end{enumerate}
\end{definition}


\begin{proposition}\label{cor:systems_of_bundles}
	Let $P \to Q$ be a homomorphism of fine, saturated, sharp monoids, let $S$ be a fine, saturated, logarithmic scheme, and let $S \to \mathcal A_P$ be a morphism of stacks with logarithmic structures.  Let $\pi : \sigma \to S$ denote the  Olsson precone obtained by base change from $\mathcal A_Q \to \mathcal A_P$. 
Then quasicoherent sheaves on $\sigma$ are equivalent to $Q/P$-filtered quasicoherent sheaves on $S$.
\end{proposition}

\begin{proof}[First proof.]
	The data in \Cref{def:6} are equivalent to the structure of a graded $C$-module, in the notation of \Cref{prop:local_description_of_relative_Artin_cones}.  Indeed, one gets a $Q^{\rm gp}$-graded quasicoherent $\mathcal O_S$-module from the direct sum of the $F_\alpha$.  The maps $f_{\alpha,\beta}$ encode the action of $A$, in the notation of \Cref{prop:local_description_of_relative_Artin_cones}~; the isomorphisms $F_{\alpha + \delta} \simeq F_\alpha \otimes \mathcal O_S(\delta)$ correspond to the promotion to a $B$-action~; and the identification of $f_{\alpha+\delta,\alpha}$ with $\varepsilon_\delta$ when $\delta \in P$ corresponds to the descent to a $C$-action.
\end{proof}

We give a second, direct demonstration of \Cref{cor:systems_of_bundles}.

\begin{proof}[Second proof.]
	Suppose that $\mathcal E$ is a quasicoherent sheaf on $\sigma$.  Set 
	\begin{equation*}
		\Phi(\mathcal E) = \{ \pi_\ast \mathcal E(\alpha)\}_{\alpha\in Q^{\rm gp}},
	\end{equation*}
	with maps induced by $\mathcal E(\alpha)\to \mathcal E(\beta)$ whenever $\alpha\leq\beta$ in $Q^{\rm gp}$, which is a $Q/P$-filtered quasicoherent sheaf on $S$. 
	
	Conversely, if $F = \{ F_\alpha \}_{\alpha\in Q^{\rm gp}}$ is a $Q/P$-filtered quasicoherent sheaf on $S$, we can form a quasicoherent sheaf on $\sigma$ by
	\begin{equation} \label{eqn:1}
		\Psi(F) = \kern-1em \varinjlim_{\substack{\alpha, \beta \in Q^{\rm gp}
        \\ \alpha \leq \beta}} 
        \pi^\ast F_\alpha(-\beta) = \coker \bigl( \bigoplus_{\alpha \leq \beta} \pi^\ast F_\alpha(-\beta) \to \bigoplus_\alpha \pi^\ast F_\alpha(-\alpha) \bigr) .
	\end{equation}
	The first direct sum is taken over pairs $(\alpha,\beta) \in Q^{\rm gp} \times Q^{\rm gp}$ such that $\alpha \leq \beta$ and the second is taken over $Q^{\rm gp}$. On the $F_\alpha(-\beta)$ component, the map is the difference of the natural map $\varepsilon_{\beta-\alpha}$ to the $F_\alpha(-\alpha)$ summand and the transition morphism $f_{\alpha,\beta}$ to the $F_\beta(-\beta)$ summand.

	We argue that $\Phi(\Psi(F))=F$ using the projection formula%
	\footnote{Since $\pi : \sigma \to S$ factors as an affine morphism followed by the projection $\mathrm BT \to S$, where $T$ is a diagonalizable group, we have $\pi_\ast ( G \otimes \pi^\ast F ) = \pi_\ast(G) \otimes F$ for \emph{all} quasicoherent sheaves $F$ on $S$, and not only the flat ones.}
	and exactness of $\pi_\ast$~:
	\begin{equation*} \begin{aligned}
		\Phi(\Psi(F))_\gamma 
		& = \pi_\ast \Bigl( \kern-1ex \varinjlim_{\substack{ \alpha, \beta \in Q^{\rm gp}\\\alpha \leq \beta}} \kern-1ex \pi^\ast F_\alpha(-\beta + \gamma) \Bigr) 
		\\ & = \kern-1ex \varinjlim_{\substack{ \alpha, \beta \in Q^{\rm gp}\\\alpha \leq \beta}} \kern-1ex F_\alpha \otimes \pi_\ast \mathcal O_\sigma(-\beta + \gamma)
		\\ & = \kern-3.75ex \varinjlim_{\substack{ \alpha, \beta \in Q^{\rm gp} , \; \delta \in P^{\rm gp}\\\alpha \leq \beta \leq \gamma - \delta}} \kern-3.75ex F_\alpha(\delta) 
		\\ & = \varinjlim_{\delta \in P^{\rm gp}} F_{\gamma - \delta}(\delta)  = F_\gamma
	\end{aligned}
	\end{equation*}

	Thus $\Phi(\Psi(F)) \simeq F$, so $\Psi$ is fully faithful.  On the other hand, if $\mathcal E$ is quasicoherent on $\sigma$ then we have a map $\Psi(\Phi(\mathcal E)) \to \mathcal E$.  Let $K$ and $L$ be its kernel and cokernel.  Then $\Phi(K) = \Phi(L) = 0$ since $\Phi$ is exact by Proposition~\ref{prop:exactness}.  Therefore it suffices to show that $\Phi(\mathcal E) = 0$ only if $\mathcal E = 0$.

	But there is an diagonalizable group $T$ and a factorization of $\sigma \to S$ through $\mathrm B T$ with $\tau : \sigma \to \mathrm B T$ affine.  If $\Phi(\mathcal E) = 0$ then $\tau_\ast \mathcal E = 0$, which then implies that $\mathcal E = 0$ because $\tau$ is affine.
\end{proof}

\begin{remark}
    When $S$ is a point with trivial log structure, and thus $\sigma$ is simply the Artin cone $\Acal_Q$, the previous proposition reduces to Perling's equivalence  between quasicoherent sheaves and $\sigma$-families \cite{Perling}, recalled in Section~\ref{sec:2}.
\end{remark}

\begin{corollary}
	Let $h\colon P \to Q$ and $g\colon Q \to R$ be homomorphisms of fine, saturated, sharp monoids.  Let $S$ be a fine, saturated logarithmic scheme with a morphism to $\Acal_P$, and let $\psi\colon\tau \to \sigma$ be the morphism of Olsson cones over $S$ pulled back from $\mathcal A_R \to \mathcal A_Q$. 
    Let $\mathcal F$ be a quasicoherent sheaf on $\sigma$ and let $\{ F_\alpha \}_{\alpha \in Q}$ be its associated $Q/P$-filtered $\mathcal O_S$-module.  Then $\psi^\ast \mathcal F$ has associated $R/P$-filtered module
	\[\{\varinjlim_{g(\alpha)\leq\beta} F_{\alpha} \}_{\beta \in R^{\rm gp}}. \]
\end{corollary}
\begin{proof}
	Let $\rho : \tau \to S$ and $\pi : \sigma \to S$ be the projections.  Then from Proposition~\ref{cor:systems_of_bundles} we have 
	\begin{equation*} \begin{aligned}
		(\psi^*\mathcal F)_{\beta}&=\rho_\ast(\psi^*\mathcal F\otimes\mathcal O_{\tau}(\beta))=\pi_\ast\left(\mathcal F\otimes \psi_\ast \mathcal O_\tau(\beta) \right) \\[.5\baselineskip]
		& = \pi_\ast \Bigl( \mathcal F  \otimes \kern-.5ex \varinjlim_{\substack{\alpha\in Q^{\rm gp} \\ g(\alpha)\leq \beta}} \kern-1ex \mathcal O_{\sigma}(\alpha) \Bigr) = \varinjlim_{\substack{\alpha \in Q^{\rm gp} \\ g(\alpha) \leq \beta}} \pi_\ast \mathcal F(\alpha) = \varinjlim_{\substack{\alpha \in Q^{\rm gp} \\ g(\alpha) \leq \beta}} F_\alpha .
	\end{aligned}
	\end{equation*}
\end{proof}

We obtain a simpler description if we assume that the morphism $S\to\Acal_P$ factors through the toric variety $X_P$, in which case we may use the local description of \Cref{cor:1}.

\begin{definition} \label{def:5}
	Let $h : P \to Q$ be a homomorphism of fine, saturated, sharp monoids. Let $S$ an algebraic stack with a 
	morphism to $X_P$.  Let $K$ be the kernel of $P^{\rm gp} \to Q^{\rm gp}$, let $L$ be the cokernel, and let $M$ be the image.  Fix a splitting $P^{\rm gp} = K \oplus M$ of the surjection $P^{\rm gp} \to M$.  A \emph{reduced $Q/P$-filtered} quasicoherent sheaf on $S$ is a system of quasicoherent sheaves $\{E_{\alpha}\}_{\alpha\in L}$ on $S$, together with 
	\begin{enumerate}[label=(\roman*)]
		\item maps $x^q=x^q_{\alpha,\beta} : E_\alpha \to E_\beta$ for each $q \in Q$ such that $q + \alpha = \beta$, and
		\item isomorphisms $u^p = u^p_\alpha : E_\alpha \to E_\alpha$ for each $p \in K$,
	\end{enumerate}
	such that
	\begin{enumerate}[resume*]
		\item $x^q x^{q'} = x^{q+q'}$ for all $q, q' \in Q$, 
		\item $u^p u^{p'} = u^{p+p'}$ for all $p, p' \in K$,
		\item $x^{h(p)} : E_\alpha \to E_\alpha$ coincides with multiplication by $\varepsilon_p(v^{-p})$ for $p \in M$,
		\item $x^{h(p)} : E_\alpha \to E_\alpha$ coincides with multiplication by $u^p \varepsilon_p(v^{-p})$ for $p \in K$, and
		\item $x^q$ and $u^p$ commute for all $q \in Q$ and  $p \in K$.
	\end{enumerate}
\end{definition}

\Cref{def:5} simplifies considerably in the situation of primary interest, where $P \to Q$ is injective.  In this case, the splitting $g$ is unique and the isomorphisms $u$ may be suppressed.

\begin{corollary}\label{cor:reducedQPsystems}
    Let $P \to Q$ be a homomorphism of fine, saturated, sharp monoids and let $S \to X_P$ be a morphism of schemes.  Let $\pi : \sigma \to S$ denote the relative Olsson cone obtained by base change from $\mathcal A_Q \to \mathcal A_P$. 
Then quasicoherent sheaves on $\sigma$ are equivalent to reduced $Q/P$-filtered quasicoherent sheaves on $S$.
\end{corollary}

Note that the equivalence in \Cref{cor:reducedQPsystems} is not canonical, but rather depends in the choice of a splitting of the inclusion $K \subset P^{\rm gp}$.  However, the equivalence is canonical when this splitting is unique, namely when $K = 0$ or $K = P^{\rm gp}$.

\begin{proof}
	This is simply an interpretation of the notion of a graded $C''$-module (in the notation of \Cref{cor:15}) in terms of the graded pieces.
\end{proof}

Some properties of quasicoherent sheaves on $\sigma$ can be effectively translated in terms of $Q/P$-filtered systems. We discuss finite presentability and, under an integrality assumption, (log) flatness. We work locally in $S=\Spec(R)$ and assume that there is a morphism $S\to X_P$.  We continue to use $C''$ to refer to the algebra from \Cref{cor:15}.  Note that when $P^{\rm gp} \to Q^{\rm gp}$ is injective, $C'' = R\otimes_{\ZZ[P]}\ZZ[Q]$.

\begin{proposition}
	Let $\sigma$ be an Olsson precone on $S$, obtained by base change from $\mathcal A_Q \to \mathcal A_P$ for a homomorphism $h : P \to Q$ of fine, saturated, sharp monoids.  Let $\Ecal$ be a quasicoherent sheaf on $\sigma$. Let $\{ E_q \}_{q \in Q^{\rm gp}}$ be its associated $Q/P$-filtered system.  Let $E$ be the associated $C$-module in the notation of \Cref{prop:local_description_of_relative_Artin_cones}.  The following conditions are equivalent after \'etale-localization in $S$~:
	\begin{enumerate}[label=(\roman*)]
		\item \label{it:20} The $\mathcal O_\sigma$-module $\mathcal E$ is of finite type.
		\item \label{it:21} The $C$-module $E$ is of finite type.
		\item \label{it:28} The $\mathcal O_\sigma$-module $\mathcal E$ is a quotient of a finite direct sum $\bigoplus \mathcal O_\sigma(\alpha_i)$.
		\item \label{it:26} There is a finite subset $F \subset Q^{\rm gp}$ such that for every $q \in Q^{\rm gp}$, there is some $p \in P$ such that the map
			\begin{equation*}
				\bigoplus_{\substack{f \in F\\f \leq h(p) + q}} E_f \to E_{h(p) + q}
			\end{equation*}
			is surjective.
	\end{enumerate}
	If $\mathcal E$ has a reduced $P/Q$-filtered system $\{ E'_\alpha \}_{\alpha \in Q^{\rm gp} / P^{\rm gp}}$ (in other words, if $S \to \mathcal A_P$ factors through the toric variety $X_P$), we take $C'' = \mathcal O_S \otimes_{\mathbf Z[P]} \mathbf Z[Q]$ and $E'$ to be the associated $C''$-module.  Then the above conditions are also equivalent to the following ones~:
	\begin{enumerate}[resume*]
		\item \label{it:27} The $C''$-module $E'$ is of finite type.
		\item \label{it:25} There is a finite subset $F$ of $Q^{\rm gp} / P^{\rm gp}$ such that $E'_\alpha$ is of finite type over $\mathcal O_S$ for all $\alpha \in F$ and $E'$ is generated as a $C''$-module by the $E'_\alpha$ for $\alpha \in F$.
	\end{enumerate}
\end{proposition}
\begin{proof}
	The equivalence of~\ref{it:20},~\ref{it:21}, and~\ref{it:27} is immediate by descent.  Since a graded $C$-module is always a quotient of a direct sum of shifts of $C$, a graded $C$-module is of finite type if and only if it is a quotient a finite direct sum of shifts of $C$.  Hence~\ref{it:21} and~\ref{it:28} are equivalent.  Item~\ref{it:26} is just a reinterpretation of~\ref{it:28} in terms of a $Q/P$-filtered system, and~\ref{it:25} is a reinterpretation in terms of a reduced $Q/P$-filtered system.
\end{proof}

%

For the next proposition we are going to assume that the monoid homomorphism $h\colon P\to Q$ is injective and free, with a finite basis $B\subset Q$. In particular, the associated Olsson cone $\sigma$ will be integral over the base $S$, see \Cref{rmk:integral-free}. The quotient group $(Q/P)^{\rm gp}$ will be denoted by $L$.  Following \cite[Definition~7.11.5]{GillamLogFlat}, we choose a \emph{finite spawning set} $G\subset Q$~: this means that the submonoid generated by $G$ contains $B$ (and therefore $G$ generates $Q/P$ and $L$). 

\begin{proposition}
    Under the previous assumptions, a quasicoherent sheaf $\Ecal$ on $\sigma$ is flat over $\sigma$ if and only if the associated reduced $Q/P$-filtered system $\{E_\alpha\}_{\alpha\in(Q/P)^{\rm gp}}$ satisfies the following properties :
    \begin{enumerate}
        \item every $E_{\alpha}$ is a flat $\OO_S$-module, and
	\item for every $g\in G$, $\on{Tor}_1^{C''}(\oplus E_{\alpha},C''/x^gC'')=0$ and $(\oplus E_\alpha)\otimes_{C''}C''/x^gC''$ is flat over the vanishing locus of $g$ in $\sigma$.
    \end{enumerate}
\end{proposition}
\begin{proof}
	This is rephrasing \cite[Theorem 7.11.7]{GillamLogFlat}. Indeed, flatness over the Olsson cone $\sigma$ is equivalent to $(L,C'')$-graded flatness over $\mathcal O_S$ (graded by the trivial group) by \cite[Proposition 7.12.2]{GillamLogFlat}. The aforementioned theorem provides a criterion for a graded $C''$ module to be $(Q^{\rm gp},C'')$-graded flat over $\mathcal O_S$.  We make several observations to connect graded flatness with our notation.
Consider the diagram :
\[ X=\Spec(C'')\xrightarrow{\pi} Y=[\Spec(C'')/T_{L}]\xrightarrow{\rho} Z=[\Spec(C'')/T_{Q}].\]
Since $\pi$ is a flat cover, a quasicoherent sheaf $\Ecal$ on $Y$ is flat if and only if $\pi^*\Ecal$ is $(L,C'')$-graded flat. On the other hand, for an $L$-graded $C''$-module $\Ecal$ we have $(\rho_*\Fcal)_\alpha=\Fcal_{[\alpha]}$, i.e. $\rho_*\Fcal=\oplus_{P^{\rm gp}}\Fcal$, so $\pi^*\Ecal$ is $(L,C'')$-graded flat if and only if it is $(Q^{\rm gp},C'')$-graded flat.
\end{proof}
\begin{remark}
    These propositions generalize the characterization of finitely presentable, flat modules over $\Theta_R$ and $\overline{ST}_R$ given in \cite[\S 7.2]{AHLH2023}, see also \cite[Corollaries 7.11.9-11]{GillamLogFlat}.
\end{remark}
\begin{remark}
    Let $\pi\colon X\to S$ be a family of logarithmic curves over the spectrum of a discrete valuation ring $R$ with its divisorial logarithmic structure, and let $\tau\colon X\to\Gamma$ be the tropicalization map.  Assume that the central fiber of $X$ consists of two components joined at a single node, so that $\Gamma$ can be identified with $\overline{ST}_R$. Let $L$ be a line bundle on $X$. Then (despite the tropicalization map not being proper) $\tau_* L$ is a finitely presentable, flat module over $\Gamma$, as can be seen by pushing further down to $S$. A vector subbundle $\Ecal\subset\tau_*L$ can be identified with a \emph{limit linear series} in the sense of Osserman \cite{OssermanLLS} by taking the associated reduced $\NN^2/\NN$-filtered system.
\end{remark}

		\subsection{Vector bundles}

			We begin with a direct proof of \Cref{line_bundles_on_cones} (which will be demonstrated more generally later in \Cref{cor:6}) using the theory of good moduli spaces.

            \begin{proposition}\label{line_bundles_on_cones}
                Let $S$ be a fine, saturated logarithmic scheme and let $\pi : \sigma \to S$ be a universally surjective Olsson precone.  If $\mathcal L$ is an invertible sheaf on $\sigma$ then, locally in the strict \'etale topology of $S$, there exists a section $\lambda$ of $\bar M_\sigma^{\rm gp}$, unique up to the addition of sections of $\bar M_S^{\rm gp}$, and an invertible sheaf $L$ on $S$ such that there is an isomorphism
			\begin{equation*}
				\mathcal L \simeq \pi^*L(\lambda).
			\end{equation*}
            \end{proposition}
            \begin{proof}
							We note first that we may assume $S$ is strictly henselian.  Indeed, suppose that $\tilde S$ is the henselization of $S$ at a geometric point, and $\mathcal L \big|_{\tilde S} \simeq \pi^*L(\lambda)$ for an invertible sheaf $L$ on $S$ and a section $\lambda$ of $\bar M_{S,s}$.  Then $L$, $\lambda$, and the isomorphism all extend to an \'etale neighborhood of $s$ (the extension of the isomorphism uses the fact that $\sigma$ is quasicompact and quasiseparated over $S$).

							By \Cref{gms_univclosed}, $\pi$ is a good moduli space. By \cite[Theorem 10.3]{ALPgood}, pullback along $\pi$ identifies the category of vector bundles on $S$ with the full subcategory of vector bundles $\Ecal$ on $\sigma$ such that, for every geometric point $\bar x\to \sigma$ with closed image, the action of the stabiliser group $G_{\bar x}$ on $\Ecal\otimes k(\bar x)$ is trivial. Since the assertion is local in the strict \'etale topology of $S$, we may assume that $\sigma$ is pulled back from $\Acal_Q \to \Acal_P$ for some exact homomorphism of fine saturated monoids, $P \to Q$.  We can assume that $S \to \Acal_P$ factors through the toric variety $X_P$ and use the description of $\sigma$ from \Cref{cor:1}.			If $s$ is a geometric point of $S$, the fiber $\pi^{-1} s$ contains a unique closed stratum $Z$ isomorphic to $\mathrm BT$ where $T$ is the diagonalizable group with character group $Q^{\rm gp} / P^{\rm gp}$. The restriction of $\mathcal L$ to $Z$ determines a character $\mu$ of $T$.  Let $\lambda$ be a lift of $\mu$ to $Q^{\rm gp}$.  Then $\mathcal L(-\lambda)$ has trivial stabilizer along $Z$.  But $Z$ is the unique closed stratum of $\sigma$, so by \cite[Theorem~10.3]{ALPgood}, $L = \pi_\ast \mathcal L(-\lambda)$ is invertible and $\mathcal L \simeq \pi^\ast L(\lambda)$.
            \end{proof}

		\begin{proposition} \label{prop:vect-on-cone}
			Let $S$ be a fine and saturated logarithmic scheme.  Let $\pi : \sigma \to S$ be a universally surjective Olsson precone over $S$.  If $\mathcal E$ is a locally free sheaf of finite rank on $\sigma$ then, locally in the strict \'etale topology of $S$, there exist sections $\lambda_0, \ldots, \lambda_r$ of $\bar M_\sigma^{\rm gp}$ and invertible sheaves $E_0, \ldots, E_r$ on $S$ and an isomorphism
			\begin{equation*}
				\mathcal E \simeq E_0(\lambda_0) \oplus \cdots \oplus E_r(\lambda_r) .
			\end{equation*}
			The $\lambda_0, \ldots, \lambda_r$ are uniquely determined up to permutation, further localization in the strict \'etale topology of $S$, and addition of sections of $\bar M_S^{\rm gp}$.
		\end{proposition}

		\begin{proof}
			Since the assertion is strict-\'etale local on $S$, we may assume that $\sigma$ is pulled back from $\Acal_Q \to \Acal_P$ for some exact homomorphism of fine saturated monoids, $P \to Q$.  We can assume that $S \to \Acal_P$ factors through the toric variety $X_P = \Spec \mathbf Z[P]$ and therefore that quasicoherent sheaves on $\sigma$ may be identified with $Q^{\rm gp} / P^{\rm gp}$-graded quasicoherent $Q \otimes_P \mathcal O_S$-modules.  

			If $s$ is a geometric point of $S$, the fiber $\pi^{-1} s$ contains a unique closed stratum $Z$ isomorphic to $\mathrm BT$ where $T$ is the diagonalizable group with character group $Q^{\rm gp} / P^{\rm gp}$.  If $\mathcal E$ is a locally free sheaf on $\sigma$ then the restriction of $\mathcal E$ to $Z$ can be viewed as a $T$-equivariant locally free $\mathcal O_S$-module, or equivalently, as a $Q^{\rm gp} / P^{\rm gp}$-graded locally free $\mathcal O_S$-module.  If $E$ is the graded $Q \otimes_P \mathcal O_S$-module corresponding to $\mathcal E$ then the graded $\mathcal O_S$-module corresponding to $\mathcal E \big|_Z$ is the associated graded module $\operatorname{gr} E$ of $E$.  

			A graded generator of $\operatorname{gr} E$ corresponds to a homomorphism $\mathcal O_Z(\lambda) \to \mathcal E \big|_Z$ for some $\lambda \in Q$.  Since $E$ is a graded $Q \otimes_P \mathcal O_S$-module, we can (at least after shrinking $S$ further) lift the generating maps $\mathcal O_Z(\lambda) \to \operatorname{gr} E$ to maps $\mathcal O_\sigma(\lambda) \to \mathcal E$~: a lift corresponds to lifing a homogeneous element of $\operatorname{gr} E$ to $E$.  Selecting lifts for each of the finitely many homogeneous generators of $\operatorname{gr} E$ in a produces a map 
			\begin{equation} \label{eqn:8}
				\bigoplus \mathcal O_\sigma(\lambda_i) \to \mathcal E
			\end{equation}
			that restricts to an isomorphism over $Z$.  Its kernel and cokernel are supported on a closed substack $W \subset \sigma$.  But by Proposition~\ref{gms_univclosed}, the image $\pi(W)$ of $W$ in $S$ is a closed subset of $S$.  We know that $s \not\in \pi(W)$ because $W \subset \sigma$ is closed and does not contain the unique closed point of $\pi^{-1} s$.  Therefore the complement of $\pi(W)$ in $S$ is an open neighborhood of $s$ over which~\eqref{eqn:8} is an isomorphism.

			In the argument above, the $\lambda_i$ were determined by the grading of $ \operatorname{gr} E$ by $Q^{\rm gp} / P^{\rm gp}$.  Therefore they are uniquely determined modulo $\bar M_S^{\rm gp}$.
		\end{proof}

		\begin{proposition} \label{prop:5}
			Let $S$ be a fine, saturated logarithmic scheme.  Let $T$ be a tropical torus over $S$ with character group $X$.  Suppose that $\mathcal E$ is a locally free sheaf on $T$.  Then there is a unique splitting $\mathcal E \simeq \bigoplus_{\lambda} E_\lambda(\lambda)$, indexed by characters%
			\footnote{If $T$ is not split, the collection of characters $\lambda$ and their multiplicities must be regarded as a section of $\Sym^r X$, where $X$ is the character group of $T$ and $r$ is the rank of $\mathcal E$.}
			$\lambda$ of $T$, where each $E_\lambda$ is a locally free sheaf on~$S$.
		\end{proposition}

		\begin{proof}
			Let $\mathcal E$ be a locally free sheaf on $T$.  Let $X$ be the character group of $T$.  The assertion is local in the strict \'etale topology of $S$, so we may assume that $X \simeq \mathbf Z^r$, and therefore that $T \simeq \logGm^r$.

			First we demonstrate the uniqueness.  If $\mathcal E = \bigoplus E_\lambda(\lambda)$, with the sum taken over $\lambda \in X$, then Corollary~\ref{cor:4} implies that we can recover $E_\lambda$ as $\pi_\ast \mathcal E(-\lambda)$ where $\pi : T \to S$ is the projection.

			Now we prove the existence.  For each $\lambda \in X$, let $E_\lambda = \pi_\ast \mathcal E(-\lambda)$, where $\pi : T \to S$ is the projection.  By adjunction, we have a map
			\begin{equation} \label{eqn:12}
				\phi\colon\bigoplus E_\lambda(\lambda) \to \mathcal E .
			\end{equation}
			We will argue that each $E_\lambda$ is locally free on $S$ and that $\phi$ is an isomorphism.  

			The assertion that $\phi$ is an isomorphism means that, for every fine, saturated logarithmic scheme $S'$ over $T$, the induced map $\phi \big|_{S'}$ is an isomorphism.  If $S'$ is a fine, saturated logarithmic scheme then any morphism $S' \to T$ factors through some universal surjection $\sigma \to T$, where $\sigma$ is an Olsson precone over $S$.  Therefore it will be sufficient to show that $\phi \big|_\sigma$ is an isomorphism whenever $\sigma$ is an exact Olsson precone over $S$ and $\sigma \to T$ is a universal surjection.  With these assumptions, let  $\rho : \sigma \to S$ be the projection.

			Let us represent $\sigma$ by an extension $Q$ of $\bar M_S^{\rm gp}$.  Proposition~\ref{prop:vect-on-cone} implies that, at least after replacing $S$ with a strict \'etale cover, $\mathcal E \big|_\sigma$ splits as $\bigoplus F_\mu(\mu)$, where the $F_\mu$ are locally free sheaves on $S$ and the $\mu$ are sections of $Q^{\rm gp}$.  The $\mu$ are well defined in $Q^{\rm gp} / \bar M_S^{\rm gp}$.  We argue that they descend to $X$.

			The map $\sigma \to T$ is given by a morphism $X \to Q^{\rm gp}$.  The universal surjectivity corresponds to the requirement that 
			\begin{equation*}
				\bar M_S^{\rm gp} \oplus X \to Q^{\rm gp}
			\end{equation*}
			be exact, in the sense that the pre\"image of $Q$ is $\bar M_S \oplus 0$.  In particular, $X \to Q^{\rm gp} / \bar M_S^{\rm gp}$ must be injective.  We can represent the fiber product $\sigma \mathop\times_ T\sigma$ by a sheaf of monoids $R$, obtained as the image of $Q \oplus Q$ in the quotient of $Q^{\rm gp} \oplus Q^{\rm gp}$ by the antidiagonal copy of $\bar M_S^{\rm gp} \oplus X$.  One may verify directly that $R$ is sharp, and therefore that 
			\begin{equation*}
				R^{\rm gp} = Q^{\rm gp} \oplus Q^{\rm gp} / \bar M_S^{\rm gp} \oplus X .
			\end{equation*}
			In particular, $R^{\rm gp} / Q^{\rm gp} \simeq Q^{\rm gp} / \bar M_S^{\rm gp} \oplus X$ so there is an exact sequence~:
			\begin{equation} \label{eqn:11}
				0 \to X \to Q^{\rm gp} / \bar M_S^{\rm gp} \xrightarrow{p_1^\ast - p_2^\ast} R^{\rm gp} / Q^{\rm gp} \to 0
			\end{equation}

			We return to the splitting $\mathcal E \big|_\sigma \simeq \bigoplus F_\mu(\mu)$.  Let $\tau = \sigma\mathop\times_T\sigma$.  Then $\mathcal E \big|_\tau$ gets two splittings from the two projections~:
			\begin{equation*}
				\mathcal E \big|_\tau \simeq \bigoplus F_\mu(p_1^\ast \mu) \simeq \bigoplus F_\mu(p_2^\ast \mu)
			\end{equation*}
			However, Proposition~\ref{prop:vect-on-cone} implies that, modulo constants, the $p_1^\ast \mu$ and $p_2^\ast \mu$ are permutations of one another, and that the permutation relating them restricts to the identity on the diagonal $\sigma \to \tau$.  Therefore the permutation must be the identity and $p_1^\ast \mu$ agrees with $p_2^\ast \mu$ modulo constants.  But now the exactness of~\eqref{eqn:11} implies that, after adjustment by a constant, $\mu$ lies in $X$.

			Now we can write $\mathcal E \big|_\sigma \simeq \bigoplus F_\mu(\mu)$, with the sum taken over characters $\mu$ of $T$.  Proposition~\ref{prop:vect-on-cone} implies that
			\begin{equation*}
				\rho_\ast \mathcal E (-\lambda) \big|_\sigma = \bigoplus \rho_\ast F_\mu(\mu-\lambda) = F_\lambda
			\end{equation*}
			since $\mu - \lambda$ is unbounded below on $\sigma$ (because $\sigma \to T$ is universally surjective), hence has vanishing global sections by Corollary~\ref{cor:pushforward}.  It follows as well that
			\begin{equation} \label{eqn:13}
				\bigoplus_\lambda\rho^\ast F_\lambda(\lambda) \to \mathcal E \big|_\sigma
			\end{equation}
			is an isomorphism.

			We may identify
			\begin{equation*}
				\pi_\ast \mathcal E(-\lambda) = \varprojlim_{\sigma \to T} \rho_\ast \mathcal E(-\lambda) \big|_\sigma
			\end{equation*}
			with the limit taken over all universal surjections $\sigma \to T$, 
            If $\tau \to \sigma$ is a universal surjection of exact Olsson precones over $S$, and $\psi : \tau \to S$ is the projection, then $\rho_\ast \mathcal E(-\lambda) \big|_\sigma \to \psi_\ast \mathcal E(-\lambda) \big|_\tau$ is an isomorphism by the projection formula.  Since the fiber product of any two universal surjections to $T$ is also a universal surjection to $T$, this system must be constant.%
			\footnote{Let $C$ be a nonempty category with products and let $F$ be a presheaf on $C$ such that $F(\sigma) \to F(\tau)$ is an isomorphism for all $\tau \to \sigma$ in $C$.  Then $F$ is constant~: choose an object $\sigma \in C$.  For every $\tau \in C$, we obtain an isomorphism~:
			\begin{equation*}
				F(\sigma) \xrightarrow\sim F(\sigma \times \tau) \xleftarrow\sim F(\tau)
			\end{equation*}
			This isomorphism is compatible with the restriction maps because if $\tau \to \omega$ is a morphism in $C$ then we can apply $F$ to the commutative diagram below~:
			\begin{equation*} \begin{tikzcd}[ampersand replacement=\&]
				\tau \ar[dd] \& \tau \times \sigma \ar[l] \ar[dr] \ar[dd] \\
				\&\& \sigma \\
				\omega \ar[r] \& \omega \times \sigma \ar[ur] 
			\end{tikzcd}
			\end{equation*}
			}
			We conclude that $\pi_\ast \mathcal E(-\lambda)$ is the constant value of the $\rho_\ast \mathcal E(-\lambda) \big|_\sigma$, and in particular is locally free.  Furthermore, the map~\eqref{eqn:12} corresponds to the system of isomorphisms~\eqref{eqn:13}, and therefore is an isomorphism as well.
		\end{proof}

		\begin{corollary} \label{cor:3}
			Suppose that $S$ is a fine, saturated logarithmic scheme and $\sigma$ is a unviersally surjective weakly convex Olsson precone over $S$.   Let $\mathcal E$ be a locally free sheaf on $\sigma$.  Then, locally in the strict \'etale topology $S$, there is a splitting $\mathcal E \simeq \bigoplus \mathcal O_\sigma(\lambda_i)$ for linear functions $\lambda_0, \ldots, \lambda_r$ on $\sigma$.  The $\lambda_i$ and their multiplicities are unique up to permutation of their indices and addition of constants.
		\end{corollary}
		\begin{proof}
			Let $T$ be the lineality space of $\sigma$ and let $\sigma' = \sigma / T$.  Then Proposition~\ref{prop:polyhedral domain} implies that $\sigma'$ is an exact Olsson precone over $S$ and $\sigma$ is a $T$-torsor over $\sigma'$.  Proposition~\ref{prop:polyhedral domain} also implies that this torsor is locally trivial in the strict \'etale topology of $S$.  After replacing $S$ by a strict \'etale cover, we may therefore identify $\sigma$ with $\sigma' \times T$ by the choice of a map $q : \sigma \to T$.  By Proposition~\ref{prop:5}, a locally free sheaf $\mathcal E$ on $\sigma' \times T$ splits as $\bigoplus E_\lambda(\lambda)$, with the sum taken over characters $\lambda$ of $T$ and with the $E_\lambda$ being locally free sheaves on $\sigma'$. By Proposition~\ref{prop:vect-on-cone}, each of the $E_\lambda$ splits, locally in the strict \'etale topology of $S$, as a direct sum of $F_i(\mu_i)$, where each $\mu_i$ is a linear function on $\sigma'$ and each $F_i$ is locally free on $S$.   After re\"indexing, we obtain a splitting of $\mathcal E$ into summands $F_i(\mu_i + \lambda_i)$~:
			\begin{equation*}
				\mathcal E \simeq \bigoplus F_i(\mu_i + q^\ast \lambda_i)
			\end{equation*}

			The ambiguities in these splittings are permutation of the indices and addition of constants to the $\mu_i$, by Proposition~\ref{prop:vect-on-cone}, as well as the choice of the splitting $q : \sigma \to T$.  The ambiguity of the latter is a morphism $a : \sigma' \to T$.  Replacing $q$ with $q + a$ will replace $\mu_i$ with $\mu_i - a^\ast \lambda_i$.  Then the function
			\begin{equation*}
				\mu_i - a^\ast \lambda_i + (q + a)^\ast \lambda_i = \mu_i + q^\ast \lambda_i
			\end{equation*}
			on $\sigma$ remains unchanged.  Thus the only ambiguity in the splitting of $\mathcal E$ is permutation of the indices and addition of constants.
		\end{proof}

		\begin{corollary} \label{cor:6}
			Suppose that $S$ is a fine, saturated logarithmic scheme and $\sigma$ is a universally surjective, weakly convex Olsson precone over $S$.  Let $\mathcal L$ be an invertible sheaf on $\sigma$.  Then, locally in the strict \'etale topology of $S$, there is an invertible sheaf $L$ on $S$ and a section $\lambda$ of $\bar M_\sigma^{\rm gp}$ such that $\mathcal L \simeq \pi^\ast L(\lambda)$.  The section $\lambda$ is unique up to addition of constants.
		\end{corollary}
		\begin{proof}
			This is the rank~$1$ case of Corollary~\ref{cor:3}.
		\end{proof}
            
\begin{corollary}\label{exact_sequence_polyhedra}
     Let $S$ be a fine, saturated logarithmic scheme.  Let $\pi : \sigma \to S$ be a universally surjective, weakly convex Olsson precone.  Then
		 \begin{equation*}
			 M_S^{\rm gp} \to \pi_\ast M_\sigma^{\rm gp}
		 \end{equation*}
		 is an isomorphism and there is an exact sequence 
		 \begin{equation*}
				0 \to \bar M_S^{\rm gp} \to \pi_\ast \bar M_\sigma^{\rm gp} \to \mathrm R^1 \pi_\ast \mathcal O_S^\ast \to 0  .
		 \end{equation*}
\end{corollary}

\begin{proof}
	We apply $\pi_\ast$ to the following exact sequence~:
    \[0\to\OO_{\sigma}^*\to M_\sigma^{\rm gp}\to \bar M_\sigma^{\rm gp}\to0.\]
	We obtain a commutative diagram with exact rows~:
	\begin{equation} \label{eqn:6} \begin{tikzcd}
		0 \ar[r] & \OO_S^\ast \ar[r] \ar[d] & M_S^{\rm gp} \ar[r] \ar[d]& \bar M_S^{\rm gp} \ar[r]  \ar[d]& 0 \ar[d] \\
		0 \ar[r] & \pi_\ast \OO_\sigma^\ast \ar[r] & \pi_\ast M_\sigma^{\rm gp} \ar[r] & \pi_\ast \bar M_\sigma^{\rm gp} \ar[r] & \mathrm R^1 \pi_\ast \OO_\sigma^\ast \ar[r] & \mathrm R^1 \pi_\ast M_\sigma^{\rm gp} 
	\end{tikzcd}
	\end{equation}
	The map $\OO_S^\ast \to \pi_\ast \OO_\sigma^\ast$ is an isomorphism by Corollary~\ref{cor:structuresheaf}.  The fiber of $\pi_\ast M_\sigma^{\rm gp} \to \pi_\ast \bar M_\sigma^{\rm gp}$ over $\gamma$ is $\pi_\ast \mathcal O_\sigma^\ast(-\gamma)$~: in other words, the sheaf of trivializations of $\mathcal O_\sigma^\ast(-\gamma)$ on $S$.  By Corollary~\ref{cor:6}, this sheaf is empty unless $\gamma$ is constant, so $M_S^{\rm gp} \to \pi_\ast M_\sigma^{\rm gp}$ is surjective.  Exactness of $M_\sigma$ over $M_S$ implies that $\bar M_S^{\rm gp} \to \pi_\ast \bar M_\sigma^{\rm gp}$ is injective, which implies that $M_S^{\rm gp} \to \pi_\ast M_\sigma^{\rm gp}$ is also injective.  The bottom row of~\eqref{eqn:6} can therefore be rewritten~:
	\begin{equation*}
		0 \to \bar M_S^{\rm gp} \to \pi_\ast \bar M_\sigma^{\rm gp} \to \mathrm R^1 \pi_\ast \mathcal O_\sigma^\ast \to \mathrm R^1 \pi_\ast M_\sigma^{\rm gp}
	\end{equation*}
	The map $\pi_\ast \bar M_\sigma^{\rm gp} \to \mathrm R^1 \pi_\ast \mathcal O_\sigma^\ast$ is surjective by \Cref{cor:6}.
    \end{proof}

\section{Moduli of vector bundles on Olsson fans} \label{sec:moduli}
\tocdesc{conditions guaranteeing algebraicity of the moduli stack of vector bundles on an Olsson fan and examples where algebraicity fails~; connections with Bruhat--Tits buildings}

\subsection{Moduli of vector bundles on integral cones and subdivisions}

\begin{proposition} \label{prop:8}
	Let $S$ be a fine, saturated logarithmic scheme and let $\sigma \to S$ be a universally surjective, integral, weakly convex Olsson cone over $S$. Then the locally free sheaves on $\sigma$ are parameterized by an algebraic stack over $S$.  It is a disjoint union of stacks $\mathrm BG_{\boldsymbol{\lambda}}$ indexed by multisubsets $\boldsymbol{\lambda}$ of the character group of the tropical torus spanned by $\sigma$.  The group $G_{\boldsymbol{\lambda}}$ is, locally in the strict \'etale topology of $S$, isomorphic to the group of invertible matrices in
	\begin{equation*}
		\bigoplus_{i,j} \mathcal O_S(\inf_\sigma(\lambda_j - \lambda_i)) ,
	\end{equation*}
	where the sum is taken over pairs of indices in $\boldsymbol\lambda$.
\end{proposition}
\begin{proof}
	At least after replacing $S$ by a strict \'etale cover, Corollary~\ref{cor:3} implies that every locally free sheaf $\mathcal E$ on $\sigma$ splits as $\bigoplus E_i(\lambda_i)$, with the $\lambda_i$ and their multiplicities uniquely determined in the character group of the torus spanned by $\sigma$.  Therefore the stack $\uHom(\sigma, \BGL_n)$ on the strict \'etale site of $S$ splits as a disjoint union of components indexed by the choices of the $\lambda_i$ and their multiplicities.  For any fixed choice $\boldsymbol{\lambda} = (\lambda_1, \ldots, \lambda_n)$, let $\uHom(\sigma, \BGL_n)_{\boldsymbol{\lambda}}$ be the corresponding component of $\uHom(\sigma, \BGL_n)$.  Any object of $\uHom(\sigma, \BGL_n)_{\boldsymbol{\lambda}}$ is locally isomorphic to $\bigoplus E_i(\lambda_i)$.  Thus $\uHom(\sigma, \BGL_n)_{\boldsymbol{\lambda}} = \mathrm BG_{\boldsymbol{\lambda}}$ where $G_{\boldsymbol{\lambda}}$ is the automorphism group of $\bigoplus E_i(\lambda_i)$, namely the group of invertible matrices in
	\begin{equation*}
		\pi_\ast \bigoplus E_j \otimes E_i^\vee(\lambda_j - \lambda_i) = \bigoplus E_j \otimes E_i^\vee ( \inf_\sigma (\lambda_j - \lambda_i) ) .
	\end{equation*}
	The equality comes from Corollary~\ref{cor:pushforward}.  This group is algebraic, so $\uHom(\sigma, \BGL_n) = \mathrm B G_{\boldsymbol{\lambda}}$ is as well.  After further localization in $S$, we can assume that the $E_i$ are all trivial.  This completes the proof.
\end{proof}
\begin{remark} \label{rem:1}
	One can also demonstrate \Cref{prop:8} using \cite[Theorem 5.7]{LunaEtaleSlice} (over a field) or \cite[Theorem 10.14]{EtaleLocalStructure} (in mixed characteristic).  Note that the flatness hypotheses in the above cited results are implied here by the hypothesis that $\sigma$ be integral over $S$ here.

	While it is true that if $\pi\colon \sigma\to S$ is an exact Olsson cone, then $\pi$ is a good moduli space by \Cref{gms_univclosed}, this hypothesis is not sufficient for the conclusion of \Cref{prop:8}~: see Example~\ref{ex:1}, below.
\end{remark}

\begin{example} \label{ex:1}
 Let $h\colon P\to Q$ be the monoid homorphism of \Cref{exa:exact_not_integral}, which is exact but not integral. The cone $\sigma_Q$ is cut out by the inequalities $0\leq x,y\leq z$.  It can also be described as the cone over the square  with vertices $(0,0,1),(1,0,1),(0,1,1),(1,1,1)$. The map $\phi\colon\sigma_Q\to\sigma_P=\mathbf R_{\geq0}^2$ is the projection to the first two coordinates, which is surjective but maps the ray spanned by $(1,1,1)$ to the diagonal of $\mathbf R_{\geq0}^2$.


 Let $S=\Aaff^2$ with its standard logarithmic structure, and $\pi\colon \sigma=S\times_{\Acal^2}\Acal_Q\to S$. By \Cref{prop:vect-on-cone}, every vector bundle on $\sigma$ splits (moreover, vector bundles on $S=\Aaff^2$ are trivial). Notice that both $x$ and $y$ are dominated by $z$, but $P_{\leq z}$ is not filtered, so $\pi_*\OO_\sigma(z)=\Ical_0$ is the ideal sheaf of the origin in $\Aaff^2$, which is a rank-one torsion-free sheaf, but not a line bundle.

	Consider the multicharacter $\boldsymbol\lambda = (0,z)$, and $\Ecal=\OO_\sigma\oplus\OO_\sigma(z)$. Let $\mathfrak V_{\boldsymbol\lambda}$ be the stack of locally free sheaves on $\sigma$ with multicharacter $\boldsymbol\lambda$.  Then $\mathfrak V_{(0,z)}\simeq \mathrm B G$, where $G$ is the automorphism group of $\Ecal$. For every $S$-scheme $T$, we have:
 \[ G(T)=\begin{pmatrix}
\mathcal O^*_T & \mathcal I_0 \otimes \mathcal O_T \\
0&\mathcal O^*_T
\end{pmatrix}\subseteq \operatorname{GL}_2(T),\]
which is not representable (a coherent sheaf is representable if and only if it is locally free~: see \cite[Corollary~2]{Nitsure}). Hence $\mathfrak V_{\boldsymbol\lambda}$ is not an algebraic stack.
\end{example}
 We can extend the representability result to weak cone complexes with no interesting topology, or at least constant fiberwise topology.
\begin{definition}
    Let $S$ be an atomic, fine, saturated logarithmic scheme and let $\Sigma$ be a weakly convex cone complex over $S$. We write $J(\Sigma)$ for the partially ordered set of faces of $\sigma$ and call it the \emph{order complex} of $\Sigma$.
    The \emph{exact intersection complex} of $\Sigma$ is the partially ordered subset $I(\Sigma)=I(\Sigma/S) \subset J(\Sigma)$ consisting only of those cones that are integral and universally surjective over $S$.
    
\end{definition}

We can also think of $I(\Sigma)$ and $J(\Sigma)$ as categories, and will do so in the future without comment. 

Next, we recall a few definitions from category theory~:
\begin{definition}
    Let $\phi\colon I\to J$ be a functor. For an object $j$ of $J$, the \emph{comma category} $(j\downarrow \phi)$ is the category whose objects are pairs $(i,u)$ where $i$ is an object of $I$, and $u\colon j\to\phi(i)$ is an arrow of $J$; arrows from $(i,u)$ to $(i',u')$ are arrows $v\colon i\to i'$ of $I$ such that $\phi(v)\circ u=u'$.  We say $\phi$ is \emph{cofinal} if every comma category is non-empty and connected.

    The \emph{nerve} of a category $C$ is the simplicial set  whose $n$-simplices are the functors $[n] \to C$, with $[n]$ being the category associated with the the totally ordered finite set $\{ 0 < \cdots < n \}$.  A functor between categories is called a \emph{homotopy equivalence} if it induces a homotopy equivalence after taking the nerve.  A category is \emph{contractible} if it is homotopy equivalent to a punctual category.
\end{definition}

\begin{remark} \label{rem:3}
	\begin{enumerate}
		\item The significance of cofinality is that if $\phi : I \to J$ is cofinal and $F : J^{\rm op} \to C$ is a functor, then $\varprojlim_{j \in J} F(j) = \varprojlim_{i \in I} F(\phi(i))$ \cite[Tag~002R]{stacks-project}.
\item Any category with an initial or a final object is contractible.  
\item A contractible category is certainly nonempty and connected, so a functor whose comma categories $j \downarrow \phi$ are contractible is cofinal.
	\end{enumerate}
\end{remark}

\begin{proposition}\label{thm:5}
	Let $S$ be a fine, saturated logarithmic scheme and let $\Sigma$ be 
    a weak cone complex over $S$. Assume that $I(\Sigma)$ is cofinal in $J(\Sigma)$ on an atomic cover of $S$ in the strict \'etale topology. Then the locally free sheaves on $\Sigma$ are parameterized by an algebraic stack over $S$.
\end{proposition} 
\begin{proof}
	The assertion is local in the strict \'etale topology of $S$.  We therefore assume that $S$ is atomic.  Then 
	\begin{equation*}
		\uHom(\Sigma, \BGL_n) = \varprojlim_{\sigma \in J(\Sigma)} \uHom(\sigma, \BGL_n)
	\end{equation*}
	with the limit taken over the \emph{full} order complex $J(\Sigma)$ of $\Sigma$.  If the \emph{exact} intersection complex $I(\Sigma)$ is cofinal in $J(\Sigma)$, it is equivalent to take the limit over the former \cite[Section IX.3]{MacLane}.  A finite limit of algebraic stacks is algebraic, so it is sufficient to observe that $\uHom(\sigma, \BGL_n)$ is an algebraic stack over $S$ when $\sigma$ is an integral, universally surjective, weakly convex Olsson cone over $S$, by Proposition~\ref{prop:8}.
\end{proof}
Next we  show that an integral subdivision of an integral, universally surjective, weakly convex Olsson precone satisfies the cofinality assumption.

\begin{remark}
    Not every subdivision of an integral cone is integral~:

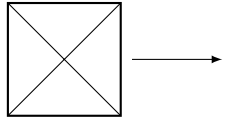
		\begingroup\par\bigskip\centering
    \begin{tikzpicture}[>=latex,scale=1.5]
  \draw[thick] (0,0) rectangle (1 cm,1 cm);
  \draw[thin] (0,0) -- (1 cm,1 cm) (0,1 cm) -- (1 cm,0);
  \draw[->] (1.1 cm,.5 cm) -- (1.9 cm,.5 cm);
  \draw[thick] (2,0 cm) -- (2 cm,1 cm);
\end{tikzpicture}
\par\bigskip\endgroup

For a non-integral subdivision the following proposition may not hold.

\end{remark}
\begin{proposition} \label{prop:2}
	Let $S$ be an atomic, fine, saturated logarithmic scheme.  Let $\Sigma' \to \Sigma$ be a subdivision of integral, universally surjective, weak cone complexes over $S$.  Then $J(\Sigma')\to J(\Sigma)$ and $I(\Sigma') \to I(\Sigma)$ are homotopy equivalences.
\end{proposition}
\begin{proof}
	We explain $I$, leaving the easier case of $J$ to the reader.
    
    By Quillen's Theorem A~\cite{Quillen}, it is sufficient to show that the comma category $I(\Sigma')\downarrow\tau$ is contractible for every object $\tau\in I(\Sigma)$. 
		Therefore we reduce to the case where $\Sigma = \sigma$ is a single integral, universally surjective, weakly convex Olsson cone and we need only show that $I(\Sigma')$ is contractible.  

	We reduce to a subdivision of usual polyhedra. Let $s$ be a geometric point of the closed stratum of $S$.  Let $\bar M_{S,s} \to \mathbf R_{\geq 0}$ be a local valuation of $\bar M_S$ and define $M_s$ to be the induced logarithmic structure on $s$ (whose characteristic monoid has value group $\mathbf R$).  Let $\sigma(s)$ be the set of $s$-points of $\sigma$, 
		with its usual topology.  Then $\Sigma'(s)$ is homeomorphic to $\sigma(s)$.  We argue furthermore that the integral intersection complex $I(\Sigma')$ is the same as the intersection complex of the $\tau(s) \subset \sigma(s)$ for $\tau$ among the polyhedra of $\Sigma'$.  For this, it will be sufficient to show, for every polyhedron $\tau$ in $J(\Sigma')$, that $\tau(s)$ is contractible if and only if $\tau$ is universally surjective over $S$.  In fact, $\tau(s)$ is convex, so it is contractible if and only if it is nonempty, so we only need to show that $\tau(s)$ is nonempty if and only if $\tau$ is universally surjective over $S$.  If $\tau(s)$ is nonempty then, choosing $t \in \tau(s)$, the dual homomorphism of monoids $\bar M_{\sigma,s} \to \bar M_{\tau,t}$ must be local~; since it is also integral, it is exact by \Cref{rmk:integral-exact}, so $\tau \to \sigma$ is universally surjective.  Conversely, if $\tau \to \sigma$ is universally surjective then the base change $s \mathop\times_S \tau \to s$ is surjective, so $s \mathop\times_S \tau$ has an $s$-valued point by the Nullstellensatz and its monoidal analogue.%
		\footnote{If $P$ is a finitely generated monoid then it has a local homomorphism to $\mathbf R_{\geq 0}$.}

    We conclude that $I(\Sigma')$ is equivalent to the intersection complex of a polyhedral subdivision of the convex region $\sigma(s)$, hence is contractible.  \end{proof}

\begin{proposition} \label{prop:4}
	Let $S$ be an atomic, fine, saturated logarithmic scheme.  Let $\sigma$ be an integral, universally surjective, weakly convex Olsson cone over $S$, and let $\Sigma'$ be an integral subdivision of $\sigma$.  Then the inclusion $I(\Sigma') \to J(\Sigma')$ is cofinal.
\end{proposition}
\begin{proof}
Let us write $(\tau\downarrow I(\Sigma'))$ for the comma category of the inclusion $I(\Sigma') \to J(\Sigma')$ where  $\tau$ is a cone in $\Sigma'$. To prove the statement it suffices to show that  $(\tau\downarrow I(\Sigma'))$ is non-empty and connected, but in fact we will show that it is contractible.

	By definition of comma category, the objects of $(\tau\downarrow I(\Sigma'))$  are the cones in $\Sigma'$ which contain $\tau$ and are integral and exact over $S$.  By Lemma~\ref{lem:exact_stars} this is identified with $I(\Star_\tau(\Sigma)/\Star_\tau(S))$.  By \Cref{lem:1}, $\Star_\tau(\Sigma')$ is a subdivision of $\Star_\tau(\sigma)$, so by Proposition~\ref{prop:2}, $I(\Star_\tau(\Sigma'))$ is homotopy equivalent to $I(\Star_\tau(\sigma))$.  But, $\sigma$ is a final object of $I(\Star_\tau(\sigma))$, so $I(\Star_\tau(\sigma))$ is contractible.
\end{proof}

We conclude with a corollary to Proposition~\ref{prop:8}, which we feel may offer a clue to the appropriate extension of Theorem~\ref{thm:5} to more general tropical spaces.  To make the statement, we require a definition.

\begin{definition} \label{def:3}
	Let $S$ be a fine, saturated logarithmic scheme and let $\pi : \sigma \to S$ be a universally surjective, integral, weakly convex Olsson cone over $S$.  Let $T$ be the tropical torus spanned by $\sigma$.  Let $\boldsymbol\lambda = (\lambda_1, \ldots, \lambda_n)$ be a tuple of characters of $T$.  For each $i$ and $j$, let $\mu_{i,j} \in \bar M_S^{\rm gp} = \inf_\sigma(\lambda_i - \lambda_j)$ be the minimum value taken by $\lambda_i - \lambda_j$ on $\sigma$.  The \emph{Weyl convex hull} of $\sigma$ is the weakly convex Olsson cone in $T$ defined by the inequalities
	\begin{equation*}
		\lambda_i - \lambda_j \geq \mu_{i,j}
	\end{equation*}
	for every $i$ and $j$.
\end{definition}

\begin{corollary} \label{cor:13}
	Let $S$ be a fine, saturated logarithmic scheme,  $\pi : \sigma \to S$ a universally surjective, integral, weakly convex Olsson cone over $S$, and  $\mathcal E$  a locally free sheaf on $\sigma$.  Then $\mathcal E$ extends uniquely, up to unique isomorphism, to the Weyl convex hull of $\sigma$.
\end{corollary}
\begin{proof}
	Since the assertion is local in the strict \'etale topology of $S$, we may assume by Proposition~\ref{prop:vect-on-cone} that $\mathcal E = \bigoplus E_i(\lambda_i)$ for some locally free sheaves $E_i$ on $S$ and characters $\lambda_i$ of the torus $T$ spanned by $\sigma$.  But then the formula $\bigoplus E_i(\lambda_i)$ gives an extension of $\mathcal E$ to a locally free sheaf $\mathcal F$ on $T$.  Then $\mathcal F \big|_\tau$ extends $\mathcal E$.  Proposition~\ref{prop:vect-on-cone} implies that $\mathcal F \big|_\tau$ is the unique extension of $\mathcal E$ to $\tau$, up to isomorphism.  But Proposition~\ref{prop:8} shows that the automorphism group of $\mathcal F \big|_\tau$ is the same as the automorphism group of $\mathcal E$, so this extension is unique up to unique isomorphism.
\end{proof}

\subsection{Cautionary tales about the general case}
 \subsubsection{}
				Not all line bundles on Olsson \emph{fans} arise from conewise linear functions. For instance, let $\Sigma$ be a loop with a single vertex (the $\Gm$ quotient of an irreducible nodal cubic). Normalising shows that a line bundle is determined by an integer (the slope or degree on the normalisation) together with an invertible scalar (gluing datum). The only line bundle corresponding  to a conewise linear function is the trivial one.

\subsubsection{}            Vector bundles on Olsson \emph{fans} do not split in general. It is enough to consider equivariant bundles on non-affine toric varieties~:   the tangent bundle of $\mathbf P^2$ does not split, for example.  However, the Euler sequence provides a resolution by direct sums of conewise-linear functions.

            Another example is the following rank-two vector bundle on the tripod, which is an equivariant version of (a $1$-parameter smoothing of) a rational curve obtained by attaching three $\PP^1$ to the points $\{0,1,\infty\}$ of a fourth $\PP^1$, see \cite[Counterexample 4.5]{martensvariations}. We describe the bundle in terms of its Perling data. Let $R$ denote a discrete valuation ring with uniformiser $\pi$ and field of fractions $K$. Consider the four lattices up to homothety
            \begin{align*}
                L_0=&\langle e_0,e_1\rangle,\\
                L_1=&\langle e_0,\pi e_1\rangle,\\
                L_2=&\langle e_0,\pi^{-1}e_1\rangle,\\
                L_3=&\langle \pi e_0,e_0+e_1\rangle.
            \end{align*}
						inside $K^2$. These lattices, and the inclusions among them, define a rank two vector bundle on a tripod $\Sigma$ over $R$, with central vertex $v_0$ and leaves $v_1,v_2,v_3$. There does not exist a basis of $K^2$ diagonalising all the inclusions,\footnote{Although there exist diagonal bases for every pair of edges, for instance $\{e_0,e_1\}$ for $\{0,1,2\}$, $\{e_0,e_0+e_1\}$ for $\{0,1,3\}$, and $\{e_0+e_1,e_1\}$ for $\{0,2,3\}$.} meaning that the vector bundle on the tripod is not split. Equivariant vector bundles on a toric scheme $X_{\Sigma}$ (over $R$) are classified by maps from the fan $\Sigma$ to the Bruhat--Tits buildings of $\mathrm{PGL}$ \cite{kaveh2022toric}~; vector bundles that split correspond to maps that factor through a single apartment. This example is modeled on the Bruhat--Tits building of $\mathrm{PGL}_2(K)$ so as to be the smallest example not contained in a single apartment.

\subsubsection{} The next example shows that even the stack of line bundles on an Olsson fan may fail to be algebraic.
 
Let $S$ be the spectrum of a complete discrete valuation ring, equipped with its divisorial logarithmic structure.  Write $\delta$ for the generator of $\bar M_S$.  Let $\sigma$ be the Olsson fan dual to the monoid generated by $\alpha$ and $\beta$ with the relation $\alpha + \beta = \delta$.  Thus $\sigma$ is a family of intervals of length~$\delta$ over $S$.  Let $\Sigma$ be the result of gluing the loci $\{ \alpha = 0 \}$ and $\{ \beta = 0 \}$.  Then $\Sigma$ is a family of loops of length~$\delta$ over $S$.

For each $n$, let $S_n$ be the $n$-th order neighborhood of the closed point, $S_0$.  Let $\Sigma_n$ and $\sigma_n$ be the preimages of $S_n$ in $\Sigma$ and in $\sigma$, respectively.  An invertible sheaf on $\sigma_n$ is $\mathcal O_{\sigma_n}(\lambda)$ for some linear function $\lambda$ on $\sigma_n$.  Since $\lambda$ is only well-defined up to constants, we can choose $\lambda$ to have value $0$ at $\{ \alpha = 0 \}$ so that $\lambda = n \alpha$ for some integer $n$.  Therefore we can identify the Picard group of $\sigma_n$ with $\mathbf Z$.  Since the Picard group of $\Sigma_n$ has the additional choice of gluing along $\{ \alpha = 0 \}$ and $\{ \beta = 0 \}$, we get $\Pic(\Sigma_n) \simeq \mathbf Z \times \Gamma(S_n, \mathcal O_{S_n}^\ast)$.

We can also compute the Picard group of $\Sigma$.  Again, we have $\Pic(\sigma) = \mathbf Z$.  The gluing along the loci $\{ \alpha = 0 \}$ and $\{ \beta = 0 \}$ is a section $g$ of $\mathcal O_S^\ast$, but this section must restrict to $1$ at the generic point.  Indeed, $\Sigma$ is obtained from $\sigma$ by an \'etale equivalence relation that restricts to be trivial over the open point.  But if $K$ is the field of fractions of $\mathcal O_S$ then $\mathcal O_S^\ast$ embeds in $K^\ast$, and therefore $g = 1$.  Thus $\Pic(\Sigma) = \mathbf Z$.

On the other hand, if $\Pic(\Sigma)$ were representable by an algebraic space (or if it were the sheaf of isomorphism classes of an algebraic stack), $\Gamma(S, \Pic(\Sigma))$ would coincide with $\varprojlim \Gamma(S_n, \Pic(\Sigma_n))$.  But we have just seen that
\begin{equation*}
	\varprojlim \Gamma(S_n, \Pic(\Sigma_n)) = \mathbf Z \times \varprojlim \Gamma(S_n, \mathcal O_{S_n}^\ast) = \mathbf Z \times \Gamma(S, \mathcal O_S^\ast) \neq \mathbf Z.
\end{equation*}

\bibliographystyle{alpha}
\bibliography{biblio} 

\stoptoc
\subsection*{Funding}
L.B. is a member of INdAM group GNSAGA.
\resumetoc

\end{document}